\documentclass[11pt,english,a4paper, leqno]{article}
\usepackage{amssymb,amsfonts,amsmath,amsthm,amscd,mathrsfs}
\usepackage{MnSymbol}
\usepackage[T1]{fontenc}
\usepackage[latin9]{inputenc}
\usepackage{geometry}
\usepackage{xcolor}
\usepackage{pdfcolmk}
\usepackage{textcomp}

\usepackage{PSFigure,psfrag}

\usepackage{graphicx}

\usepackage{esint}
\PassOptionsToPackage{normalem}{ulem}
\usepackage{ulem}
\def\f{\footnotesize}
\def\IE{{\mathbb E}}
\def\IZ{{\mathbb Z}}
\def\IR{{\mathbb R}}
\def\IP{{\mathbb P}}

\def\IL{{\mathbb L}}
\def\IZ{{\mathbb Z}}
\def\IL{{\mathbb L}}
\def\IQ{{\mathbb Q}}
\def\cA{{\cal A}}
\def\cB{{\cal B}}
\def\cC{{\cal C}}

\def\cE{{\cal E}}
\def\cF{{\cal F}}
\def\cI{{\cal I}}

\def\cS{{\cal S}}
\def\cU{{\cal U}}
\def\cV{{\cal V}}

\def\cG{{\cal G}}
\def\n{\noindent}
\def\dis{\displaystyle}
\def\ov{\overline}
\def\wt{\widetilde}
\def\wh{\widehat}
\def\dsl{\textstyle\sum\limits}
\def\r{\rightarrow}
\def\ve{\varepsilon}
\def\point{{\mbox{\large $.$}}}
\makeatletter

\def\va{\stackrel{\cV^a}{\longleftrightarrow} \hspace{-2.7ex} \mbox{\f $/$}}

\def\vv{\stackrel{v}{\longleftrightarrow} \hspace{-2.5ex} \mbox{\f $/$}}
\def\ww{\stackrel{w}{\longleftrightarrow} \hspace{-2.5ex} \mbox{\f $/$}}
\def\vvv{\stackrel{v'}{\longleftrightarrow} \hspace{-2.5ex} \mbox{\f $/$}}
\def\uu{\stackrel{u}{\longleftrightarrow} \hspace{-2.5ex} \mbox{\f $/$}}
\def\lr{\longleftrightarrow \hspace{-2.8ex} \mbox{\f $/$}}

\providecolor{lyxadded}{rgb}{0,0,1}
\providecolor{lyxdeleted}{rgb}{1,0,0}

\numberwithin{equation}{section}
\newtheorem{theorem}{Theorem}[section]
\newtheorem{lemma}[theorem]{Lemma}
\newtheorem{corollary}[theorem]{Corollary}
\newtheorem{proposition}[theorem]{Proposition}
\newtheorem{remark}[theorem]{Remark}

\theoremstyle{remark}

\title{\large{\textbf{EXCESS DEVIATIONS FOR  POINTS DISCONNECTED \\
BY RANDOM INTERLACEMENTS}}}

\date{}

\usepackage[english]{babel}

\addtocounter{section}{-1}

\makeatother

\usepackage{babel}

\begin{document}
\maketitle \thispagestyle{empty}
\begin{center} \vspace{-0.7cm}  Alain-Sol Sznitman
\end{center}  
%\begin{center}
%Preliminary Draft
%\end{center}

\medskip
\begin{abstract}
We consider random interlacements on $\IZ^d$, $d \ge 3$, when their vacant set is in a strongly percolative regime. Given a large box centered at the origin, we establish an asymptotic upper bound on the exponential rate of decay of the probability that the box contains an excessive fraction $\nu$ of points that are disconnected by random interlacements from the boundary of a concentric box of double size. As an application, we show that when $\nu$ is not too large this asymptotic upper bound matches the asymptotic lower bound derived in \cite{Szni19d}, and the exponential rate of decay is governed by the variational problem in the continuum involving the percolation function of the vacant set of random interlacements that was studied in \cite{Szni19c}. This is a further confirmation of the pertinence of this variational problem.
\end{abstract}

\vspace{8cm}

%\vfill
\noindent
Departement Mathematik\\   %\hfill May 2021
ETH Z\"urich\\
CH-8092 Z\"urich\\ 
Switzerland \\

\newpage
\thispagestyle{empty} \mbox{}
\newpage  \pagestyle {plain}

\section{Introduction}
In this article we consider random interlacements on $\IZ^d$, $d \ge 3$, when their vacant set is in a strongly percolative regime. In this regime, several kinds of disconnection events of a large deviation nature and their resulting effect on the random interlacements have recently been investigated in \cite{ChiaNitz20a}, \cite{NitzSzni}, \cite{Szni19b}, \cite{Szni17}. Random interlacements are also closely connected to the Gaussian free field, see for instance \cite{DrewPrevRodr}, and similar largely deviant disconnection events have likewise been investigated in the context of the level-set percolation of the Gaussian free field, see \cite{ChiaNitz20b}, \cite{Nitz18}, \cite{Szni19b}, \cite{Szni15}.  In the present work, given a large box centered at the origin of side-length of order $N$, we study the asymptotic exponential rate of decay for the probability that the box contains an excessive fraction $\nu$ of points that are disconnected by random interlacements from the boundary of a concentric box of double size. We establish a general asymptotic upper bound. In particular, we show that when $\nu$ is not too large, this asymptotic upper bound matches in principal order the asymptotic lower bound of \cite{Szni19d}, and confirms the pertinence of the variational problem studied in \cite{Szni19c}. Importantly, in contrast to \cite{Szni19b}, no thickening is involved in the definition of the excess event that we consider, and the resulting effect is markedly different. It remains open whether the assumption on the size of $\nu$ can be removed, and whether in the case of a large enough $\nu$ macroscopic secluded droplets are present and contribute to the excess volume of disconnected points. Such a behaviour would share some flavor with the phase separation and the emergence of a macroscopic Wulff shape for the Bernoulli percolation or for the Ising model, see \cite{Cerf00}, \cite{Bodi99}.
However, it should be pointed out that in the present context, and in the case of the Gaussian free field as well, Dirichlet energy and capacity replaces total variation and perimeter, and the rough order of the exponential decay of the probability of the large deviations is $N^{d-2}$ and not $N^{d-1}$.  \linebreak

%\medskip
We will now describe the results of this article in more details. We denote by $\cI^u$ the random interlacements at level $u \ge 0$ in $\IZ^d$ and by $\cV^u = \IZ^d \backslash \cI^u$ the corresponding vacant set. We are interested in the strongly percolative regime for the vacant set, that is, we assume that
\begin{equation}\label{0.1}
0 < u < \ov{u} ( \le u_*),
\end{equation}
where the precise definition of $\ov{u}$ from (2.3) of \cite{Szni17} is recalled in (\ref{1.26}) below, and $u_*$ denotes the critical level for the percolation of the vacant set of random interlacements. Thanks to the results in  \cite{DrewRathSapo14a}, it is known that $\ov{u}$ is positive, and it is plausible, and presently the object of active research, that $\ov{u} = u_*$. In the context of the closely related model of the level-set percolation of the Gaussian free field, the corresponding equality has been established in the recent work \cite{DumiGoswRodrSeve19}. 

\medskip
We denote by $\theta_0$ the percolation function:
\begin{equation}\label{0.2}
\theta_0(a) = \IP[0  \va \;\;\, \infty], a \ge 0,
\end{equation}
where $\{0 \va \;\; \infty\}$ stands for the event that $0$ does not belong to an infinite component of $\cV^a$, see Figure 1. The function $\theta_0$ is non-decreasing, left-continuous, identically equal to $1$ on $(u_*, \infty)$, with a possible (but not expected) jump at $u_*$, see \cite{Teix09a}. One also knows from \cite{Szni19c} that $\theta_0$ is $C^1$ and has positive derivative on $[0,\wh{u})$, where the definition of $\wh{u}$ (as the supremum of values $v$ in $[0,u_*)$ such that the {\it no large finite cluster property} holds on $[0,v]$) is recalled in (\ref{1.31}) below. One knows from \cite{DrewRathSapo14a} that $\wh{u} > 0$, and the equality $\wh{u} = u_*$ is also plausible but presently open.

\pagebreak
\begin{center}
\psfrag{1}{$1$}
\psfrag{0}{$0$}
\psfrag{a}{$a$}
\psfrag{wt}{$\theta^*$}
\psfrag{t0}{$\theta_0$}
\psfrag{u*}{$u_*$}
\psfrag{uu}{\f $(\sqrt{u} + c_0 (\sqrt{\ov{u}} - \sqrt{u}))^2$}
\includegraphics[width=8cm]{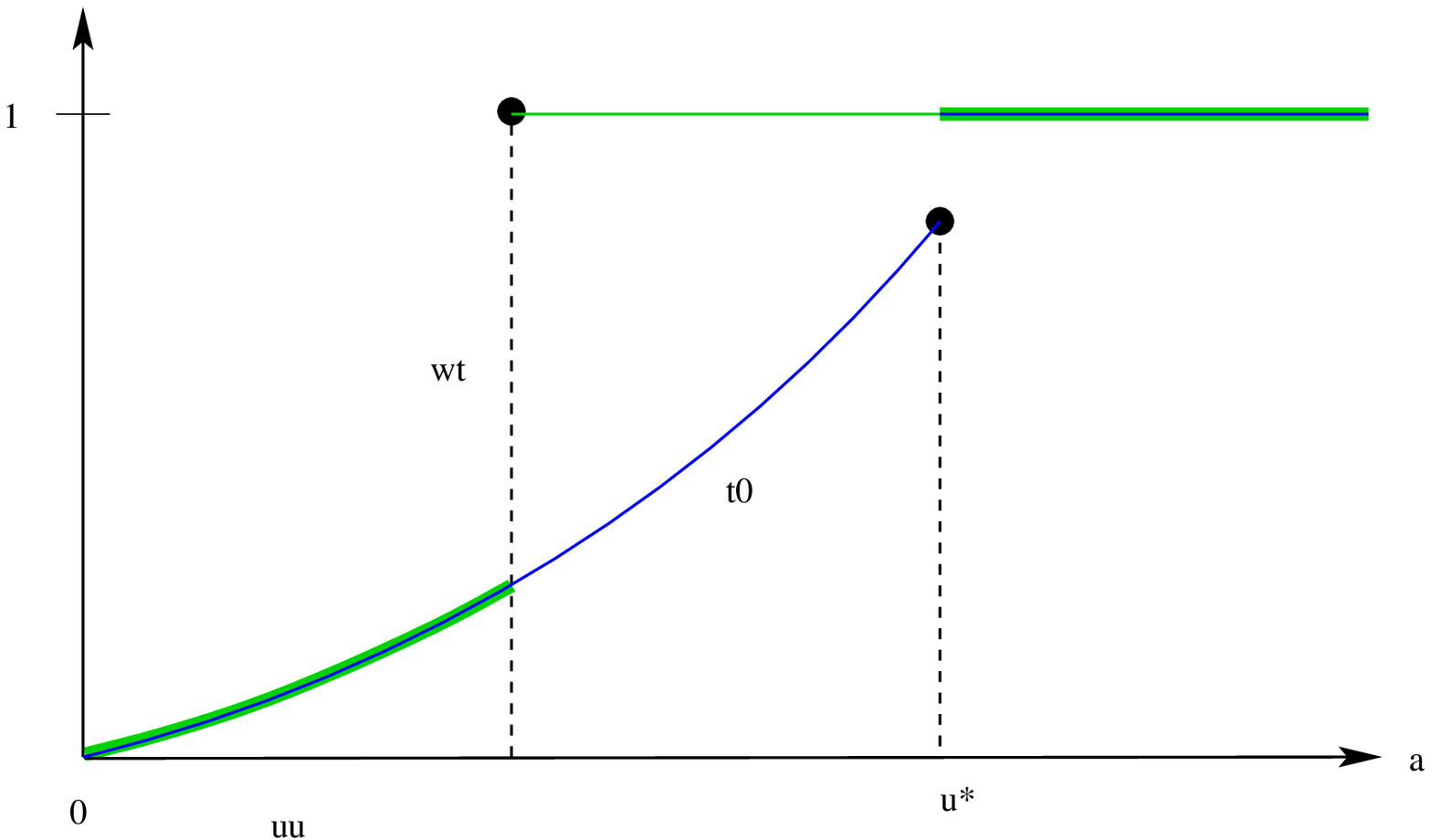}
\end{center}

\bigskip
\begin{tabular}{ll}
Fig.~1: & A heuristic sketch of the functions $\theta_0$ (with a possible but not expected jump\\
&at $u_*$) and $\theta^*$ in (\ref{0.11}). The constant $c_0$ stems from Theorem \ref{theo3.1}.
\end{tabular}

\bigskip
We consider $N \ge 1$ and the discrete box centered at the origin 
\begin{equation}\label{0.3}
D_N = [-N,N]^d \cap \IZ^d.
\end{equation}
We view $D_N$ as the discrete blow-up $(ND) \cap \IZ^d$ of the continuous model shape
\begin{equation}\label{0.4}
D = [-1,1]^d.
\end{equation}
Further, for integer $r \ge 0$, we write $S_r = \{x \in \IZ^d$; $|x|_\infty = r\}$ for the set of points in $\IZ^d$ with sup-norm equal to $r$ and define
\begin{equation}\label{0.5}
\mbox{$\cC^u_r =$ the connected component of $S_r$ in $\cV^u \cup S_r$ (so $S_r \subseteq \cC^u_r$ by convention)}.
\end{equation}
We single out the set of points in $D_N$ that get disconnected by $\cI^u$ from $S_{2N}$, that is $D_N \backslash \cC^u_{2N}$, and its subset $D_N \backslash \cC^u_N$ of points in the interior of $D_N$ that get disconnected by $\cI^u$ from $S_N$. We are interested in their ``excessive presence'' in $D_N$. More precisely, we consider
\begin{equation}\label{0.6}
\nu \in [\theta_0(u), 1),
\end{equation}
and the {\it excess events} (where for $U$ finite subset of $\IZ^d$, $|U|$ denotes the number of points in $U$)
\begin{equation}\label{0.7}
\cA_N = \{|D_N \backslash \cC^u_{2N} | \ge \nu\,|D_N|\} \supseteq \cA^0_N = \{|D_N \backslash \cC^u_N| \ge \nu \,|D_N|\}.
\end{equation}
An asymptotic lower bound on $\IP[\cA^0_N]$ was derived in (6.32) of \cite{Szni19d}. Combined with Theorem 2 of \cite{Szni19c} it shows that
\begin{align}
&\liminf\limits_N \; \dis\frac{1}{N^{d-2}} \;\log \IP[\cA^0_N] \ge - \ov{J}_{u,\nu}, \;\mbox{where} \label{0.8} 
\\[1ex]
& \ov{J}_{u,\nu} = \min \Big\{\mbox{\f $\dis\frac{1}{2d}$} \;\dis\int_{\IR^d} |\nabla \varphi|^2 dz; \varphi \ge 0, \varphi \in D^1(\IR^d), \dis\strokedint_D \;\ov{\theta}_0 \big((\sqrt{u} + \varphi)^2\big) \, dz \ge \nu\Big\} \label{0.9}
\end{align}

\n
and $\ov{\theta}_0$ stands for the right-continuous modification of $\theta_0$, $\strokedint_D ...$ for the normalized integral $\frac{1}{|D|} \,\int_D ...$, with $|D| = 2^d$ the Lebesgue measure of $D$, and $D^1(\IR^d)$ for the space of locally integrable functions $f$ on $\IR^d$ with finite Dirichlet energy that decay at infinity, i.e.~such that $\{|f| > a\}$ has finite Lebesgue measure for all $a > 0$, see Chapter 8 of \cite{LiebLoss01}.

\medskip
The lower bound (\ref{0.8}) is derived via the change of probability method and for $\varphi$ in (\ref{0.9}), $(\sqrt{u} + \varphi)^2 (\frac{\point}{N})$ can heuristically be interpreted as the slowly varying local levels of the {\it tilted interlacements} that enter the derivation of the lower bound (see Section 4 and Remark 6.6 2) of \cite{Szni19d}). It is an open question whether for large enough $\nu$ the minimizers $\varphi$ in (\ref{0.9}) reach the value $\sqrt{u}_* - \sqrt{u}$. The region where they reach the value  $\sqrt{u}_* - \sqrt{u}$ could reflect the occurrence of droplets secluded by the interlacements that might share the burden of producing an excess volume of disconnected points, somewhat in the spirit of the Wulff droplet in the case of the Bernoulli percolation or for the Ising model, see Theorem 2.12 of \cite{Cerf00}, and \cite{Bodi99}.

\medskip
Our main interest here lies with the derivation of an asymptotic upper bound on $\IP[\cA_N]$ $( \ge \IP[\cA^0_N])$ that possibly matches (\ref{0.8}). In the main Theorem \ref{theo4.3} of this article we show that there is a dimension dependent constant $c_0 \in (0,1)$, constructed in Theorem \ref{theo3.1}, so that when $0 < u <\ov{u}$, setting $\theta^*$ to be the function on $\IR_+$ such that $\theta^*(v) = \theta_0(v)$ for $v < (\sqrt{u} + c_0 (\sqrt{\ov{u}} - \sqrt{u}))^2$, and $\theta^*(v) = 1$ otherwise, see Figure 1, one has for all $\nu \in [\theta_0(u), 1)$
\begin{align}
&\limsup\limits_N \; \dis\frac{1}{N^{d-2}} \; \log \IP [\cA_N] \le - J^*_{u,\nu}, \; \mbox{where}\label{0.10}
\\[2ex]
&J^*_{u,\nu} =   \min \Big\{\mbox{\f $\dis\frac{1}{2d}$} \;\dis\int_{\IR^d} |\nabla \varphi|^2 dz; \varphi \ge 0, \varphi \in D^1(\IR^d),  \dis\strokedint_D \; \theta^* \big((\sqrt{u} + \varphi)^2\big) \,dz \ge \nu\Big\}. \label{0.11}
\end{align}
As an application of Theorem \ref{theo4.3}, i.e.~(\ref{0.10}) and (\ref{0.11}), we are able to show that in the ``small excess'' regime the asymptotic upper bound (\ref{0.10}) matches the asymptotic lower bound (\ref{0.8}). More precisely, we show in Corollary \ref{cor4.3} that 
\begin{equation}\label{0.12}
\begin{array}{l}
\mbox{when $0 < u < \ov{u} \wedge \wh{u}$, there exists $\nu_0 > \theta_0 (u)$ such that for all $\nu \in [\theta_0(u), \nu_0)$}
\\
\lim\limits_N \; \dis\frac{1}{N^{d-2}} \;\log \IP[\cA_N] =\lim\limits_N \; \dis\frac{1}{N^{d-2}} \;\log \IP[\cA^0_N] = - \ov{J}_{u,\nu}.
\end{array}
\end{equation}
It is a natural question whether the asymptotics in (\ref{0.12}) actually holds for all $\nu$ in $[\theta_0(u), 1)$. Incidentally, this issue is also related to the question whether $c_0$ mentioned above (\ref{0.10}) and that appears in Theorem \ref{theo3.1} can be chosen arbitrarily close to $1$, see Remarks \ref{rem4.4a} and \ref{rem4.4} 2).  As an aside, if $\ov{u} =u_* $ holds, formally setting $c_0 = 1$ one finds that $\theta^*$ coincides with $\ov{\theta}_0$ and $J^*_{u,\nu}$ with $\ov{J}_{u,\nu}$. Another natural problem is whether the set of disconnected points can be replaced by the set of points outside the infinite cluster $\cC^u_\infty$ of $\cV^u$, and (\ref{0.12}) actually holds with $\cA_N$ replaced by the bigger event $\{|D_N \backslash \cC^u_\infty| \ge \nu |D_N|\}$, when $0 < u < u_*$ and $\theta_0(u) \le \nu < 1$, see Remark \ref{rem4.4} 3).

\medskip
A question of a similar nature to that of the asymptotic behavior of $\IP[\cA^0_N]$ was investigated in \cite{Szni19b}. There, $\cC^u_N$ was replaced by $\wt{\cC}^u_N$, a certain {\it thickening} of $\cC^u_N$ (obtained by adding to $\cC^u_N$ points at a suitable sub-macroscopic distance $\wt{L}_0(N) = o(N)$ of $\cC^u_N$). It was in particular shown in \cite{Szni19b} that for $\nu$ small enough so that the closed Euclidean ball $B_\nu$ with center $0$ and volume $\nu |D| ( = 2^d \nu)$ is contained in the interior of $D$, one has for $0 < u < \ov{u}$,
\begin{equation}\label{0.13}
\left\{
\begin{array}{l}
\limsup\limits_N \; \dis\frac{1}{N^{d-2}} \;\log \IP[|D_N \backslash \wt{\cC}^u_N| \ge \nu |D_N|] \le -\frac{1}{d} ( \sqrt{\ov{u}} - \sqrt{u})^2 {\rm cap}_{\IR^d}(B_\nu),
\\[3ex]
\liminf\limits_N \; \dis\frac{1}{N^{d-2}} \;\log \IP[|D_N \backslash \wt{\cC}^u_N| \ge \nu |D_N|] \ge -\frac{1}{d} ( \sqrt{u_{**}} - \sqrt{u})^2 {\rm cap}_{\IR^d}(B_\nu),
\end{array}\right.
\end{equation}
with ${\rm cap}_{\IR^d}(B_\nu)$ the Brownian capacity of $B_\nu$ (see for instance \cite{PortSton78}, p.~57, 58), and $u_{**}$ ($\ge u_*)$ the critical level for the strongly non-percolative regime of the vacant set (here again $u_{**} = u_*$ is expected, but currently open, so plausibly the right members in (\ref{0.13}) are equal).

\medskip
In the present work (unlike in \cite{Szni19b}) there is no thickening of $\cC^u_N$ or $\cC^u_{2N}$ entering the definitions of $\cA^0_N$ and $\cA_N$ (and both contain the event under the probability in (\ref{0.13})). The variational quantity $\ov{J}_{u,\nu}$ plays the role of $\frac{1}{d} (\sqrt{u}_* - \sqrt{u})^2 \,{\rm cap}_{\IR^d}(B_\nu)$ in (\ref{0.13}) (assuming the equalities $\ov{u} = u_* = u_{**}$). Informally, this last quantity corresponds to a choice of a test function $\varphi = (\sqrt{u}_* - \sqrt{u}) \,h_{B_\nu}$ in (\ref{0.9}), with $h_{B_\nu}$ the equilibrium potential in $\IR^d$ of the ball $B_\nu$, and the replacement of $\ov{\theta}_0(v)$ by the smaller function $1 \{ v \ge u_*\}$. We refer to Proposition 6.5 of \cite{Szni19d} for further links between these variational quantities.

\medskip
Let us say a few words about the proof of the main asymptotic upper bound (\ref{0.10}). The main step is carried out in Proposition \ref{prop4.1}. A substantial challenge stems from the possible presence of ``bubbles'' intersecting $D_N$ that can occupy a macroscopic share of volume, on the surface of which random interlacements have a local level above $u_*$, thus creating ``insulating fences'' that may block connections in $\cV^u$ between the interior of such bubbles and $S_{2N}$. Such ``bubbles'' are non-local objects and accounting for the cost of their presence is a delicate matter. Importantly, in the absence of a thickening of $\cC^u_{2N}$, the ``bubbles'' that we are faced with are irregular and lack inner depth. This specific feature precludes the use of the coarse graining procedure developed in Section 4 of \cite{NitzSzni} that played a crucial role in \cite{ChiaNitz20a} and \cite{Szni19b}, as well as in \cite{Nitz18}, \cite{ChiaNitz20a}.

\begin{center}
\psfrag{DN}{$D_N$}
\includegraphics[width=6cm]{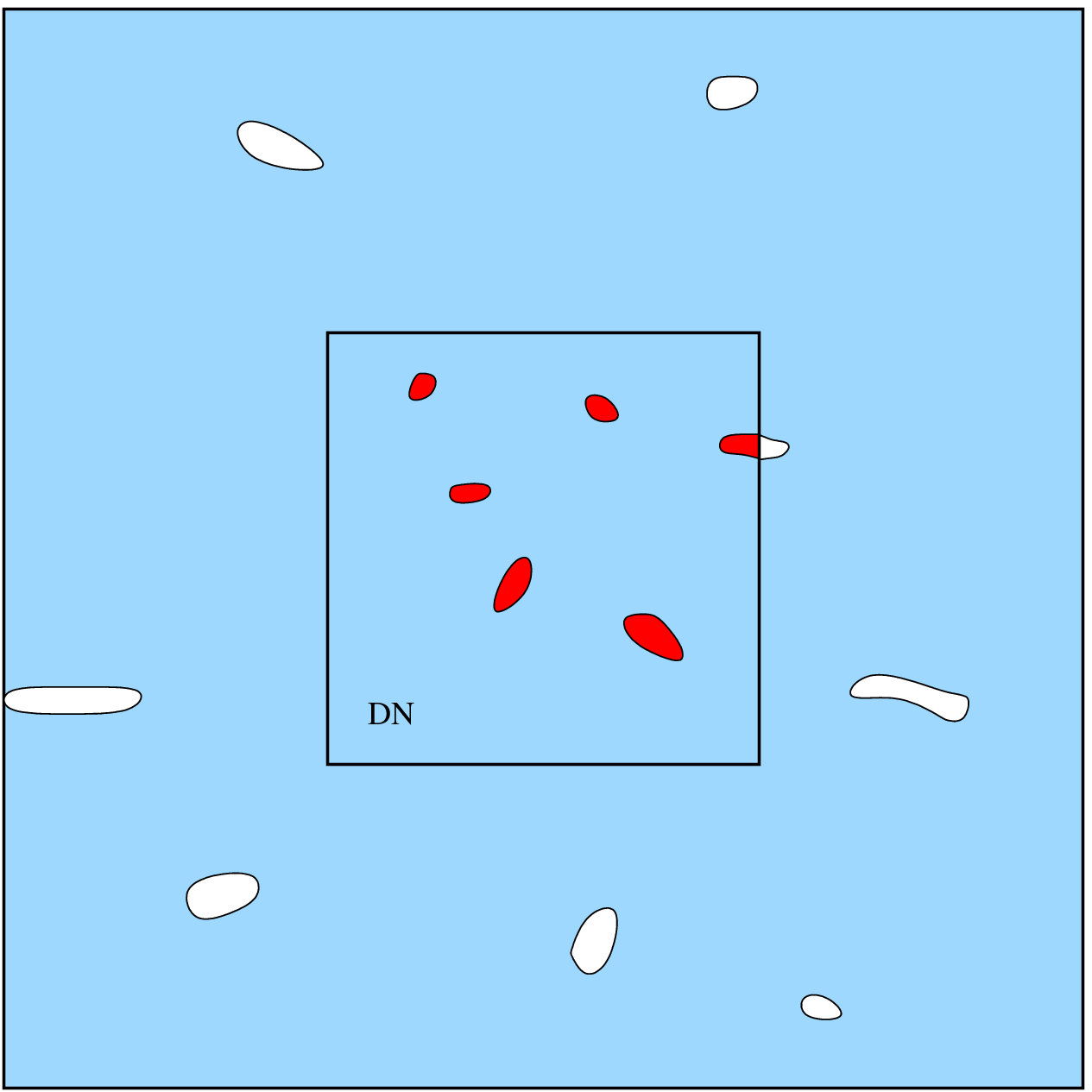}
\end{center}

\medskip
\begin{center}
\begin{tabular}{ll}
Fig.~2: & An informal illustration of the bubble set $B u b$ from (\ref{1.47}) in red. \\
& The light blue region consists of $B_1$-boxes where the random \\
& set $\cU_1$ enters deeply enough in $B_1$ (see (\ref{1.46})).
\end{tabular}
\end{center}

\medskip\n
In the first main step corresponding to Theorem \ref{theo2.1} we perform local averaging in the $B_1$-boxes (of scale $L_1$) contained in $D_N$ that lie outside the {\it bubble set} $B u b$. After this step the task of bounding the probability of $\cA_N$ is replaced by that of bounding the probability of $\cA'_N$, see (\ref{2.11}), that roughly corresponds to the event 
\begin{equation}\label{0.14}
| B u b| + \dis\sum\limits_{B_1 \subseteq D_N \backslash B u b} \wt{\theta} (u_{B_1}) \,|B_1| \ge \nu' \,|D_N|,
\end{equation}
where $\nu'$ is slightly smaller than $\nu$, $u_{B_1}$ denotes the local level of the interlacements in the box $B_1$ (see (\ref{1.45})), and the function $\wt{\theta}$ equals $\theta_0$ up to a level close to $\ov{u}$ and then equals $1$. As in \cite{Szni19d} the scale $L_1$ of the $B_1$-boxes is roughly $N^{\frac{2}{d}}$, see (\ref{1.8}), so that $(N/L_1)^d$ roughly equals $N^{d-2}$. This choice corresponds to the following constraints. The scale $L_1$ should not be too large, so that with overwhelming probability most $B_1$ boxes behave well with respect to local averaging, and the scale $L_1$ should not be too small, so that coarse graining for the local levels $u_{B_1}$ of all $B_1$-boxes in $D_N$ can be performed with $\exp(o(N^{d-2}))$ complexity. The bubble set $B u b$ is defined in (\ref{1.47}). It consists of the $B_1$-boxes in $D_N$ where a certain random set $\cU_1$ does not get ``deep'' inside $B_1$. The random set $\cU_1$, see (\ref{1.40}), as in \cite{NitzSzni}, is defined via exploration starting outside a larger box concentric to $D_N$ (namely $[-3N,3N]^d$) with $B_0$-boxes of size $L_0$ comparable to $N^{\frac{1}{d-1}}$ (much smaller than $L_1$, see (\ref{1.7}), (\ref{1.8})) that have a good local behavior (so-called $(\alpha, \beta,\gamma)${\it -good boxes}, see (\ref{1.38})), and such that the local level of random interlacements in these boxes remains strictly below $\ov{u}$ (namely at most $\beta$, see (\ref{1.40})). The random set $\cU_1$ brings along a {\it profusion of highways} in the vacant set $\cV^u$ that permit to exit $[-2N,2N]^d$, and thus reach $S_{2N}$ when starting in $D_N$. The choice of the scale $L_0$ corresponds to the following constraints. The scale $L_0$ (as reflected in the choice $N^{\frac{1}{d-1}}$) should not be too large so that $D_N$ contain at least $N^{d-2}$ columns of $L_0$-boxes and the $(\alpha,\beta,\gamma)$-bad boxes do not spoil projection arguments. It should also be not too small, see (\ref{2.73}), (\ref{2.81}), so that the communication between the local levels in $B_1$-boxes and the local levels in the $B_0$-boxes, which they contain, functions harmoniously. The choice $L_0 = [N^{\frac{1}{d-1}}]$ in (\ref{1.7}) fits these requirements.

\medskip
The bubble set $B u b$ that appears in (\ref{0.14}) thus consists of $B_1$-boxes in $D_N$ that are not met ``deep inside'' by $\cU_1$. No thickening is performed in the original question we handle and $B u b$ is quite irregular. In particular, it lacks sufficient inner depth (roughly corresponding to the ``nearly macroscopic'' scale $\wh{L}_0$ in (4.19) of \cite{NitzSzni}) and the coarse graining procedure of Section 4 of \cite{NitzSzni} does not apply. In Section 3 we devise a new coarse graining approach, using a ``bird's-eye view'' to address this central issue. In the crucial Theorem \ref{theo3.1} we construct a random set $C_\omega$ that can take at most $\exp(o(N^{d-2}))$ shapes, which has small volume, which is made of well-spaced $B_0$-boxes that are $(\alpha, \beta,\gamma)$-good with local level above $\beta$, and which is such that the (discrete) equilibrium potential of $C_\omega$ is at least $c_0$ on the bubble set $B u b$ apart from a set of small volume. This is where the important constant $c_0$ entering the definition of the function $\theta^*$ above (\ref{0.10}) appears. The random set $C_\omega$ is extracted from the $B_0$-boundary of $\cU_1$ (see below (\ref{1.41})), and the coarse graining procedure that we employ uses some ideas from the {\it method of enlargement of obstacles} (see for instance Chapter 4 in \cite{Szni98a}). In Section 4 we complete the proof of Proposition \ref{prop4.1} (which is the main step towards (0.10), (0.11)). An important aspect is to find an adequate formulation implementing the constraint (\ref{0.14}) that behaves well under scaling limit. This corresponds to (\ref{4.26}) - (\ref{4.29}), where the event $\cA'_N$ gets coarse grained and certain non-negative super-harmonic functions solving an obstacle problem accounting for the random set $C_\omega$ from Theorem \ref{theo3.1} and the local levels $u_{B_1}$ away from $C_\omega$ enter the new formulation. After that the proof proceeds along the same lines as in Section 5 of \cite{Szni19d}.

\medskip
We will now describe the organization of this article. Section 1 collects some notation and recalls various facts about simple random walk, potential theory, and random interlacements. Lemma \ref{lem1.2} due to \cite{AsseScha20} and Lemma \ref{lem1.1} are related to capacity and their application enters the coarse graining procedure for the construction of the random set $C_\omega$ in Theorem \ref{theo3.1}. The important random set $\cU_1$ and the bubble set $B u b$ are respectively defined in (\ref{1.40}) and (\ref{1.47}). Section 2 is devoted to the proof of Theorem \ref{theo2.1} where local averaging is performed, and the event $\cA'_N$ (in essence corresponding to (\ref{0.14})) is introduced. An extensive use is made of the important soft local time technique of \cite{PopoTeix15} in the version developed by \cite{ComeGallPopoVach13}. When looking at well-separated boxes of a given size it provides access in each box to an ``undertow'' (corresponding to the local level of random interlacements in the box) and a ``wavelet part'' (corresponding to a collection of excursions) with good independence properties for the wavelet parts. Section 3 is devoted to the construction of the random set $C_\omega$ in Theorem \ref{theo3.1}. This is where the pivotal constant $c_0$ appears. Finally, Section 4 contains Theorem \ref{theo4.3} that proves the crucial upper bound (\ref{0.10}), (\ref{0.11}). However, the main work is carried out in Proposition \ref{prop4.1}. The application to the small excess regime is presented in Corollary \ref{cor4.3} where (\ref{0.12}) is proved. Remark \ref{rem4.4} lists several open questions.

\medskip
To conclude, let us state our convention about constants. Throughout the article we denote by $c, \wt{c}, c'$ positive constants changing from place to place that simply depend on the dimension $d$. Numbered constants $c_0,c_1,c_2, \dots$ refer to the value corresponding to their first appearance in the text. Dependence on additional parameters appears in the notation.

\section{Notation, some useful results and random sets}
\setcounter{equation}{0}

In this section we introduce some further notation. We recall and collect some facts concerning random walks, potential theory, and random interlacements. We also introduce important random sets such as $\cU_1$, see (\ref{1.40}), and the ``bubble set''  $B u b$, see (\ref{1.47}).

\medskip
First some notation:  for $(a_n)_{n \ge 1}$ and $(b_n)_{n \ge 1}$ positive sequences, $a_n \gg b_n$ or $b_n = o(a_n)$ means that $b_n/a_n \underset{n}{\longrightarrow} 0$. We write $| \cdot |$ and $| \cdot |_\infty$ for the Euclidean and the supremum norms on $\IR^d$. Throughout we assume that $d \ge 3$. Given $x \in \IZ^d$ and $r \ge 0$, we let $B(x,r) = \{y \in \IZ^d; |y - x|_\infty \le r\}$ stand for the closed ball of radius $r$ around $x$ for the supremum distance (note that $D_N$ in (\ref{0.3}) coincides with $B(0,N)$). Given $L \ge 1$ integer, we say that a subset $B$ of $\IZ^d$ is an $L$-box when it is a translate of $\IZ^d \cap [0,L)^d$. We sometimes write $[0,L)^d$ in place of $\IZ^d \cap [0,L)^d$ when no confusion arises. Given $A, A'$ subsets of $\IZ^d$, we denote by $d_\infty (A,A') = \inf\{|x - x'|_\infty$; $x \in A, x' \in A'\}$ the mutual supremum distance between $A$ and $A'$, and write $d_\infty(x,A')$ for simplicity when $A = \{x\}$. We let ${\rm diam}(A) = \sup\{|x-x'|_\infty$; $x$, $x' \in A\}$ stand for the sup-norm diameter of $A$, and $|A|$ for the cardinality of $A$. We write $A \subset \subset \IZ^d$ to state that $A$ is a finite subset of $\IZ^d$. We denote by $\partial A = \{y \in \IZ^d \backslash A$; $\exists x \in A$, $|y - x| = 1\}$ and $\partial_i A = \{x \in A$; $\exists y \in \IZ^d \backslash A$, $|y - x| = 1\}$ the boundary and the internal boundary of $A$. When $f,g$ are functions on $\IZ^d$, we write $\langle f,g \rangle = \sum_{x \in \IZ^d} f(x) \,g(x)$ when the sum is absolutely convergent. We also use the notation $\langle \rho, f \rangle$ for the integral of a function $f$ (on an arbitrary space) with respect to a measure $\rho$, when this quantity is meaningful.

\medskip
Concerning connectivity properties, we say that $x,y$ in  $\IZ^d$ are neighbors when $|y - x| = 1$ and call $\pi$: $\{0,\dots,n\} \rightarrow \IZ^d$ a path, when $\pi(i)$ and $\pi(i-1)$ are neighbors for $1 \le i \le n$. For $A, B, U$ subsets of $\IZ^d$, we say that $A$ and $B$ are connected in $U$ and write $A \stackrel{U}{\longleftrightarrow} B$ when there exists a path with values in $U$ which starts in $A$ and ends in $B$. When no such path exists we say $A$ and $B$ are not connected in $U$ and write $A \stackrel{U}{\longleftrightarrow} \hspace{-2.7ex}\mbox{\scriptsize$/$} \; \;  B$ .

\medskip
We turn to the notation concerning continuous time simple random walk on $\IZ^d$. For $U \subseteq \IZ^d$, we write $\Gamma(U)$ for the set of right-continuous, piecewise constant functions from $[0,\infty)$ to $U \cup \partial U$ with finitely many jumps on any finite interval that remain constant after their first visit to $\partial U$. For $U \subset \subset \IZ^d$ the space $\Gamma (U)$ conveniently carries the law of certain excursions contained in the trajectories of the random interlacements. We also view the law $P_x$ of the continuous time simple random walk on $\IZ^d$ with unit jump rate, starting in $x \in \IZ^d$, as a measure on $\Gamma (\IZ^d)$. We write $E_x$ for the corresponding expectation. Given $U \subseteq \IZ^d$, we denote by $H_U = \inf\{t \ge 0$; $X_t \in U\}$ and $T_U = \inf\{t \ge 0$; $X_t \notin U\}$ the respective entrance time in $U$ and exit time from $U$.

\medskip
We denote by $g(\cdot,\cdot)$ the Green function of the simple random walk:
\begin{equation}\label{1.1}
g(x,y) = E_x \Big[\dis\int^\infty_0 1 \{X_s = y\} \,ds\Big], \;\mbox{for $x,y \in \IZ^d$},
\end{equation}
and when $f$ is a function on $\IZ^d$ such that $\sum_{y \in \IZ^d} g(x,y) \, |f(y)| < \infty$ for all $x$ in $\IZ^d$, we write
\begin{equation}\label{1.2}
G f(x) = \dis\sum\limits_{y \in \IZ^d} g(x,y) \,f(y), \; \mbox{for $x \in \IZ^d$}.
\end{equation}
The Green function is symmetric and translation invariant. Further, one knows that $g(x,y) (= g(x-y,0)) \sim \frac{d}{2} \;\Gamma (\frac{d}{2} - 1) \, \pi^{- \frac{d}{2}} |y - x|^{2 -d }$, as $|y - x| \rightarrow \infty$ (see Theorem 1.5.4, p.~31 of \cite{Lawl91}), and we denote by $c_*$ a positive constant such that
\begin{equation}\label{1.3}
g(x,y) \le c_* \,|y - x|^{2-d}, \;\mbox{for $x,y \in \IZ^d$}.
\end{equation}
Given $A \subset \subset \IZ^d$, we write $e_A$ for the equilibrium measure of $A$ and ${\rm cap}(A)$ for its total mass, the capacity of $A$. The equilibrium measure $e_A$ is supported by the internal boundary of $A$ and one knows that
\begin{equation}\label{1.4}
\mbox{$G e_A = h_A$ where $h_A(x) = P_x [H_A < \infty]$, $x \in \IZ^d$, is the equilibrium potential of $A$}.
\end{equation}
When $A \not= \phi$, we also write $\ov{e}_A = e_A / {\rm cap}(A)$ for the normalized equilibrium measure of $A$. In the special case of boxes, one knows (for instance by \cite{Lawl91}, p.~31) that with $B= [0,L)^d$,
\begin{equation}\label{1.5}
\mbox{$c L^2 \le G 1_B (x) \le c' \,L^2$, for $x \in B$ and $L \ge 1$},
\end{equation}
as well as (see (2.16), p.~53 of \cite{Lawl91})
\begin{equation}\label{1.6}
\mbox{$c L^{d-2} \le {\rm cap}(B) \le c' \,L^{d-2}$, for $L \ge 1$}.
\end{equation}
Apart from the macroscopic scale $N$ (governing the size of the box $D_N$ in (\ref{0.3})) two length scales will play an important role for us:
\begin{align}
L_0 & =[N^{\frac{1}{d-1}}], \;\mbox{and} \label{1.7}
\\[1ex]
L_1 &=\mbox{$k_N \,L_0$, where $k_N$ is the integer such that $k_N \,L_0 \le N^{\frac{2}{d}} (\log N)^{\frac{1}{d}} < (k_N + 1) \,L_0$}.\label{1.8}
\end{align}
Note that $\frac{1}{d-1} < \frac{2}{d}$ so that $k_N \r \infty$ and $L_1/L_0 \r \infty$, with $L_0 \sim N^{\frac{1}{d-1}}$ and $L_1 \sim N^{\frac{2}{d}} (\log N)^{\frac{1}{d}}$ as $N \r \infty$. Also $\frac{1}{d} < \frac{1}{d-1}$, so that $(L_1 / L_0)^2 = o(L_1)$ as $N \r \infty$, and we will use this feature in the next section, see (\ref{2.73}) and (\ref{2.81}).

\medskip
We will call $B_0$-box or $L_0$-box any box of the form
\begin{equation}\label{1.9}
B_{0,z} = z + [0,L_0)^d, \;\mbox{where} \; z \in \IL_0 \stackrel{\rm def}{=} L_0 \,\IZ^d.
\end{equation}
Often we will simply write $B_0$ to refer to a generic box $B_{0,z}$, $z \in \IL_0$ and call $z$ the base point of $B_0$. Likewise, we will call $B_1$-box or $L_1$-box any box of the form
\begin{equation}\label{1.10}
\mbox{$B_{1,z} = z + [0,L_1)^d$, where $z \in \IL_1 \stackrel{\rm def}{=} L_1 \,\IZ^d (\subseteq \IL_0$ by (\ref{1.8}), (\ref{1.7}))},
\end{equation}
and write $B_1$ for a generic box $B_{1,z}, z \in \IL_1$.

\medskip
At this stage it is perhaps helpful to provide some comments on the role of these boxes and their size. In essence, following \cite{Szni19d}, the $B_1$-boxes will be used to perform local averages, see for instance (\ref{2.19}). The fact that $(N/L_1)^d$ is comparable to $N^{d-2}$ (up to a logarithmic factor) ensures on the one hand that with overwhelming probability the local averaging involving the local level of random interlacements in the box is applicable to most of the $B_1$-boxes that we consider, and on the other hand that there are not too many boxes, so that we can perform coarse graining, see below (\ref{4.29}). As for the $B_0$-boxes, following \cite{NitzSzni}, they will be used to construct the random set $\cU_1$ that provides ``highways in $\cV^u$'' to get beyond $[-2N, 2N]^d$, see below (\ref{1.41}). The bubble set (\ref{1.47}) will morally correspond to the $B_1$-boxes contained in $D_N \backslash \cU_1$. Here, having $L_0$ comparable to $N^a$  with $\frac{1}{d} < a \le \frac{1}{d-1}$, ensures that the number of columns of $B_0$-boxes in $D_N$ is at least of order $N^{d-2}$ so that we can use projection arguments and cope with the occurrence of bad $B_0$-boxes, see for instance (\ref{3.24}) - (\ref{3.27}), but also ensures that the $B_0$-boxes sitting in a $B_1$-box are large enough and typically receive a sufficient number of excursions, see (\ref{2.71}) - (\ref{2.75}).

\medskip
The next two lemmas will be applied in the proof of Theorem \ref{theo3.1} in Section 3 as part of the construction of the crucial random set $C_\omega$ (made of $B_0$-boxes) that will help us keep track of the cost of the bubble set (\ref{1.47}), see below (\ref{3.15}) and below (\ref{3.57}). In the statement below, {\it coordinate projection} refers to any one of the $d$ canonical projections on the respective hyperplanes of points with vanishing $i$-th coordinate, for $1 \le i \le d$.

\begin{lemma}\label{lem1.1}
Given $K \ge 100d$, $a \in (0,1)$, then for large $N$, for any $L$-box $B$ with $L \ge L_1$, and any set $A$ union of $B_0$-boxes contained in $B$ such that for a coordinate projection $\pi$ one has
\begin{equation}\label{1.11}
|\pi(A) | \ge a\, |B|^{\frac{d-1}{d}},
\end{equation}
one can find a subset $\wt{A}$ of $A$, which is a union of $B_0$-boxes having base points with respective $\pi$-projections at mutual supremum distance at least $\ov{K} L_0$ (with $\ov{K} = 2 K+3)$, and such that
\begin{equation}\label{1.12}
{\rm cap}(\wt{A}) \ge c(a) \,|B|^{\frac{d-2}{d}} \;\mbox{and} \;\; | \pi(\wt{A})| \ge \ov{K}\;\!^{-(d-1)} |\pi (A)|.
\end{equation}
\end{lemma}

\begin{proof}
We first trim $A$ and pick one $B_0$-box in $A$ in each column in the $\pi$-direction that intersects $A$. The columns of $B_0$-boxes in the $\pi$-direction can be split into $\ov{K}\;\!^{d-1}$ collections of $\ov{K} L_0$-spaced columns. Then, we restrict this trimmed set to one of the $\ov{K}\;\!^{d-1}$ collections so as to obtain
\begin{equation}\label{1.13}
\left\{ \begin{array}{l}
\mbox{$\wt{A}$ subset of $A$ that contains at most one $B_0$-box per column, in columns that} 
\\
\mbox{are $\ov{K} L_0$-spaced, with $\pi$-projection $\pi (\wt{A})$ such that $| \pi(\wt{A})| \ge \ov{K}\;\!^{-(d-1)} | \pi (A)|$}.
\end{array}\right.
\end{equation}
We then introduce the probability measure $\mu$ supported by $\wt{A}$ (see below (\ref{1.4}) for notation):
\begin{equation}\label{1.14}
\mu = \dis\frac{1}{\wt{n}} \; \dsl_{B_0 \subseteq \wt{A}} \, \ov{e}_{B_0}, \; \mbox{with} \; \wt{n} = \dis\frac{|\wt{A}|}{|B_0|} \ge a \ov{K}\;\!^{-(d-1)} \Big(\dis\frac{|B|}{|B_0|}\Big)^{\frac{d-1}{d}} \;\mbox{(by (\ref{1.11}), (\ref{1.13}))}.
\end{equation}
One has the variational identity ${\rm cap}(\wt{A}) = \sup\{\langle \rho \otimes \rho, g \rangle^{-1}$; $\rho$ probability measure supported by $\wt{A}\}$ (where $\otimes$ denotes the product of measures and $g$ the Green function as in (\ref{1.1})). Hence, ${\rm cap} (\wt{A}) \ge \langle \mu \otimes \mu, g\rangle^{-1}$, and our aim is now to bound $\langle \mu \otimes \mu, g\rangle$ from above in order to prove (\ref{1.12}). We first write with hopefully obvious notation
\begin{equation}\label{1.15}
\begin{split}
\langle \mu \otimes \mu, g \rangle & \le \sup\limits_{B_0 \subseteq \wt{A}} \;\dis\frac{1}{\wt{n}} \; \dsl_{B'_0 \subseteq \wt{A}} \langle \ov{e}_{B_0} \otimes \ov{e}_{B'_0}, g\rangle
\\
& \le \sup\limits_{B_0 \subseteq \wt{A}} \;\dis\frac{1}{\wt{n}} \; (\langle \ov{e}_{B_0} \otimes \ov{e}_{B_0}, g\rangle + \dsl_{B'_0 \subseteq \wt{A}, B'_0 \not= B_0} \langle \ov{e}_{B_0} \otimes \ov{e}_{B'_0}, g \rangle).
\end{split}
\end{equation}
We note that by (\ref{1.4})
\begin{equation}\label{1.16}
\langle \ov{e}_{B_0} \otimes \ov{e}_{B_0}, g\rangle = {\rm cap}(B_0)^{-1},
\end{equation}
and for the second term in the last line of (\ref{1.15}), setting $x_0$ as the unique point in $\IL_0 \cap B_0$, see (\ref{1.9}), $y_0 = \pi (x_0)$, and likewise $x'_0$ as the unique point in the $\IL_0 \cap B'_0$, $y'_0 = \pi(x'_0)$, we see that for any $B_0 \subseteq \wt{A}$ by (\ref{1.3})
\begin{equation}\label{1.17}
\dis\frac{1}{\wt{n}} \; \dsl_{B'_0 \subseteq \wt{A}, B'_0 \not= B_0} \langle \ov{e}_{B_0} \otimes \ov{e}_{B'_0}, g \rangle \le \dis\frac{c}{\wt{n}}  \dsl_{B'_0 \subseteq \wt{A}, B'_0 \not= B_0} |y'_0 - y_0|^{-(d-2)} \stackrel{\rm def}{=} S_{B_0},
\end{equation}
where it should be observed that due to (\ref{1.13}), $y'_0 - y_0 \in \ov{K} L_0 \,\IZ^{d-1}$ in the last sum (and we have tacitly identified $\pi(\IZ^d)$ with $\IZ^{d-1}$). We then consider
\begin{equation}\label{1.18}
\begin{array}{l}
\mbox{$\wt{B}$ the Euclidean ball in $\ov{K} L_0 \, \IZ^{d-1}$ with center $0$ and smallest radius $\wt{R}$}\\
\mbox{such that $\wt{B}$ contains $\wt{n}$ points.}
\end{array}
\end{equation}

\medskip\n
Note that due to the lower bound on $\wt{n}$ in (\ref{1.14}), for large $N$, one has
\begin{equation}\label{1.19}
c \,\wt{n} \ge \Big(\dis\frac{\wt{R}}{\ov{K} L_0}\Big)^{d-1} \ge c' \, \wt{n}.
\end{equation}
Looking whether $y'_0 - y_0$ lies in $\wt{B}$ or outside $\wt{B}$ in the sum defining $S_{B_0}$ in (\ref{1.17}) we thus find that for any $B_0 \subseteq \wt{A}$
\begin{equation}\label{1.20}
\begin{split}
S_{B_0} & \le \dis\frac{c}{\wt{n}} \; \dsl_{0 \not= y \in \wt{B} \cap (\ov{K} L_0 \,\IZ^{d-1})} |y|^{-(d-2)} \le \dis\frac{c'}{\wt{n}} \;(\ov{K} L_0)^{-(d-2)} \dsl_{1 \le \ell \le c \wt{n}^{1/(d-1)}} 1
\\
& \le c''\,(\ov{K} L_0)^{-(d-2)} \,\wt{n}^{-\frac{d-2}{d-1}} \; \stackrel{(\ref{1.19})}{\le} c\,\wt{R}^{-(d-2)}.
\end{split}
\end{equation}
Thus, coming back to (\ref{1.15}), we find with (\ref{1.16}), (\ref{1.17}) and the above bound that
\begin{equation}\label{1.21}
\langle \mu \otimes \mu, g \rangle \le \dis\frac{c}{\wt{n} L_0^{d-2}} + \dis\frac{c}{\wt{R}^{d-2}} \stackrel{(\ref{1.19})}{\le} c'' \; \dis\frac{\ov{K}\!\,^{d-1} L_0}{\wt{R}^{d-1}} + \dis\frac{c}{\wt{R}^{d-2}} \le \dis\frac{c}{\wt{R}^{d-2}} \,\Big(1 + \wt{c} \; \dis\frac{\ov{K}\!\,^{d-1} L_0}{\wt{R}}\Big).
\end{equation}
This shows that for large $N$,
\begin{equation}\label{1.22}
{\rm cap}(\wt{A}) \ge \langle \mu \otimes \mu, g \rangle^{-1} \ge c \,\wt{R}^{d-2} (1 + c' \,K^{d-1} L_0 / \wt{R})^{-1}.
\end{equation}
We also know by (\ref{1.19}) that for large $N$
\begin{equation}\label{1.23}
\wt{R} \ge c\, \ov{K} L_0 \,\wt{n}^{\frac{1}{d-1}} \stackrel{(\ref{1.14})}{\ge} c \,a^{\frac{1}{d-1}} \,L \stackrel{L \ge L_1}{\ge} K^{d-1} L_0,
\end{equation}
and we find that
\begin{equation}\label{1.24}
{\rm cap}(\wt{A}) \ge c' \wt{R}^{d-2} \stackrel{(\ref{1.23})}{\ge} c(a) \,L^{d-2} = c(a) \,|B|^{\frac{d-2}{d}}.
\end{equation}
Together with (\ref{1.13}), this completes the proof of (\ref{1.12}) and hence of Lemma \ref{lem1.1}.
\end{proof}

The next lemma is due to \cite{AsseScha20}, and it will also be used in the construction of the random set $C_\omega$ in Theorem \ref{theo3.1} of Section 3. 

\begin{lemma}\label{lem1.2} $(\ov{K} = 2 K + 3)$

\medskip\n
For $K,N \ge c_1$, when $\wt{A}$ is a union of $B_0$-boxes with base points that are at mutual $| \cdot |_\infty$-distance at least $\ov{K} L_0$, then there exists a union of $B_0$-boxes $A' \subseteq \wt{A}$ such that
\begin{equation}\label{1.25}
{\rm cap}(A') \ge  c\, {\rm cap}(\wt{A})  \;\,\mbox{and}\; \,\dis\frac{|A'|}{|B_0|} \le c' \dis\frac{{\rm cap}(\wt{A})}{{\rm cap}(B_0)}. 
\end{equation}
\end{lemma}

\begin{proof}
The claim is a straightforward consequence of Theorem 1.4 of \cite{AsseScha20}. 
\end{proof}

We will now introduce some notation and collect several facts concerning random interlacements. We also refer to the end of Section 1 of \cite{Szni17} and the references therein for more details. The random interlacements $\cI^u$, $u \ge 0$, and the corresponding vacant sets $\cV^u = \IZ^d \backslash \cI^u, u \ge 0$, are defined on a certain probability space denoted by $(\Omega, \cA, \IP)$. In essence, $\cI^u$ corresponds to the trace left on $\IZ^d$ by a certain Poisson point process of doubly infinite trajectories modulo time-shift that tend to infinity at positive and negative infinite times, with intensity proportional to $u$. As $u$ grows, $\cV^u$ becomes thinner and there is a critical value $u_* \in (0, \infty)$ such that for $u < u_*$, $\IP$-a.s., $\cV^u$ has an infinite component, and for $u > u_*$ all components of $\cV^u$ are finite, see \cite{Szni10}, \cite{SidoSzni09}, as well as the monographs \cite{CernTeix12}, \cite{DrewRathSapo14c}.

\medskip
In this work we are mainly interested in the strong percolative regime of $\cV^u$ that will correspond to $u < \ov{u}$, where following (2.3) of \cite{Szni17} we set
\begin{equation}\label{1.26}
\mbox{$\ov{u} = \sup\{ s > 0$, such that for all $v < u$ in $(0,s)$, (\ref{1.27}) and (\ref{1.28}) hold$\}$},
\end{equation}
where writing $B = [0,L)^d$, the condition (\ref{1.27}) is
\begin{equation}\label{1.27}
\mbox{$\lim\limits_L \;\dis\frac{1}{\log L} \; \log \IP[\cV^u \cap B$ has no component of diameter at least $\frac{L}{10}] = - \infty$,}
\end{equation}
and the condition (\ref{1.28}) is that for all $B' = L \, e + B$ with $|e| = 1$, with $D = [-3L, 4L)^d$ (a subset of $\IZ^d$ not to be confused with (\ref{0.4})):
\begin{align}
\lim\limits_L \; \dis\frac{1}{\log L} \; \log \IP\; [ & \mbox{there exist connected components of $B \cap \cV^u$ and $B' \cap \cV^u$ of}\label{1.28}
\\[-2ex]
& \mbox{diameter at least $\frac{L}{10}$, which are not connected in $D \cap \cV^v] = - \infty$}. \nonumber
\end{align}
One knows that $\ov{u} > 0$ (by \cite{DrewRathSapo14a}) and that $\ov{u} \le u_*$, see (2.4), (2.6) of \cite{Szni17}. Also as explained in Remark 2.1 1) of \cite{Szni17},
\begin{align}
& \mbox{for $u > v$ in $(0, \ov{u})$},   \label{1.29}
\\
& \lim\limits_L \; \mbox{\f $\dis\frac{1}{\log L}$} \; \log \IP  \;  [ \mbox{there exist two connected components of $B \cap \cV^u$ of}   \nonumber
\\[-2ex]
&\qquad \qquad \qquad \quad  \;\; \mbox{diameter at least $\frac{L}{10}$, which are not connected in $D \cap \cV^v] = - \infty$.}\nonumber
\end{align}

\begin{remark}\label{rem1.3} \rm
Let us also mention that with $\wt{B} = [-L,2L)^d$ and $D$ as above
\begin{align}
&\mbox{for $v > w$ in $(0, \ov{u})$},\label{1.30}
\\
& \lim\limits_L \; \mbox{\f $\dis\frac{1}{\log L} $}\; \log \IP \;[ \mbox{there exist two connected components of $\wt{B} \cap \cV^v$ of}\nonumber
\\[-2ex]
&\qquad \qquad \qquad \quad  \mbox{diameter at least $\frac{L}{10}$, which are not connected in $D \cap \cV^w] = - \infty$.}\nonumber
\end{align}
Indeed, to prove (\ref{1.30}), one considers the scales $L'' = [L / 10^3] \le L' = [L/ 10^2] \le L$. Given $v > w$ in $(0,\ov{u})$, except on a set of super-polynomially decaying probability in $L$, for all boxes $z + [0,L'')^d$, $z \in \IZ^d$, intersecting $[-2L, 3L)^d$, and with $v, \frac{v+ w}{2}$ in place of $u,v$ all the events corresponding to the complement of what appears in (\ref{1.27}), (\ref{1.28}) are satisfied, as well as for all boxes $z + [0,L')^d$, $z \in \IZ^d$, intersecting $[-2L, 3L)^d$, with $\frac{v + w}{2}$, $w$ in place of $u,v$, the events corresponding to the complement of (\ref{1.29}) are satisfied. Then, for large $L$ on the intersection of the above events, given $A_1, A_2$ connected components of $\wt{B} \cap \cV^v$ with diameter at least $\frac{L}{10}$, one can construct a path of non-intersecting nearest neighbor $L''$-boxes in $[-2L,3L)^d$, such that the restriction of $A_1$ to the first box contains a connected component of diameter at least $\frac{L''}{10}$ and the last box meets $A_2$. Then, one can construct a path in $\cV^{\frac{v + w}{2}} \cap D$ starting in $A_1$ with an end point at supremum distance at most $L''$ from $A_2$ belonging to the last box of the path of $L''$-boxes. One can then consider an $L'$-box with center (in $\IR^d$) within supremum distance $1$ from the center of last box in the path of $L''$-boxes, and link $A_2$ in $\cV^w \cap D$ with the path in $\cV^{\frac{v+w}{2}} \cap D$ that linked $A_1$ to a point of the last box of the path of $L''$-boxes. This provides a path in $\cV^w \cap D$ between $A_1$ and $A_2$. The claim (\ref{1.30}) now follows. \hfill $\square$

\end{remark}

The equality $\ov{u} = u_*$ is expected but presently open. In the closely related model of the level-set percolation for the Gaussian free field a corresponding equality can be proved as shown in the recent work  \cite{DumiGoswRodrSeve19}.

\medskip
An additional critical level $\wh{u} \in (0,u_*]$ has been helpful in the study of the $C^1$-property of the percolation function $\theta_0$ in (\ref{0.2}), see Theorem 1 of \cite{Szni19c}. In the present work, it will show up in the context of Corollary \ref{cor4.3}, see (0.12), when we identify the exponential rate of decay of $\IP[\cA_N]$ (or $\IP[\cA_N^0]$) for $\nu > \theta_0(u)$ close to $\theta_0(u)$ when $0 < u < \ov{u} \wedge \wh{u}$. It is defined as (see (3) of \cite{Szni19c}):
\begin{equation}\label{1.31}
\wh{u} = \sup\{ u \in [0,u_*) ; NLF(0,u) \; \mbox{holds}\}
\end{equation}
where for $0 \le v < u_*$, $NLF(0,v)$, i.e.~the {\it no large finite cluster property} on $[0,v]$, is defined as
\begin{equation}\label{1.32}
\begin{array}{l}
\mbox{there exists $L(v) \ge 1$, $c(v) > 0$, $\gamma(v) \in (0,1]$ such that (with $S(0,L) = \partial_i B(0,L)$),}\\
\mbox{for all $L \ge L(v)$ and $0 \le w \le v$, $\IP [0 \stackrel{w}{\longleftrightarrow} S(0,L), \; 0\stackrel{w}{\longleftrightarrow} \hspace{-2.6ex}\mbox{\scriptsize $/$} \;\; \;\infty] \le e^{-c(v) L^{\gamma(v)}}$}.
\end{array}
\end{equation}
One knows by Corollary 1.2 of \cite{DrewRathSapo14a} that $\wh{u} > 0$, and by Theorem 1 of \cite{Szni19c} that 
\begin{equation}\label{1.33}
\mbox{$\theta_0$ is $C^1$ and has positive derivative on $[0,\wh{u})$}.
\end{equation}
It is plausible but presently open that the equalities $\ov{u} = \wh{u} = u_*$ hold.

\medskip
We now introduce further boxes related to the length scale $L_0$, which take part in the definition of the important random set $\cU_1$ defined in (\ref{1.40}) below and in the proof of Theorem \ref{theo2.1} in the next section. Throughout the integer $K$ implicitly satisfies
\begin{equation}\label{1.34}
K \ge 100 .
\end{equation}
In the spirit of (2.9), (2.10) of \cite{Szni17}, we consider the boxes
\begin{equation}\label{1.35}
\begin{split}
B_{0,z} & = z + [0,L_0)^d \subseteq \wt{B}_{0,z} = z + [-L_0, 2 L_0)^d \subseteq D_{0,z} = z + [-3 L_0, 4 L_0)^d
\\
& \subseteq U_{0,z} = z + [-KL_0 + 1, K L_0 - 1)^d, \; \mbox{with} \; z \in \IL_0 (= L_0 \IZ^d).
\end{split}
\end{equation}
Given a  box $B_0$ as above and the corresponding $D_0$, we denote by $Z^{D_0}_\ell$, $\ell \ge 1$, the successive excursions in the interlacements that go from $D_0$ to $\partial U_0$, see (1.41) of \cite{Szni17}. We then denote by (see also (2.14) and (1.42) of \cite{Szni17}):
\begin{equation}\label{1.36}
\begin{split}
N_v (D_0)  = &\; \mbox{the number of excursions from $D_0$ to $\partial U_0$ in the interlacement}
\\
&\; \mbox{trajectories with level at most $v$, for $v \ge 0$}.
\end{split}
\end{equation}
The notion of $(\alpha, \beta, \gamma)$-good boxes that we now recall is an important ingredient in the definition of the random set $\cU_1$, see (\ref{1.40}) below. We consider
\begin{equation}\label{1.37}
\alpha > \beta > \gamma \; \mbox{in} \; (0,\ov{u})
\end{equation}

\n
(eventually we will choose them close to $\ov{u}$, see (4.8) in Section 4).

\medskip
Given an $L_0$-box $B_0$ and the corresponding $D_0$ (and likewise $D'_0$ corresponding to $B'_0$ below), see (\ref{1.35}), we say that $B_0$ is an $(\alpha, \beta,\gamma)${\it -good box} (see (2.11) - (2.13) of \cite{Szni17}) if:
\begin{equation}\label{1.38}
\left\{ \begin{array}{rl}
{\rm i)} & \mbox{$B_0 \backslash ({\rm range} \, Z_1^{D_0} \cup \dots \cup {\rm range} \, Z^{D_0}_{\alpha \,{\rm cap}(D_0)})$ contains a connected set with}
\\
& \mbox{diameter at least $\frac{L_0}{10}$ (and the set in parenthesis is empty if $\alpha \,{\rm cap}(D_0) < 1$)},
\\[2ex]
{\rm ii)} & \mbox{for any neighboring $L_0$-box $B'_0 = L_0 \,e + B_0$ with $|e| = 1$, any two connected}
\\
&\mbox{sets with diameter at least $\frac{L_0}{10}$ in $B_0 \backslash ({\rm range} \, Z_1^{D_0} \cup \dots \cup {\rm range} \, Z^{D_0}_{\alpha \,{\rm cap}(D_0)})$ and}
\\
& \mbox{$B'_0 \backslash ({\rm range} \, Z_1^{D'_0} \cup \dots \cup Z^{D'_0}_{\alpha \,{\rm cap}(D'_0)})$ are connected in}
\\
& \mbox{$D_0 \backslash ({\rm range} \, Z_1^{D_0} \cup \dots \cup  Z^{D_0}_{\beta \,{\rm cap}(D_0)})$ (with a similar convention as in i))}
\\[2ex]
{\rm iii)} & \dsl_{1 \le \ell \le \beta\,{\rm cap} (D_0)} \dis\int^{T_{U_0}}_0 \ov{e}_{D_0} \big(Z^{D_0}_\ell (s)\big)\,ds \ge \gamma \; \mbox{(with $T_{U_0}$ the exit time of $U_0$)},
\end{array}\right.
\end{equation}
and otherwise we say that $B_0$ is $(\alpha, \beta,\gamma)${\it -bad}.

\medskip
We now fix a level $u$ as in (\ref{0.1}), that is
\begin{equation}\label{1.39}
0 < u < \ov{u}
\end{equation}

\n
and following (4.27) of \cite{NitzSzni} (or (3.8) of \cite{Szni19b}), we introduce the random set $\cU_1$ as
\begin{equation}\label{1.40}
\begin{split}
\cU_1 = & \;\mbox{the union of $L_0$-boxes $B_0$ that are either contained in $([-3N, 3N]^d)^c$ or linked} 
\\
&\;\mbox{to an $L_0$-box contained in $([-3N, 3N]^d)^c$ by a path of $L_0$-boxes $B_{0,z_i}$, $0 \le i \le n$, }
\\
&\;\mbox{which are all except maybe for the last one $(\alpha,\beta,\gamma)$-good and such that}
\\
&\;N_u (D_{0,z_i}) < \beta \,{\rm cap} (D_{0,z_i}).
\end{split}
\end{equation}

In addition, as shown in Lemma 6.1 of \cite{Szni17}, one has the following connectivity property:
\begin{equation}\label{1.41}
\left\{ \begin{array}{l}
\mbox{if $B_{0,z_i}$, $0 \le i \le n$, is a sequence of neighboring $L_0$-boxes which are}
\\
\mbox{$(\alpha,\beta,\gamma)$-good, and $N_u (D_{0,z_i}) < \beta \,{\rm cap}(D_{0,z_i})$, for $0 \le i \le n$, then, for}
\\
\mbox{any connected set in $B_{0,z_0} \backslash ({\rm range} \,Z_1^{D_{0,z_0}} \cup \dots \cup {\rm range} \,Z^{D_{0,z_0}}_{\alpha \, {\rm cap}(D_{0,z_0})}$) with}
\\
\mbox{diameter at least $\frac{L_0}{10}$, there is a path starting in this set, contained in}
\\
\mbox{$\big(\bigcup\limits_{0 \le i \le n} D_{0,z_i}\big) \cap \cV^u$, and ending in $B_{0,z_n}$}.
\end{array}\right.
\end{equation}

\n
Thus, in view of (\ref{1.40}) and (\ref{1.41}), the random set $\cU_1$ provides paths in $\cV^u$ going from any $B_0$-box in $\cU_1 \cap D_N$ to $([-3N + L_0, 3N - L_0]^d)^c$ (and such paths go through $S_{2N}$).

\medskip
We will use the notation $\partial_{B_0} \,\cU_1$ to refer to the (random) collection of $B_0$-boxes that are not contained in $\cU_1$ but are neighbor of a $B_0$-box in $\cU_1$.

\medskip
We will also need  a statement quantifying the rarity of $(\alpha, \beta, \gamma)$-bad $B_0$-boxes.

\begin{lemma}\label{lem1.4}
Given $K \ge c_2 (\alpha, \beta, \gamma)$, there exists a non-negative function $\rho(L)$ depending on $\alpha, \beta, \gamma, K$, satisfying $\lim_L \rho (L) = 0$, such that
\begin{equation}\label{1.42}
\begin{array}{l}
\mbox{$\lim\limits_N \; \dis\frac{1}{N^{d-2}} \; \log \IP [\cB_N] = - \infty$, where $\cB_N$ stands for the event}
\\[2ex]
\cB_N = \{\mbox{there are more than $\rho (L_0) N^{d-2} \; (\alpha, \beta, \gamma)$-bad $B_0$-boxes}
\\
\qquad \quad \; \mbox{intersecting $[-3N, 3N]^d\}$}.
\end{array}
\end{equation}
\end{lemma}

\begin{proof}
The constant $c_2(\alpha, \beta, \gamma)$ corresponds to $c_8 (\alpha, \beta,\gamma)$ above (5.5) of \cite{Szni17}. We only sketch the proof, which is similar to that of Theorem 5.1 in the same reference, see also Proposition 3.1 of \cite{Szni19d}. It revolves around a stochastic domination argument: for finite collections of $L_0$-boxes with base points at mutual distance at least $\ov{K} L_0$, the indicator functions of the events that the boxes are $(\alpha, \beta,\gamma)$-bad are stochastically dominated by i.i.d. Bernoulli variables with success probability $\eta (L_0)$, for a function $\eta(L)$ depending on $\alpha, \beta, \gamma, K$ such that $\lim_L \frac{1}{\log L} \log \eta(L) = - \infty$. One sets $\rho(L) = \sqrt{\frac{\log L}{|\log \eta (L)|}}$, and considers for fixed $\tau \in \{0,\dots, \ov{K}-1\}^d$ the $L_0$-boxes $B_{0,z}$ with $z \in L_0 \tau + \ov{K} \IL_0$ that intersect $[-3N,3N]^d$. Setting $m = (\frac{8 N}{\ov{K} L_0})^d$ (an upper bound on the number of such boxes when $N$ is large) and $\wt{\rho} = \rho (L_0) N^{d-2} / (\ov{K}^d m)$, one has $\log \frac{\wt{\rho}}{\eta} \sim \log \frac{1}{\eta(L_0)}$, as $N \r \infty$, see for instance (3.16) of \cite{Szni19d}, so that (writing $\eta$ for $\eta(L_0)$ and  $\rho$ for $\rho(L_0)$):
\begin{equation*}
\begin{split}
m \Big\{\wt{\rho} \log \dis\frac{\wt{\rho}}{\eta} + ( 1- \wt{\rho}) \; \log \Big(\dis\frac{1 - \wt{\rho}}{1 - \eta}\Big) \Big\} & \sim m \wt{\rho} \log \mbox{\f $\dis\frac{1}{\eta}$} = \rho \ov{K}^{-d} N^{d-2} \log \mbox{\f $\dis\frac{1}{\eta}$}
\\
& = \sqrt{\log L_0 \log \mbox{\f $\dis\frac{1}{\eta}$}} \; \ov{K}^{-d} \,N^{d-2} \gg N^{d-2}, \; \mbox{as} \; N \r \infty.
\end{split}
\end{equation*}

\n
The claim (\ref{1.42}) then follows from the usual exponential bounds on sums of i.i.d.~Bernoulli variables. 
\end{proof}

\medskip
We now turn to the length scale $L_1$, see (\ref{1.8}). In addition to the boxes $B_{1,z}, z \in \IL_1$ in (\ref{1.10}), we consider the boxes (with the same $K$ as in (\ref{1.34}))
\begin{equation}\label{1.43}
U_{1,z} = z + [-K L_1 + 1, KL_1 -1)^d (\supseteq B_{1,z} = z + [0,L_1)^d) \; \mbox{for} \; z \in \IL_1 = L_1 \,\IZ^d.
\end{equation}
Similarly to above (\ref{1.36}), given a box $B_1$ and the corresponding $U_1$, we denote by $Z^{B_1}_\ell, \ell \ge 1$, the successive excursions in the interlacements that go from $B_1$ to $\partial U_1$, and also use the notation
\begin{equation}\label{1.44}
\begin{split}
N_v (B_1) = &\; \mbox{the number of excursions from $B_1$ to $\partial U_1$ in the interlacement} 
\\
& \; \mbox{trajectories with level at most $v$, for $v \ge 0$}.
\end{split}
\end{equation}
The quantity
\begin{equation}\label{1.45}
u_{B_1} = N_u (B_1) / {\rm cap}(B_1)
\end{equation}
plays the role of the local level (or the ``undertow'') of the interlacements (at the level $u$ chosen in (\ref{1.39})) in the box $B_1$.

\medskip
In essence, we will perform local averaging operations in $B_1$-boxes that will only retain the information contained in $u_{B_1}$, see (\ref{2.8}) and Theorem \ref{theo2.1}.

\medskip
We then proceed with the definition of the bubble set. First, given a box $B_1$, we denote by ${\rm Deep} \, B_1$ the set
\begin{equation}\label{1.46}
{\rm Deep} \, B_1 = \bigcup\limits_{z \in \IL_0, D_{0,z} \subseteq B_1} B_{0,z}
\end{equation}
obtained in essence by ``peeling off'' a shell of depth $3L_0$ from the surface of $B_1$, thus only keeping the $B_0$-boxes such that the corresponding $D_0$ is contained in $B_1$. One then defines the {\it bubble set}
\begin{equation}\label{1.47}
B u b = \bigcup\limits_{B_1 \subseteq D_N, \,\cU_1 \cap {\rm Deep} \, B_1 = \phi} B_1
\end{equation}
that is the union of the $B_1$-boxes contained in $D_N$ such that $\cU_1$ does not reach ${\rm Deep} \, B_1$, see Figure 2.

\medskip
We will perform local averaging in boxes $B_1$ outside the bubble set in the next section. But an important challenge will then be to ascribe a cost to a bubble set of non-negligible volume. The coarse grained random set $C_\omega$ constructed in Theorem \ref{theo3.1} will provide the required tool.

\medskip
As a last piece of notation, we write
\begin{equation}\label{1.48}
\mbox{$L^v_x, x \in \IZ^d$, for the {\it field of occupation times at level} $v \ge 0$}
\end{equation}
that records the total time spent at sites of $\IZ^d$ by trajectories with level at most $v$ in the interlacements. It will come up in Sections 2 and 4.

\section{Local averaging: departing from the microscopic picture}

The main object of this section is the proof of Theorem \ref{theo2.1}. It shows that we can replace the excess event $\cA_N = \{|D_N \backslash \cC^u_{2N} | \ge \nu \,|D_N|\}$ of (\ref{0.7}) with an event $\cA'_N$, in our quest for an upper bound on the exponential rate of decay of $\IP[\cA_N]$. This event $\cA'_N$ is solely expressed in terms of the volume of the bubble set and of the local levels $u_{B_1}$ of the $B_1$-boxes that lie in the complement of the bubble set in $D_N$.

\medskip
As in (\ref{0.1}), see also (\ref{0.6}), we assume that
\begin{align}
&0 < u < \ov{u}, \;\;\mbox{and}\label{2.1}
\\[1ex]
&\theta_0(u) < \nu < 1 \,. \label{2.2}
\end{align}
We also pick
\begin{align}
&\alpha > \beta > \gamma \;\;\mbox{in $(u,\ov{u})$, and}\label{2.3}
\\[1ex]
&K \ge 100 \,. \label{2.4}
\end{align}

\medskip\n
The parameters $\alpha, \beta, \gamma, u, K$ (and $N$) enter the definition of the random sets $\cU_1$ in (\ref{1.40}), which is a union of $B_0$-boxes, and of the bubble set $B u b$ in (\ref{1.47}), which is a union of $B_1$-boxes.

\medskip
We introduce an additional parameter $\ve$ in $(0,1)$ such that
\begin{equation}\label{2.5}
\nu > 10^3 \ve + \theta_0 (u).
\end{equation}

\n
We also choose a finite grid of values for the local levels $u_{B_1}$ (see (\ref{1.45})), namely, we consider a set $\Sigma^0 (\gamma, u, \ve)$ determined by $d,\gamma, u, \ve$ such that
\begin{equation}\label{2.6}
\left\{ \begin{array}{l}
\mbox{$\Sigma^0 \subseteq (0, \gamma]$ is finite, contains $u$ and $\gamma$, and is such that between consecutive points}
\\
\mbox{of $\{0\} \cup \Sigma^0$ the functions $\theta_0(\cdot)$ and $\sqrt{\cdot}$ vary at most by $10^{-3} \ve$}
\end{array}\right.
\end{equation}
(with $\theta_0(\cdot)$ as in (\ref{0.2})).

\medskip
In the course of the proof of Theorem \ref{theo2.1} below, we will add further points to this grid. For the time being we record some notation. We write
\begin{equation}\label{2.7}
\mbox{$\gamma_- \in \Sigma^0$ for the largest element of $\Sigma^0$ smaller than $\gamma$},
\end{equation}
and given an $L_1$-box $B_1$ with local level $u_{B_1}$, we write
\begin{equation}\label{2.8}
\left\{ \begin{split}
\lambda^-_{B_1}  = &\; \mbox{the largest element of $\{0\} \cup \Sigma^0$ smaller or equal to $u_{B_1}$},
\\
\lambda^+_{B_1}  = &\; \mbox{the smallest element of $\Sigma^0$ bigger than $u_{B_1}$, if $u_{B_1} < \gamma$, and $\gamma$ otherwise.}
\end{split}\right.
\end{equation}
Thus, when $u_{B_1} < \gamma$, we have $\lambda^-_{B_1} \le u_{B_1} < \lambda^+_{B_1}$ and $(\lambda^-_{B_1}, \lambda^+_{B_1}) \cap \Sigma^0 = \phi$. We will use $\lambda^-_{B_1}$ as a discretization of the local level $u_{B_1}$ of the box $B_1$. Further, we denote by $\cC_0$ and $\cC_1$ the respective subsets of $\IL_0$ and $\IL_1$, see (\ref{1.9}), (\ref{1.10}):
\begin{equation}\label{2.9}
\cC_0 = \{z \in \IL_0; B_{0,z} \subseteq D_N\}, \; \cC_1 = \{z \in \IL_1; B_{1,z} \subseteq D_N\},
\end{equation}
and routinely write $B_0 \in \cC_0$ to mean $B_{0,z}$ with $z \in \cC_0$ and $B_1 \in \cC_1$ to mean $B_{1,z}$ with $z \in \cC_1$.

\medskip
Here is the main object of this section. We recall (\ref{0.2}), (\ref{0.7}), (\ref{1.47}) for notation.

\begin{theorem}\label{theo2.1}
Given $u,\nu$ as in (\ref{2.1}), (\ref{2.2}), $\alpha >\beta > \gamma$ as in (\ref{2.3}), and $\ve$ as in (\ref{2.5}), there exists $c_3 (\alpha, \beta, \gamma, u, \ve)$ such that for $K \ge c_3$
\begin{equation}\label{2.10}
\limsup\limits_N \; \dis\frac{1}{N^{d-2}} \; \log \IP[\cA_N] \le \limsup\limits_N \; \dis\frac{1}{N^{d-2}} \; \log \IP [ \cA'_N],
\end{equation}
where $\cA'_N$ stands for the event
\begin{equation}\label{2.11}
\cA'_N = \{| B u b| + \dsl_{B_1 \subseteq D_N \backslash B u b} \wt{\theta} (\lambda^-_{B_1}) \,| B_1 | \ge (\nu - 6 \ve) \,|D_N|\},
\end{equation}
with $\wt{\theta}(v) = \theta_0(v) \, 1\{v < \gamma^-\} + 1\{v \ge \gamma^-\}$, for $v \ge 0$.
\end{theorem}

One should note that in contrast to the original excess event $\cA_N$ of (\ref{0.7}) that involves the microscopic information stating for each $x$ in $D_N$ whether $x$ can be linked by a path in $\cV^u$ to $S_{2N}$ or not, the event $\cA'_N$ is expressed in terms of the volume of the bubble set (which relies on $\cU_1$) and the discretizations $\lambda^-_{B_1}$ of the local levels $u_{B_1}$ for $B_1 \in \cC_1$ outside the bubble set. In the proof of Theorem \ref{theo2.1} the heart of the matter will correspond to the treatment of the boxes $B_1$ outside the bubble set and such that $u_{B_1} < \gamma_-$.

\medskip\n
{\it Proof of Theorem \ref{theo2.1}:}
Recall that $\cA_N$ stands for the event $\{| D_N \backslash \cC^u_{2N}| \ge \nu |D_N|\}$, see (\ref{0.7}). On the other hand, the left member of the inequality in the definition of $\cA'_N$ is
\begin{equation}\label{2.12}
|B u b| + \dsl_{B_1 \subseteq D_N \backslash B u b} \wt{\theta} (\lambda^-_{B_1}) |B_1| = \dsl_{u_{B_1} \ge \gamma_- \;{\rm or} \; B_1 \subseteq B u b} |B_1| + \dsl_{u_{B_1} < \gamma_- \;{\rm and} \; B_1 \cap B u b = \phi} \theta_0(\lambda^-_{B_1}) |B_1|
\end{equation}

\n
where the sum runs over $B_1$ in $\cC_1$, see (\ref{2.9}), in the last two sums.

\medskip
The first step of the proof of Theorem \ref{theo2.1} will in essence bound $|D_N \backslash \cC^u_{2N}|$ from above by three terms, see (\ref{2.17}) - (\ref{2.18}) below. The first and the second term will eventually lead, up to corrections, to the first and second sum in the right member of (\ref{2.12}), and the third term will be negligible.

\medskip
As a preparation, we first introduce an additional length scale. The function $\theta_0(\cdot)$ from (\ref{0.2}) is continuous on $[0, \gamma]$, and as $L$ tends to infinity the continuous functions \hbox{$\theta_{0,L}(v) = \IP [0 \vv \;\; S(0,L)]$,} $v \in [0,\gamma]$ (with $S(0,L)$ as in (\ref{1.32})), are non-decreasing in $L$, and converge uniformly to $\theta_0(\cdot)$ on $[0,\gamma]$ (by Dini's lemma). We can thus find $R(\gamma, u, \ve)$ such that
\begin{equation}\label{2.13}
| \Sigma^0 |^4 \sup\limits_{v \in [0,\gamma]} \IP[ 0 \stackrel{v}{\longleftrightarrow} S(0,R) \; \mbox{and} \; 0 \vv \;\; \infty] \le 10^{-5} \, \ve^2 .
\end{equation}

\n
We also need to refine the finite grid $\Sigma^0 (\subseteq (0,\gamma])$: between any two consecutive points of $\Sigma^0$ we introduce $8$ equally spaced points. We denote by $\Sigma$ the enlarged grid. It is determined by $\Sigma^0$ and hence by $d, \gamma, u, \ve$. In addition, we have $|\Sigma | \le 9 \, |\Sigma^0|$. We will also routinely use the following notation: when $B_1$ is such that $u_{B_1} < \gamma_-$, so that $\lambda^-_{B_1} \le u_{B_1} < \lambda^+_{B_1} \le \gamma_-$, see (\ref{2.8}), we will write
\begin{equation}\label{2.14}
(\lambda^+_{B_1}) < \overset{\vee}{\lambda}\,\!{^+_{B_1}} < \overset{\vee}{\lambda}\,\!{^{++}_{B_1}}  < \lambda^{++}_{B_1} < \wh{\lambda}^{++}_{B_1} < \wt{\lambda}^{++}_{B_1} < \wt{\lambda}^{+++}_{B_1} <  \wh{\lambda}_{B_1}^{+++} < \lambda^{+++}_{B_1} \; (< \gamma)
\end{equation}
for the $8$ inserted values right above $\lambda^+_{B_1}$.

\medskip
Sometimes we will also consider some generic element of $\Sigma^0 \cap (0, \gamma_-]$, denoted by $\lambda^+$ and write $\overset{\vee}{\lambda}\,\!{^+} < \overset{\vee}{\lambda}\,^{++} < \lambda^{++}  < \wh{\lambda}^{++}< \wt{\lambda}^{++} < \wt{\lambda}^{+++} < \wh{\lambda}^{+++} < \lambda^{+++}$ for the $8$ inserted points of $\Sigma$ right above $\lambda^+$. These inserted values will be used to perform several sprinkling operations, and to define the basic splitting in (\ref{2.17}), (\ref{2.18}) below.

\medskip
One last preliminary remark is that setting for each $L_1$-box $B_1$
\begin{equation}\label{2.15}
B_{1,R} = \{x \in B_1; B(x, R) \subseteq B_1\},
\end{equation}
for large $N$ the boxes $B_{1,R}$, with $B_1 \in \cC_1$ carry the ``principal volume'' of $D_N$ in the sense that
\begin{equation}\label{2.16}
D_N \supseteq \bigcup\limits_{\cC_1} B_1 \supseteq \bigcup\limits_{\cC_1} B_{1,R} \; \mbox{and} \; \big| \bigcup\limits_{\cC_1} B_{1,R} \big| \, / \, |D_N| \underset{N}{\longrightarrow} 1.
\end{equation}
Our first main step corresponds to the splitting stated below. By (\ref{2.16}) and the definition $\cA_N = \{|D_N \backslash \cC^u_{2N}| \ge \nu \,|D_N|\}$ of the excess event, we see that
\begin{equation}\label{2.17}
\mbox{for large $N$ on $\cA_N$, $(\nu-\ve) \,|D_N| \le \dsl_{B_1 \in \cC_1} \; \dsl_{x \in B_{1, R}} 1\{x \notin \cC^u_{2N}\} \le I + I\!I_R + I\!I\!I_R$}
\end{equation}

\n
where we have set (with a similar convention as below (\ref{2.12})):
\begin{equation}\label{2.18}
\left\{ \begin{array}{rl}
{\rm i)} & I = \dsl_{u_{B_1} \ge \gamma_- \;{\rm or} \; B_1 \subseteq B u b} |B_1|,
\\[4ex]
{\rm ii)} & I\!I_R = \dsl_{u_{B_1} < \gamma_- \;{\rm and} \; B_1 \cap B u b = \phi}\; \dsl_{x \in B_{1,R}} 1 \{x \lr \;\;\; S(x,R): B_1, \lambda^{+++}_{B_1}\}
\\
& \mbox{where the indicator function refers to the event stating that $x$ is not}
\\
& \mbox{connected to $S(x,R)$ in $B_1 \backslash ({\rm range} \,Z_1^{B_1} \cup \dots \cup {\rm range} \, Z^{B_1}_{\lambda^{+++}_{B_1} \,{\rm cap}(B_1)})$}
\\[3ex]
{\rm iii)} & I\!I\!I_R  = \dsl_{u_{B_1}< \gamma_- \;{\rm and } \; B_1 \cap  B u b = \phi} \;\dsl_{x \in B_{1,R}} 1 \{x \longleftrightarrow  S(x,R): B_1, \lambda_{B_1}^{+++} \; \mbox{and} \; x \uu \;\;S_{2N}\} 
\\
&\mbox{with a similar notation as in ii)}
\end{array}\right.
\end{equation}
(we recall that for $x \in B_{1,R} \subseteq D_N$, $x \notin \cC^u_{2N}$ coincides with $x \uu \;\; S_{2N})$.

\medskip
The term $I$ in i) above coincides with the first sum in the right member of (\ref{2.12}). To prove Theorem \ref{theo2.1} we will show that
\begin{align}
& \mbox{for} \; K \ge c (\gamma, u, \ve), \label{2.19}
\\
&\lim\limits_N \; \dis\frac{1}{N^{d-2}} \; \log \,\IP[ I\!I_R > \dsl_{u_{B_1} < \gamma_- \;{\rm and} \; B_1 \cap B u b = \phi} \theta_0(\lambda^-_{B_1}) \,|B_1| + 2 \ve \,|D_N|] = - \infty, \nonumber
\intertext{and that}
& \mbox{for} \; K \ge c (\alpha, \beta, \gamma, u, \ve), \;\lim\limits_N \; \dis\frac{1}{N^{d-2}} \; \log \,\IP[I\!I\!I_R \ge 3 \ve \, |D_N|] = - \infty. \label{2.20}
\end{align}

\n
In view of (\ref{2.12}) and (\ref{2.17}), the claim (\ref{2.10}) of Theorem \ref{theo2.1} will then follow. In essence, the second sum in the right member of (\ref{2.12}) mainly dominates $I\!I_R$ and $I\!I\!I_R$ is a negligible quantity.

\medskip
The proof of (\ref{2.19}) will mostly be an application of the results of Section 3 of \cite{Szni19d} to the example (2.6) of that same reference, combined with a comparison between $\theta_{0,R}$, see above (\ref{2.13}), and $\theta_0$ to control the variation of $\theta_{0,R}$.

\medskip
The proof of (\ref{2.20}) will be more delicate and will revolve around the intuitive idea that for most $B_1$-boxes in $\cC_1$, if their local level $u_{B_1}$ lies below $\gamma_-$ and ${\rm Deep} \,B_1$ meets $\cU_1$ (i.e. $B_1$ is not in the bubble set, see (\ref{1.47})), then most points of $B_1$ that make it to distance $R$ avoiding the first $\lambda_{B_1}^{+++} {\rm cap}(B_1)$ excursions $Z^{B_1}_\ell$, $\ell \ge 1$, are actually connected to $S_{2N}$ in $\cV^u$.

\medskip
Thus, there remains to prove (\ref{2.19}) and (\ref{2.20}). We begin by the proof of (\ref{2.19}). We first compare the variation of $\theta_{0,R}$ between two points of $[0,\gamma]$ with that of $\theta_0$. By (43) of \cite{Szni19c}, one knows that for $0 \le v < v'$
\begin{align}
&\theta_0(v') - \theta_0(v) = \theta_{0,R}(v') - \theta_{0,R}(v) + \IP [ 0  \stackrel{v'}{\longleftrightarrow} S(0,R), 0\vvv \;\; \infty] - \IP[0 \stackrel{v}{\longleftrightarrow} S(0,R), 0 \vv \;\; \infty] \label{2.21}
\intertext{and by (\ref{2.13}) (and $|\Sigma^0 | \ge 2$) we have}
&\big|\theta_0(v') - \theta_0(v)  - \big(\theta_{0,R} (v') - \theta_{0,R}(v)\big)\big| \le 10^{-5} \ve, \; \mbox{for any $v < v'$ in $[0, \gamma]$}. \label{2.22}
\end{align}

\n
We can now apply the results of Section 3 of \cite{Szni19d} where we choose the local function $F$ as in (2.6) of \cite{Szni19d} (i.e.~for any $\ell \in [0,\infty)^{B(0,R)}$, $F(\ell) = 1 \{$any path from $0$ to $S(0,R)$ in $B(0,R)$ meets a $y$ with $\ell_y > 0\}$ and the corresponding $\theta(v) \stackrel{\rm def}{=} \IE [F((L^v_y)_{|y|_\infty \le R})] = \theta_{0,R}(v)$, see (\ref{1.48}) for notation). We select $\kappa(\gamma, u, \ve)$ and $\mu(\ve)$ so that with $\Sigma$ the refinement of $\Sigma^0$ introduced below (\ref{2.13})
\begin{equation}\label{2.23}
(1 + \kappa) \, \lambda < (1 - \kappa) \,\lambda', \; \mbox{for all $\lambda < \lambda '$ in $\Sigma$, with $0 < \kappa < \ve$, and $\mu = 10^{-3} \ve$}.
\end{equation}
The so-called $(\Sigma, \kappa, \mu)${\it -good boxes} of (2.76) of \cite{Szni19d} allow to perform local-averaging. In particular, when $B_1$ is a $(\Sigma, \kappa, \mu)$-good box, then for any $\lambda \in \Sigma$, one has
\begin{equation*}
\dsl_{x \in B_{1,R}}  1\{x \lr \;\;S(x,R): B_1, \lambda\} \le \big(\theta_{0,R} \big((1 + \kappa) \lambda\big) + \mu\big) \,|B_1|
\end{equation*}
where the indicator function in the left member refers to the event stating that $x$ is not connected to $S(x,R)$ in $B_1 \backslash ({\rm range} \,Z^{B_1}_1 \cup \dots \cup {\rm range} \,Z^{B_1}_{\lambda \,{\rm cap} (B_1)})$. Thus, when $B_1$ is a $(\Sigma, \kappa, \mu)$-good box, for any consecutive $\lambda^- < \lambda^+ < \lambda'$ in $\{0\} \cup \Sigma^0$, one has (see below (\ref{2.14}) for notation)
\begin{equation}\label{2.24}
\begin{split}
\dsl_{x \in B_{1,R}} 1\{x \lr \;\; S(x,R): B_1, \lambda^{+++}\}    \le &\; \big(\theta_{0,R} \big((1 + \kappa) \lambda^{+++}\big) + \mu\big) \,|B_1| \;
\\[-1ex]
\mbox{(and since $(1 + \kappa) \lambda^{+++} < \lambda')$}\qquad &\hspace{-3.5ex} \stackrel{(\ref{2.23})}{\le}  \big(\theta_{0,R}(\lambda') + 10^{-3} \ve) \,|B_1| 
\\
 = &\; \big(\theta_{0,R}(\lambda^-) + \theta_{0,R} (\lambda') - \theta_{0,R}(\lambda^-) + 10^{-3}\ve) |B_1|
 \\
 &\hspace{-3.5ex} \stackrel{(\ref{2.22})}{\le}  \big(\theta_{0,R}(\lambda^-) + \theta_0(\lambda') - \theta_0(\lambda^-) + 2\,10^{-3}\ve ) |B_1|
 \\
 &\hspace{-4ex} \underset{\theta_{0,R} \le \theta_0}{\stackrel{(\ref{2.6})}{\le}}  \big(\theta_0 (\lambda^-) +4\, 10^{-3} \ve) |B_1|.
\end{split}
\end{equation}

\medskip\n
On the other hand, by Proposition 3.1 and (4.9) of \cite{Szni19d}, when $K$ is large enough, most $B_1$-boxes in $D_N$ are $(\Sigma, \kappa, \mu)$-good, in the sense that
\begin{equation}\label{2.25}
\mbox{for $K \ge c(\gamma, u, \ve), \;\dis\frac{1}{N^{d-2}} \; \log  \IP \big[ \dsl_{B_1 \in \cC_1} |B_1| \,1\{B_1$ is $(\Sigma, \kappa, \mu)$-bad$\} \ge \ve \,|D_N|\big] = - \infty$}.
\end{equation}
Hence, on the complement of the event under the above probability, using (\ref{2.24}) for $(\Sigma, \kappa, \mu)$-good boxes, we find that
\begin{equation}\label{2.26}
I\!I_R \underset{(\ref{2.24})}{\stackrel{(\ref{2.18}) \,{\rm ii)}}{\le}} \; \dsl_{u_{B_1} < \gamma_- \; {\rm and} \; B_1 \cap B u b = \phi} \theta_0(\lambda^-_{B_1}) + 2 \ve \, |D_N|,
\end{equation}
and this completes the proof of (\ref{2.19}).

\medskip
We now turn to the proof of (\ref{2.20}). The implementation of the rough strategy outlined below (\ref{2.20}) to prove that statement relies on the notion of $I\!I\!I_R$-good boxes that corresponds to four ``local properties'' satisfied by such boxes, see (\ref{2.27}) - (\ref{2.30}) below. We will show that when $K$ is sufficiently large, with overwhelming probability for large $N$, the $I\!I\!I_R$-bad boxes in $D_N$ occupy a negligible fraction of volume, see Lemma \ref{lem2.3}, and that the contribution of $I\!I\!I_R$-good $B_1$-boxes in $D_N$ to $I\!I\!I_R$ is small, see Lemma \ref{lem2.2} and (\ref{2.35}), (\ref{2.36}).

\medskip
Given a box $B_1$ here are the four conditions:
\begin{equation} \label{2.27}
\mbox{all $B_0 \subseteq {\rm Deep} \, B_1$ are $(\alpha, \beta, \gamma)$-good (see (\ref{1.46}), (\ref{1.38}) for notation),}
\end{equation}
\begin{equation}\label{2.28}
\left\{\begin{array}{l}
\mbox{for all $B_0 \subseteq {\rm Deep} \,B_1$ and $\lambda < \lambda'$ in $\Sigma$, the excursions $Z^{B_1}_\ell$, $1 \le \ell \le \lambda  \,{\rm cap} (B_1)$ }
\\
\mbox{contain in total strictly less than $\lambda' \,{\rm cap}(D_0)$ excursions from $D_0$ to $\partial U_0$},
\end{array}\right.
\end{equation}
\begin{equation}\label{2.29}
 \left\{\begin{array}{l}
\mbox{for all $B_0 \subseteq {\rm Deep} \,B_1$ and $\lambda < \lambda'$ in $\Sigma$, the excursions $Z^{B_1}_\ell$, $1 \le \ell \le \lambda '\,{\rm cap} (B_1)$} 
\\
\mbox{contain in total more than $\lambda \,{\rm cap}(D_0)$ excursions from $D_0$ to $\partial U_0$,  \qquad}
\end{array}\right.
\end{equation}
and with the notation (\ref{1.35}) and below (\ref{2.14})
\begin{equation}\label{2.30}
 \left\{\begin{array}{l}
 \mbox{for all $B_0 \subseteq {\rm Deep} \, B_1$ and all $\lambda^{++}$  of the grid $\Sigma$, if two connected sets in } 
\\
\mbox{$\wt{B}_0 \backslash ({\rm range} \,Z_1^{B_1} \cup \dots \cup {\rm range} \, Z^{B_1}_{\lambda^{++} \,{\rm cap}(B_1)})$ have diameter at least $\frac{L}{10}$,}
\\
\mbox{then they are connected in $D_0 \backslash ({\rm range} \, Z_1^{B_1} \cup \dots \cup {\rm range} \, Z^{B_1}_{\lambda^+ \,{\rm cap} (B_1)})$. }
\end{array}\right.
\end{equation}
We then say that
\begin{equation}\label{2.31}
\mbox{a box $B_1$ is $I\!I\!I_R$\,{\it -good} if (\ref{2.27})\!\,\,-\,\,\!(\ref{2.30}) are satisfied, and $I\!I\!I_R${\it -bad} otherwise.}
\end{equation}
With the help of the above notion we can bound $I\!I\!I_R$ in (\ref{2.18}) iii) as follows:
\begin{equation}\label{2.32}
\left\{\begin{array}{l}
I\!I\!I_R \le I\!I\!I_{1,R} + I\!I\!I_{2,R}, \; \mbox{where}
\\[2ex]
I\!I\!I_{1,R} =  \dsl_{B_1: I\!I\!I_R{\rm -good},u_{B_1} < \gamma_- \,{\rm and}\, B_1 \cap B u b = \phi} \;\; 
\\
\qquad \quad \dsl_{x \in B_{1,R}} 1 \{x \longleftrightarrow S(x,R): B_1, \lambda^{+++}_{B_1} \; {\rm and} \; x \uu \;\;S_{2N}\}
\\[4ex]
I\!I\!I_{2,R} = \dsl_{B_1 \in \cC_1} |B_1| \, \mbox{$1\{B_1$ is $I\!I\!I_R$-bad$\}$}.
\end{array}\right.
\end{equation}

\n
The last term will be handled in Lemma \ref{lem2.3} below. For the time being we focus on $I\!I\!I_{1,R}$. Our next goal is to show that
\begin{equation}\label{2.33}
\mbox{for} \; K \ge c(\gamma, u, \ve), \; \lim\limits_N \; \dis\frac{1}{N^{d-2}} \; \log \IP[ I\!I\!I_{1,R} \ge 2 \ve \,|D_N|] = - \infty.
\end{equation}

\n
With this goal in mind, we observe that when $B_1$ is $I\!I\!I_R$-good and $u_{B_1} < \gamma_-$ (such boxes enter the sum defining $I\!I\!I_{1,R}$ in (\ref{2.32})), then by (\ref{2.28}) for each $B_0 \subseteq {\rm Deep} \,B_1$, one has $N_u(D_0) < \beta \,{\rm cap} (D_0)$ (because the excursions at level at most $u$ from $D_0$ to $\partial U_0$ are part of the $N_u(B_1) = u_{B_1} \,{\rm cap}(B_1)$ first excursions at level at most $u$ from $B_1$ to $\partial U_1$). Moreover, by (\ref{2.27}) all $B_0 \subseteq {\rm Deep} \,B_1$ are $(\alpha, \beta, \gamma)$-good. Thus, in view of the definition of $\cU_1$ in (\ref{1.40}), we see that
\begin{equation}\label{2.34}
\mbox{when $B_1 \in \cC_1$ is $I\!I\!I_R$-good, $u_{B_1} < \gamma_-$, and $\cU_1 \cap {\rm Deep} \, B_1 \neq \phi$, then $\cU_1 \supseteq {\rm Deep}\, B_1$}.
\end{equation}
Since $|B_1 \backslash {\rm Deep} \,B_1| \le c \,L_0 \,L_1^{d-1}$, we see  that
\begin{equation}\label{2.35}
\left\{\begin{split}
I\!I\!I_{1,R} & \le c(L_0/L_1) \, |D_N| +  \wt{I\!I\!I}_{1,R} + \wh{I\!I\!I}_{1,R}, \; \mbox{where}
\\[2ex]
\wt{I\!I\!I}_{1,R} & =  \dsl_{B_1: I\!I\!I_R{\rm -good},u_{B_1} < \gamma_- ,B_1 \cap B u b = \phi} \;\; \dsl_{B_0 \subseteq {\rm Deep}\,B_1} 
\\
&\quad \dsl_{x \in B_0} \;1 \{x \longleftrightarrow S(x,R): B_1, \lambda^{+++}_{B_1} \; {\rm and}  \;x \lr \;\;\;S\big(x, \big[\mbox{\f $\dis\frac{L_0}{2}$}\big]\big): B_1, \lambda^{++}_{B_1}\},
\\[2ex]
\wh{I\!I\!I}_{1,R} & = \dsl_{B_1:  I\!I\!I_R{\rm -good},u_{B_1} < \gamma_- ,B_1 \cap B u b = \phi} \;\; \dsl_{B_0 \subseteq {\rm Deep}\,B_1} \;\;
\\[1ex]
&\quad \dsl_{x \in B_0} \;1 \big\{x \longleftrightarrow S\big(x, \big[\mbox{\f $\dis\frac{L_0}{2}$}\big]\big): B_1, \lambda^{++}_{B_1} \; \mbox{and}\;\;x \uu \;\;S_{2N}\big\}.
\end{split}\right.
\end{equation}

\medskip\n
Our next step towards the proof of (\ref{2.33}) is to show that
\begin{equation}\label{2.36}
\wh{I\!I\!I}_{1,R} = 0.
\end{equation}

\n
For this purpose, we observe that by (\ref{2.34}), for any $B_1$ entering the sum defining $\wh{I\!I\!I}_{1,R}$, all $B_0 \subseteq {\rm Deep} \, B_1$ are contained in $\cU_1$, so that by (\ref{1.38}) i) and (\ref{1.40}), (\ref{1.41}) there is a component (and actually all such components) in $B_0 \backslash ({\rm range} \, Z_1^{D_0} \cup \dots \cup {\rm range} \, Z^{D_0}_{\alpha \,{\rm cap} (D_0)}$) with diameter at least $\frac{L_0}{10}$, which is linked to $S_{2N}$ in $\cV^u$.

\medskip
Due to (\ref{2.28}) and $u_{B_1} < \gamma_-$, such a component in $B_0 \backslash ({\rm range} \,Z_1^{D_0}  \cup \dots \cup {\rm range}\, Z^{D_0}_{\alpha \,{\rm cap}(D_0)})$ is a connected set in $\wt{B}_0 \backslash ({\rm range} \, Z_1^{B_1} \cup \dots \cup {\rm range} \, Z^{B_1}_{\lambda^{++}_{B_1} \,{\rm cap} (B_1)})$. By (\ref{2.30}) any $x \in B_0$ that is linked to $S(x, [\frac{L_0}{2}])$ in $B_1 \backslash ({\rm range} \, Z_1^{B_1} \cup \dots \cup {\rm range} \, Z^{B_1}_{\lambda^{++}_{B_1} \,{\rm cap} (B_1)})$ (and hence in $\wt{B}_0 \backslash  ({\rm range} \, Z_1^{B_1} \cup \dots \cup {\rm range} \, Z^{B_1}_{\lambda^{++}_{B_1} \,{\rm cap} (B_1)})$ can be linked to the above mentioned connected set in $B_0 \backslash ({\rm range} \, Z_1^{D_0} \cup \dots \cup {\rm range} \, Z^{D_0}_{\alpha \,{\rm cap} (D_0)})$ via a path in $\wt{B}_0 \backslash ({\rm range} \, Z_1^{B_1} \cup \dots \cup {\rm range} \, Z^{B_1}_{\lambda^+_{B_1} \,{\rm cap} (B_1)})$. Since $u_{B_1} < \lambda^+_{B_1}$, the above connected set in $B_0 \backslash ({\rm range} \, Z_1^{D_0} \cup \dots \cup {\rm range} \, Z^{D_0}_{\alpha \,{\rm cap} (D_0)})$ (as we already know) and the path from $x$ to this connected set are contained in $\cV^u$, so that $x \in \cC^u_{2N}$. This proves (\ref{2.36}).

\bigskip
To complete the proof of (\ref{2.33}) there remains to bound $\wt{I\!I\!I}_{1,R}$ in (\ref{2.35}). This is the objective of the next
\begin{lemma}\label{lem2.2}
\begin{equation}\label{2.37}
\mbox{For $K \ge c(\gamma, u, \ve)$, one has $\lim\limits_N \; \dis\frac{1}{N^{d-2}} \log \IP [\wt{I\!I\!I}_{1,R} \ge \mbox{\f $\dis\frac{7}{4}$} \; \ve \,|D_N|] = - \infty$}.
\end{equation}
\end{lemma}

\begin{proof}
As a first reduction we will bound $\wt{I\!I\!I}_{1,R}$ by a sum of identically distributed variables $Y_{B_0}$, see (\ref{2.39}) below, where $B_0$ ranges over $\cC_0$, see (\ref{2.9}), and each $Y_{B_0}$ solely involves the excursions $Z_\ell^{D_0}$, $\ell \ge 1$. Then we will use the soft local time technique of \cite{PopoTeix15} in the version of \cite{ComeGallPopoVach13} to bring into play stochastic domination by independent variables.

\medskip
We first note that by (\ref{2.28}), (\ref{2.29}) (recall from (\ref{2.14}) that $\lambda_{B_1}^{++} < \wh{\lambda}_{B_1}^{++} <  \wh{\lambda}_{B_1}^{+++} < \lambda_{B_1}^{+++}$ when $u_{B_1} < \gamma_-$):
\begin{equation}\label{2.38}
\begin{array}{l}
\wt{I\!I\!I}_{1,R}  \le \dsl_{B_1: I\!I\!I_R{\rm -good}, u_{B_1} < \gamma_- , B_1 \cap B u b = \phi} \; \dsl_{B_0 \subseteq {\rm Deep} \, B_1} \;\; \dsl_{x \in B_0} 1 \{x \longleftrightarrow S(x,R): D_0, \wh{\lambda}_{B_1}^{+++}  
\\
\mbox{and} \; x \lr \;\; \;S\big(x,\big[\mbox{\f $\dis\frac{L_0}{2}$}\big]\big): D_0, \wh{\lambda}^{++}_{B_1}\}
\\[1ex]
\mbox{(the notation is similar as in (\ref{2.18}) iii) with $D_0$ in place of $B_1$)}.
\end{array}
\end{equation}
Then for any box $B_0$, using the notation below (\ref{2.14}) for the eight inserted values above a generic point $\lambda^+$ of $\Sigma^0 \cap (0, \gamma_-]$, we set
\begin{equation}\label{2.39}
Y_{B_0} = \dsl_{\lambda^+ \le \gamma_- \,{\rm in} \, \Sigma^0} \;\; \dsl_{x \in B_0} \; 1\big\{ x \longleftrightarrow S(x,R): D_0, \wh{\lambda}^{+++} \;\mbox{and} \; x \lr \;\; S\big(x, \big[\mbox{\f $\dis\frac{L_0}{2}$}\big]\big): D_0, \wh{\lambda}^{++}\big\}
\end{equation}
and find that (with $\ov{K} = 2 K + 3$)
\begin{align}
\wt{I\!I\!I}_{1,R} & \le \dsl_{B_0 \in \cC_0} Y_{B_0} = \dsl_{\tau \in \{0,\dots, \ov{K}-1\}^d} \; \; \dsl_{B_0 \in \cC_{0,\tau}} Y_{B_0}, \;\mbox{where} \label{2.40}
\\
\cC_{0,\tau} & = \cC_0 \cap \{L_0 \tau + \ov{K} \IL_0\} \; \mbox{for each $\tau \in \{0,\dots , \ov{K}-1\}^d$ (and $\IL_0 = L_0 \,\IZ^d$, see (\ref{1.9}))}. \label{2.41}
\end{align}

\n
We will now stochastically dominate each sum $\sum_{B_0 \in \cC_{0,\tau}}  Y_{B_0}$, for $\tau \in \{0,\dots, \ov{K} - 1\}^d$. For this purpose we now recall some facts concerning the soft local time technique of \cite{PopoTeix15}, \cite{ComeGallPopoVach13}, and also refer to Section 4 of \cite{Szni17}, and Section 2 of \cite{Szni19d} for further details.

\medskip
Given any  $\tau \in \{0,\dots, \ov{K} - 1\}^d$, the soft local time technique provides a coupling $\IQ^\tau_0$ of the excursions $Z_\ell^{D_0}$, $\ell \ge 1$, $B_0 \in \cC_{0,\tau}$ of the random interlacements with independent excursions $\wt{Z}_\ell^{D_0}$, $\ell \ge 1$, $B_0 \in \cC_{0,\tau}$, respectively distributed as $X_{\point \wedge T_{U_0}}$ under $P_{\ov{e}_{D_0}}$, and independent right-continuous Poisson counting functions with unit intensity, vanishing at $0$, $(n_{D_0}(0,t))_{t \ge 0}$, $B_0 \in \cC_{0,\tau}$:
\begin{equation}\label{2.42}
\left\{\begin{array}{l}
\mbox{under $\IQ^\tau_0$, as $B_0$ varies over $\cC_{0,\tau}$, the $((n_{D_0}(0,t))_{t \ge 0}$, $\wt{Z}_\ell^{D_0}$, $\ell \ge 1$) are independent}
\\
\mbox{collections of independent processes with $(n_{D_0}(0,t))_{t \ge 0}$-distributed as a Poisson}
\\
\mbox{counting process of intensity $1$, and $\wt{Z}_\ell^{D_0}$, $\ell \ge 1$, as i.i.d. $\Gamma(U_0)$-valued variables}
\\
\mbox{(see notation above (\ref{1.1})) with same law as $X_{\point \wedge T_{U_0}}$ under $P_{\ov{e}_{D_0}}$}.
\end{array}\right.
\end{equation}

\n
The coupling $\IQ^\tau_0$ has an important property. For $x \in \IZ^d$ denote by $Q^0_x$ the joint law of two independent walks $X^1_\point$ and $X^2_\point$ respectively starting from $x$ and from the equilibrium measure of $\bigcup_{B_0 \in \cC_{0,\tau}} D_0$, and let $Y^0$ be the random variable equal to the location where $X^1_\point$ enters $\bigcup_{B_0 \in \cC_{0,\tau}} D_0$ if the corresponding entrance time is finite and $X^2_0$ otherwise. The important property of $\IQ^\tau_0$ is the following (see Lemma 2.1 of \cite{ComeGallPopoVach13}): if for some $\delta \in (0,1)$ and all $B_0 \in \cC_{0,\tau}$, $y \in D_0$ and $x \in \partial (\bigcup_{z \in \cC_{0,\tau}} U_{0,z})$
\begin{equation}\label{2.43}
\Big(1 - \mbox{\f $\dis\frac{\delta}{3}$}\Big) \; \ov{e}_{D_0} (y) \le Q^0_x [Y^0 = y \,| \, Y^0 \in D_0] \le \Big(1 + \mbox{\f $\dis\frac{\delta}{3}$} \Big) \, \ov{e}_{D_0} (y) ,
\end{equation}
then, for any $B_0 \in \cC_{0,\tau}$ and $m_0 \ge 1$, on the event
\begin{equation}\label{2.44}
\wt{\mbox{\f $U$}}^{m_0}_{D_0} = \big\{n_{D_0} \big(m,(1 + \delta)m\big) < 2 \delta m, (1-\delta) m < n_{D_0} (0,m) < ( 1 + \delta) m, \;\mbox{for all}\;  m \ge m_0\},
\end{equation}
one has for all $m \ge m_0$ the following inclusion among subsets of $\Gamma(U_0)$:
\begin{align}
\{\wt{Z}_1^{D_0}, \dots, \wt{Z}^{D_0}_{(1 - \delta)m}\} & \subseteq \{Z_1^{D_0}, \dots, Z^{D_0}_{(1 + 3 \delta)m}\}, \label{2.45}
\\[1ex]
\{Z_1^{D_0}, \dots,Z^{D_0}_{(1 - \delta)m}\} & \subseteq \{\wt{Z}_1^{D_0}, \dots, \wt{Z}^{D_0}_{(1 + 3 \delta)m}\}, \label{2.46}
\end{align}
where $\wt{Z}^{D_0}_v$ and $Z_v^{D_0}$ respectively stand for $ \wt{Z}^{D_0}_{[v]}$ and $Z^{D_0}_{[v]}$ when $v \ge 1$ and the sets in the left members of (\ref{2.45}) and (\ref{2.46}) are empty when $(1 - \delta) \, m < 1$. Importantly, the favorable event $\wt{U}^{m_0}_D$ is defined solely in terms of $(n_{D_0}(0,t))_{t \ge 0}$.

\medskip
We then set
\begin{equation}\label{2.47}
m_0 = [(\log L_0)^2] + 1.
\end{equation}

\n
Then, as in (4.11) of \cite{Szni17} (see also below (4.16) of the same reference), we can make sure that
\begin{equation}\label{2.48}
\begin{array}{l}
\mbox{when $K \ge c(\gamma, u, \ve)$, (\ref{2.43}) holds with a $\delta (\gamma, u, \ve) \in (0,1)$ such that}
\\[1ex]
\mbox{\f $\dis\frac{1 + 4 \delta}{1- \delta}$} < 1 + \mbox{\f $\dis\frac{\kappa}{2}$} \;\mbox{and} \; \mbox{\f $\dis\frac{1 - \delta}{1 + 4 \delta}$} > 1 - \mbox{\f $\dis\frac{\kappa}{2}$}, \;\mbox{with $\kappa$ as in (\ref{2.23})}.
\end{array}
\end{equation}
As a result, see (2.33), (2.34) of \cite{Szni19d}, for large $N$ so that for all $\lambda \in \Sigma$, $m = [\frac{\lambda}{1-\delta} \,{\rm cap}(D_0)] + 1$ satisfies $m \ge m_0$ as well as $(1 + 3 \delta)\,m < \frac{1 + 4 \delta}{1-\delta} \,\lambda {\rm cap}(D_0)$, and $m' = [\frac{\lambda}{1 + 3 \delta} {\rm cap} (D_0)]$ satisfies $m' \ge m_0$ as well as $(1-\delta) \,m' \ge \frac{1- \delta}{1 + 4 \delta} \,\lambda\, {\rm cap}(D_0)$, one has on the event $\wt{U}^{m_0}_{D_0}$, for all $\lambda \in \Sigma$:
\begin{align}
&\{Z_1^{D_0}, \dots, Z^{D_0}_{\lambda \,{\rm cap}(D_0)}\} \subseteq \{\wt{Z}^{D_0}_1,\dots, \wt{Z}^{D_0}_{(\frac{1 + 4 \delta}{1-\delta}) \lambda\,{\rm cap}(D_0)}\}, \;\mbox{and} \label{2.49}
\\[1ex]
&\{\wt{Z}_1^{D_0}, \dots, \wt{Z}^{D_0}_{(\frac{1-\delta}{1 + 4 \delta})\,\lambda \,{\rm cap}(D_0)}\} \subseteq \{Z^{D_0}_1,\dots,Z^{D_0}_{\lambda\,{\rm cap}(D_0)}\}. \label{2.50}
\end{align}
We then introduce for $B_0 \in \cC_{0,\tau}$ (with $\lambda < \lambda'$ elements of $\Sigma$ in the sum below, and we recall that the largest element of $\Sigma$ is $\gamma$, see (\ref{2.6}), and above (\ref{2.14}))
\begin{equation}\label{2.51}
\wt{Y}_{B_0} = \dsl_{\lambda < \lambda '}  \; \dsl_{x \in B_0} 1\{x \stackrel{\mbox{\Large $_\sim$}}{\longleftrightarrow} S(x, R): D_0, \lambda' \; \mbox{and} \;x\overset{\;\;\;\;\mbox{\Large $_\sim$}}{\lr} \;\;S\big(x, [\mbox{\f $\dis\frac{L_0}{2}$}\big]\big): D_0, \lambda\}
\end{equation}
where ``$\sim$'' above the arrows means that there is a connection between $x$ and $S(x,R)$ in \linebreak $D_0 \backslash({\rm range} \,\wt{Z}^{D_0}_1 \cup \dots \cup {\rm range} \, \wt{Z}^{D_0}_{\lambda ' \,{\rm cap}(D_0)})$ and an absence of connection between $x$ and $S(x, [\frac{L_0}{2}])$ in $D_0 \backslash  ({\rm range} \,\wt{Z}^{D_0}_1 \cup \dots \cup {\rm range} \, \wt{Z}^{D_0}_{\lambda  \,{\rm cap}(D_0)})$. By (\ref{2.23}) we have $(1-\kappa) \,\wh{\lambda}^{+++} > \wt{\lambda}^{+++}$ and $(1 + \kappa)\,\wh{\lambda}^{++} < \wt{\lambda}^{++}$ for each term in the sum defining $Y_{B_0}$ in (\ref{2.39}). Thus, making use of (\ref{2.48}) - (\ref{2.50}), we see that for large $N$,
\begin{equation}\label{2.52}
\mbox{for any $B_0 \in \cC_{0,\tau}$, on $\wt{U}_{D_0}^{m_0}$, one has $Y_{B_0} \le \wt{Y}_{B_0}$}.
\end{equation}
In addition, by (\ref{2.42}), (\ref{2.44}),
\begin{equation}\label{2.53}
\mbox{the $(\wt{U}^{m_0}_{D_0}$, $\wt{Y}_{B_0}$), for $B_0 \in \cC_{0,\tau}$ are i.i.d. under $\IQ^\tau_0$}.
\end{equation}
We then introduce the i.i.d. Bernoulli variables
\begin{equation}\label{2.54}
\mbox{$\wt{X}_{B_0} = 1_{(\wt{U}^{m_0}_{D_0})^c \cup \{\wt{Y}_{B_0} \ge \ve |B_0|\}}$, for $B_0 \in \cC_{0,\tau}$}.
\end{equation}
Then, by (\ref{2.52}), for large $N$, we have
\begin{equation}\label{2.55}
\dsl_{B_0 \in \cC_{0,\tau}} Y_{B_0} \le \dsl_{B_0 \in \cC_{0,\tau}} (\wt{Y}_{B_0} (1 - \wt{X}_{B_0}) + |\Sigma|^2 |B_0|\,\wt{X}_{B_0}) \le \dsl_{B_0 \in \cC_{0,\tau}} (\ve \,|B_0| +  |\Sigma |^2 \,|B_0| \,\wt{X}_{B_0}).
\end{equation}
We will now prove that
\begin{equation}\label{2.56}
\begin{array}{l}
\mbox{for large $N$, and any $\tau \in \{0,\dots, \ov{K} -1\}^d$, under $\IQ^\tau_0$ the i.i.d. Bernoulli}
\\
\mbox{variables $\wt{X}_{B_0}$ have success probability at most $\ve / (2 |\Sigma|^2)$}.
\end{array}
\end{equation}
Once (\ref{2.56}) is proved, it will follow from usual large deviation bounds on sums of i.i.d. Bernoulli variables that for large $N$ and each $\tau \in \{0,\dots, \ov{K}-1\}^d$, $\sum_{B_0 \in \cC_{0,\tau}} \,|\Sigma|^2\, \wt{X}_{B_0} < \frac{3}{4} \; \ve \, |\cC_{0,\tau}|$ except on a set of $\IQ^\tau_0$-probability at most $\exp\{-c(\gamma, u, \ve)|\cC_{0,\tau}|\}$.

\medskip
Observing that $|\cC_{0,\tau}| \ge c(K) (N/L_0)^d \gg N^{d-2}$, as $N \rightarrow \infty$, see (\ref{1.7}), we will conclude by (\ref{2.55}) that for each $\tau$, $\lim_N\; \frac{1}{N^{d-2}} \log \IP[\sum_{B_0 \in \cC_{0,\tau}} Y_{B_0} \ge  \frac{7}{4} \; \ve |B_0|\, |\cC_{0,\tau}|]  = - \infty$. Summing over $\tau$ and using (\ref{2.40}) the proof of Lemma \ref{lem2.2} will be completed.

\medskip
There remains to prove (\ref{2.56}). By classical exponential Chebyshev bounds on the Poisson distribution, with $\wt{U}^{m_0}_{B_0}$ as in (\ref{2.44}), one knows that $\IQ^\tau_0 [(\wt{U}^{m_0}_{D_0})^c]$ decays exponentially with $m_0$, see (4.20) of \cite{Szni17}. Hence, (\ref{2.56}) will follow by Chebyshev inequality once we show that
\begin{equation}\label{2.57}
\begin{array}{l}
\mbox{for $K \ge c(\gamma, u, \ve)$, for large $N$ and any $\tau \in \{0,\dots, \ov{K}-1\}^d$, and}
\\
B_0 \in \cC_{0,\tau}, E^{\IQ_0^\tau} [\wt{Y}_{D_0}] \le \frac{\ve^2}{4|\Sigma |^2} \;|B_0|.
\end{array}
\end{equation}
Writing $v = \frac{\lambda + \lambda '}{2}$ for any $\lambda < \lambda'$ in $\Sigma$, we denote by $N^{D_0}_v$ an independent Poisson variable with parameter $v \,{\rm cap}(D_0)$. Then we have, see (\ref{2.51}),
\begin{equation}\label{2.58}
\begin{array}{l}
E^{\IQ_0^\tau} [\wt{Y}_{B_0}] \le \dsl_{\lambda  < \lambda '} \; \dsl_{x \in B_0} \big(\IQ^\tau_0 [N^{D_0}_v \notin \big(\lambda \, {\rm cap}(D_0), \lambda' \,{\rm cap}(D_0)\big)\big]
\\
+ \; \IQ_0^\tau [N^{D_0}_v \in \big(\lambda \, {\rm cap} (D_0), \lambda' \,{\rm cap}(D_0)\big), x \stackrel{\mbox{\large $\sim$}}{\longleftrightarrow} S(x,R): D_0, \lambda' \;\mbox{and} 
\\[1ex]
x \stackrel{\;\;\;\mbox{\large $\sim$}}{\lr} \;\;S\big(x,\big[\mbox{\f $\dis\frac{L_0}{2}$}\big]\big): D_0, \lambda\big]\big).
\end{array}
\end{equation}
Using large deviation bounds on the Poisson distribution and comparing the effect on $B(x, [\frac{L_0}{2}])$ of $N^{D_0}_v$ independent excursions distributed as $X_{\point \wedge T_{U_0}}$ under $P_{\ov{e}_{D_0}}$ to the effect of $N^{D_0}_v$ independent excursions distributed as $X_\point$ under $P_{\ov{e}_{D_0}}$ (and hence of random interlacement at level $v$), we find that for large $N$, any $\tau \in \{0, \dots , \ov{K}-1\}^d$  and $B_0 \in \cC_{0,\tau}$: 
\begin{equation}\label{2.59}
\begin{array}{l}
E^{\IQ_0^\tau} [\wt{Y}_{B_0}] \le |\Sigma|^2 |B_0| \, e^{-c(\gamma, u,\ve)\,L_0^{d-2}} \;+
\\
\dsl_{\lambda < \lambda '} \;\dsl_{x \in B_0} \big(\IP \big[ x \stackrel{v}{\longleftrightarrow} S(x,R) \; \mbox{and} \; x \vv \;\;S\big(x, \big[\mbox{\f $\dis\frac{L_0}{2}$}\big]\big)\big] + \IP[\mbox{a trajectory in the}
\\
\mbox{interlacement at level $\le v$ enters $D_0$ and after exiting $U_0$ touches $B(x,R)]$)}
\\
\le |\Sigma |^2 |B_0| \big(e^{-c(\gamma, u, \ve) L_0^{d-2}} + \sup_{w \le \gamma}  \IP [0 \stackrel{w}{\longleftrightarrow} S(0,R) \;\mbox{and} \; 0 \ww \;\;\infty] + c \gamma \,L_0^{d-2} \;\mbox{\f $\dis\frac{R^{d-2}}{(K L_0)^{d-2}}$}\big)
\\
\stackrel{(\ref{2.13}), |\Sigma| \le 10 |\Sigma^0|}{\le} |B_0| \big(|\Sigma|^2 e^{-c(\gamma,u,\ve)L^{d-2}_0} + \mbox{\f $\dis\frac{1}{10 |\Sigma|^2}$} \;\ve^2 + |\Sigma |^2 \, c \gamma \; \mbox{\f $\dis\frac{R^{d-2}}{K ^{d-2}}$}\big).
\end{array}
\end{equation}
Since $\Sigma$ and $R$ are determined by $\gamma,u,\ve$ (and $d$) the claim (\ref{2.57}) holds. This proves (\ref{2.56}) and as explained below (\ref{2.56}) this completes the proof of Lemma \ref{lem2.2}. 
\end{proof}

\medskip
We have now established (\ref{2.33}) (thanks to (\ref{2.36}) and Lemma \ref{lem2.2}). There remains to show that with overwhelming probability the term $I\!I\!I_{2,R}$ in (\ref{2.32}) is small. This is the object of
\begin{lemma}\label{lem2.3}
\begin{equation}\label{2.60}
\mbox{For $K \ge c(\alpha, \beta, \gamma, u, \ve)$ one has $\lim\limits_N \; \mbox{\f $\dis\frac{1}{N^{d-2}}$} \;\log \IP [ I\!I\!I_{2,R} \ge \ve \,|D_N|] = - \infty$}.
\end{equation}
\end{lemma}

\begin{proof}
Recall that $I\!I\!I_{2,R} = \sum_{B_1 \in \cC_1} |B_1| \,1\{B_1$ is $I\!I\!I_{R}$-bad$\}$ and the event $\{B_1$ is $I\!I\!I_{R}$-bad$\}$ entails that one of the conditions (\ref{2.27}) - (\ref{2.30}) does not hold. We will first show that 
\begin{equation}\label{2.61}
\begin{array}{l}
\mbox{for $K \ge c(\alpha, \beta, \gamma)$, $\lim_N \; \mbox{\f $\dis\frac{1}{N^{d-2}}$} \; \log \IP [IV \ge \mbox{\f $\dis\frac{\ve}{3}$} \;|D_N|\} = - \infty$, where we have set}
\\
\mbox{$IV = \dsl_{B_1 \in \cC_1} |B_1| \;1\{B_1$ does not satisfy (\ref{2.27})$\}$}.
\end{array}
\end{equation}
The claim is a variation on (\ref{1.42}) (it is not an immediate consequence of (\ref{1.42}) because as $N \r \infty$, by (\ref{1.7}), (\ref{1.8}), $|\cC_1| \sim c(N / L_1)^d \sim c\,N^{d-2} / \log N$ which might be small compared to $\rho(L_0) \, N^{d-2}$ in (\ref{1.42})). Keeping the same notation as in the proof of Lemma \ref{lem1.4}, and with $K \ge c_2 (\alpha, \beta, \gamma)$ (that corresponds to $c_8(\alpha, \beta, \gamma)$ above (5.5) of \cite{Szni17}), one finds that for any $\tau \in \{0, \dots , \ov{K}-1\}^d$, the indicator functions of the events $B_0$ is an $(\alpha, \beta, \gamma)$-bad box, for $B_0 \in \cC_{0,\tau}$, are stochastically dominated by i.i.d. Bernoulli variables with success probability $\eta(L_0)$.

\medskip
It follows that the indicator functions of the events $B_1$ contains an $(\alpha, \beta, \gamma)$-bad box $B_0 \in \cC_{0,\tau}$, as $B_1$ varies over $\cC_1$ are stochastically dominated by i.i.d. Bernoulli variables with success probability $\eta_1 = 1 \wedge \{(\frac{L_1}{L_0})^d \,\eta (L_0)\}$. Hence, from the super-polynomial decay of $\eta$, we find that $\lim_N \;\frac{1}{\log L_1} \;\log \eta_1 = - \infty$. Then, as in Proposition 3.1 and (4.9) of \cite{Szni19d}, we can conclude that for each $\tau \in \{0, \dots, \ov{K} - 1\}^d$,
\begin{align}
\lim\limits_N \; \mbox{\f $\dis\frac{1}{N^{d-2}}$}\; \log \IP[ & \mbox{there are at least $\mbox{\f $\dis\frac{\ve}{3 \ov{K}^d}$}\;|\cC_1|$ boxes in $\cC_1$ that contain an} \label{2.62}
\\
&\mbox{$(\alpha, \beta, \gamma)$-bad box of $\cC_{0,\tau}] = - \infty$}. \nonumber
\end{align}
Summing over $\tau$ and noting that $|D_N| \ge |\cC_1| \, |B_1|$ the claim (\ref{2.61}) follows. 

\medskip
Our next step is to show that
\begin{equation}\label{2.63}
\begin{array}{l}
\mbox{for $K \ge c(\gamma, u, \ve)$, $\lim_N \; \mbox{\f $\dis\frac{1}{N^{d-2}}$} \; \log \IP [V \ge \mbox{\f $\dis\frac{\ve}{3}$} \;|D_N|] = - \infty$, where we have set}
\\
\mbox{$V = \dsl_{B_1 \in \cC_1} |B_1| \;1\{B_1$ does not satisfy (\ref{2.28}) or (\ref{2.29})$\}$}.
\end{array}
\end{equation}
In analogy with (\ref{2.41}), for each $\tau \in \{0,\dots, \ov{K} - 1\}^d$ we define
\begin{equation}\label{2.64}
\cC_{1,\tau} = \cC_1 \cap \{L_1 \tau + \ov{K} \,\IL_1\} \; \mbox{(where $\IL_1 = L_1 \, \IZ^d$, see (\ref{1.10}))}.
\end{equation}
Replacing in (\ref{2.42}) - (\ref{2.46}) $D_0$ by $B_1$ and $U_0$ by $U_1$, one can construct for each $\tau \in \{0,\dots, \ov{K}-1\}^d$ a coupling $\IQ^\tau_1$ between the excursions $Z^{B_1}_\ell$, $\ell \ge 1$, $B_1 \in \cC_{1, \tau}$ of the random interlacements with independent excursions $\wt{Z}^{B_1}_\ell$, $\ell \ge 1$, $B_1 \in \cC_{1,\tau}$ respectively distributed as $X_{\point \wedge T_{U_1}}$ under $P_{\ov{e}_{B_1}}$, and independent right-continuous Poisson counting functions with unit intensity, vanishing at $0$, $(n_{B_1}(0,t))_{t \ge 0}$, $B_1 \in \cC_{1,\tau}$, so that the corresponding statement to (\ref{2.42}), with $B_1$ in place of $D_0$ and $U_1$ in place of $U_0$ holds. Then with $Q^1_x$ and $Y^1$ analogously defined as above (\ref{2.43}), the coupling has the following property: if for some $\delta \in (0,1)$ and all $B_1$ in $\cC_{1,\tau}$, $y \in B_1$ and $x \in \partial (\bigcup_{z \in \cC_{1,\tau}} U_{1,z})$ one has
\begin{equation}\label{2.65}
\big(1 - \mbox{\f $\dis\frac{\delta}{3}$}\big) \,\ov{e}_{B_1} (y) \le Q^1_x [Y^1 = y\,| \, Y^1 \in B_1] \le \big(1 +  \mbox{\f $\dis\frac{\delta}{3}$}\big)\,\ov{e}_{B_1}(y),
\end{equation}
then for $B_1 \in \cC_{1, \tau}$ on the event
\begin{equation}\label{2.66}
\wt{\mbox{\f $U$}}^{m_1}_{B_1} = \{n_{B_1} \big(m, (1 + \delta)\,m\big) < 2 \delta m, (1-\delta)\,m < n_{B_1} (0,m) < (1 + \delta)\,m, \;\mbox{for all}\; m \ge m_1\},
\end{equation}
one has for all $m \ge m_1$ the inclusions among subsets of $\Gamma(U_1)$:
\begin{align}
\{\wt{Z}^{B_1}_1, \dots , \wt{Z}^{B_1}_{(1-\delta)\,m}\} & \subseteq \{Z^{B_1}_1, \dots , Z^{B_1}_{(1 + 3 \delta)\,m}\}, \label{2.67}
\\[1ex]
\{Z^{B_1}_1, \dots , Z^{B_1}_{(1-\delta)\,m}\} & \subseteq \{\wt{Z}^{B_1}_1, \dots , \wt{Z}^{B_1}_{(1 + 3 \delta)\,m}\}, \label{2.68}
\end{align}
(with similar conventions as stated below (\ref{2.46})).

%\medskip
We then set
\begin{equation}\label{2.69}
m_1 = [(\log L_1)^2] + 1.
\end{equation}
We now choose $\delta(\gamma, u, \ve) \in (0,1)$ such that with $\kappa$ as in (\ref{2.23})
\begin{equation}\label{2.70}
\begin{array}{rl}
{\rm i)} & \dis\frac{1 + \kappa/10}{1 - \kappa/10}  \;\dis\frac{(1 + 4 \delta)^2}{(1 - 2 \delta)^2} \le 1 + \dis\frac{\kappa}{4}, \;\; \mbox{and}\quad  {\rm ii)} \; \Big(1-\dis\frac{\kappa}{10}\Big)  \;\dis\frac{1 - \delta}{1 + 4 \delta} \ge 1 -\dis\frac{\kappa}{4}.
\end{array}
\end{equation}
Then for $K \ge c(\gamma, u, \ve)$ we can make sure that (\ref{2.65}) holds (as explained above (\ref{2.48})), and ensure that for $B_1 \in \cC_{1,\tau}$ on $\wt{U}^{m_1}_{B_1}$ the statements (\ref{2.67}), (\ref{2.68}) hold for all $m \ge m_1$.

\medskip
We further introduce for $B_1 \in \cC_{1,\tau}$ the good event
\begin{equation}\label{2.71}
\begin{split}
\wt{G}_{B_1} = \wt{U}^{m_1}_{B_1} \cap\bigcap\limits_{B_0 \subseteq {\rm Deep}\,B_1} \Big\{ &\wt{Z}^{B_1}_1, \dots, \wt{Z}^{B_1}_m \; \mbox{contain at least $\Big(1 - \dis\frac{\kappa}{10}  \Big) \; \dis\frac{{\rm cap}(D_0)}{{\rm cap}(B_1)} \; m$ and}
\\[-1ex]
&\mbox{at most $\Big(1 + \dis\frac{\kappa}{10}\Big)\;\dis\frac{{\rm cap}(D_0)}{{\rm cap}(B_1)} \; m$ excursions from $D_0$ to $\partial U_0$}
\\
&\mbox{for all $m \ge m'_1\Big\}$},
\end{split}
\end{equation}

\vspace{-2ex}\noindent
where we set
\begin{equation}\label{2.72}
m'_1 = \Big[\Big( \dis\frac{{\rm cap}(B_1)}{{\rm cap}(D_0)}\Big)^2 (\log L_1)^2 \Big] + 1.
\end{equation}
An important feature of $m'_1$ is that when $N$ goes to infinity
\begin{equation}\label{2.73}
m'_1 \ge m_1, \; m'_1 \Big(\dis\frac{{\rm cap}(D_0)}{{\rm cap}(B_1)}\Big)^2 \ge (\log L_1)^2 \; \mbox{and $m'_1 = o\big({\rm cap}(B_1)\big)$}
\end{equation}
(the last property stems from the fact that $(\frac{L_1}{L_0})^2 = o(L_1)$ as $N \r \infty$, see below (\ref{1.8})).

%\medskip
%\pagebreak
Note that under $\IQ^\tau_1$
\begin{equation}\label{2.74}
\mbox{the events $\wt{G}_{B_1}$, $B_1 \in \cC_{1,\tau}$ are i.i.d.}.
\end{equation}
In addition, when $K$ is large, they are typical in the sense that
\begin{equation}\label{2.75}
\begin{array}{l}
\mbox{for} \; K \ge c(\gamma, u, \ve), \; \lim\limits_N \; \dis\frac{1}{\log L_1} \; \log \IQ^\tau_1 [(\wt{G}_{B_1})^c] = - \infty
\\
\mbox{(the above probability does not depend on $\tau$ nor on $B_1 \in \cC_{1,\tau}$)}.
\end{array}
\end{equation}
The proof of this fact is very similar to the proof of (4.15) of \cite{Szni17}. In that proof, Lemma 4.2 of \cite{Szni17} is now replaced by the estimate for $B_0 \subseteq {\rm Deep} \, B_1$:
\begin{equation}\label{2.76}
\dis\frac{{\rm cap}(D_0)}{{\rm cap}(B_1)}  \ge p \stackrel{\rm def}{=} P\,\!_{\ov{e}_{B_1}} [H_{D_0} < T_{U_1}] \ge \dis\frac{{\rm cap}(D_0)}{{\rm cap}(B_1)}  \; \Big(1 - \dis\frac{c}{K^{d-2}}\Big)
\end{equation}
(the first inequality follows from a straightforward sweeping argument, and the second inequality comes from writing
\begin{equation*}
\begin{split}
 P\,\!_{\ov{e}_{B_1}} [H_{D_0} < T_{U_1}] & =  P\,\!_{\ov{e}_{B_1}} [H_{D_0} < \infty] -  P\,\!_{\ov{e}_{B_1}} [T_{U_1} < H_{D_0} < \infty]
 \\[1ex]
 & \ge \dis\frac{{\rm cap}(D_0)}{{\rm cap}(B_1)}  - \sup\limits_{\partial U_1} \;P_x [H_{D_0} < \infty] \ge \dis\frac{{\rm cap}(D_0)}{{\rm cap}(B_1)}  \; \Big(1 - \dis\frac{c}{K^{d-2}}\Big) \; ).
\end{split}
\end{equation*}
The proof of (\ref{2.75}) follows the same steps, see (4.27) and (4.30) of \cite{Szni17}, as the proof of (4.15) of \cite{Szni17}, with minor adjustments. With $p \ge (1 - \frac{\kappa}{20}) \, \frac{{\rm cap}(D_0)}{{\rm cap}(B_1)}$ and $\wt{p} =  (1- \frac{\kappa}{10})\, \frac{{\rm cap}(D_0)}{{\rm cap}(B_1)}$ the lower bound on the rate function in (4.26) of \cite{Szni17} is now replaced by $c(p - \wt{p})^2 \ge c ((1-\frac{\kappa}{20}) \;\frac{{\rm cap}(D_0)}{{\rm cap}(B_1)} - (1-\frac{\kappa}{10}) \frac{{\rm cap}(D_0)}{{\rm cap}(B_1)} )^2 \ge c' \, \kappa^2 (\frac{{\rm cap}(D_0)}{{\rm cap}(B_1)})^2$, so that now (4.27) of \cite{Szni17} is replaced by the fact that $\sum_{m \ge m'_1} \exp\{- c' \,\kappa^2 (\frac{{\rm cap}(D_0)}{{\rm cap}(B_1)})^2 \,m\}$ decays super-polynomially in $N$ since $(\frac{{\rm cap}(D_0)}{{\rm cap}(B_1)})^2 \,m'_1 \ge (\log L_1)^2$, see (\ref{2.73}). The bound (4.29) of \cite{Szni17} is in our context essentially unchanged (with $\overset{\vee}{D}$ replaced by $B_1$ and $\kappa$ by $\frac{\kappa}{10}$) and at the end we note that $\frac{{\rm cap}(D_0)}{{\rm cap}(B_1)} \; m'_1 \ge (\log L_1)^2$.

\medskip
The interest of the above event $\wt{G}_{B_1}$ in (\ref{2.71}) comes from the fact that similarly to (4.12), (4.13) of \cite{Szni17}, with $\delta$ chosen as in (\ref{2.70}), as we explain below, it ensures for each $B_0 \subseteq {\rm Deep}\,B_1$ a good comparison of the successive excursions $\wh{Z}^{D_0}_1, \dots, \wh{Z}^{D_0}_\ell, \dots$ between $D_0$ and $\partial U_0$ in the sequence $\wt{Z}_1^{B_1}, \dots, \wt{Z}^{B_1}_k, \dots$ with the successive excursions $Z^{D_0}_1, \dots Z_\ell^{D_0}, \dots$ in the random interlacements.

\medskip
More precisely, first observe that for large $N$, for any $B_1 \in \cC_{1,\tau}$, on $\wt{G}_{B_1}$ one has
\begin{equation}\label{2.77}
\{\wt{Z}^{B_1}_1, \dots, \wt{Z}^{B_1}_{(1-\delta) m} \} \stackrel{(\ref{2.67})}{\subseteq} \{Z^{B_1}_1, \dots, Z^{B_1}_{(1 + 3 \delta) m}\} \stackrel{(\ref{2.68})}{\subseteq} \{\wt{Z}^{B_1}_1, \dots , \wt{Z}^{B_1}_{\frac{(1 + 4 \delta)^2}{1 - \delta} \;m}\}, \;\mbox{for all $m \ge m_1$}.
\end{equation}

\n
Next, for $B_0 \subseteq {\rm Deep}\,B_1$ denote by $\wh{Z}^{D_0}_1, \dots, \wh{Z}^{D_0}_\ell, \dots$ the successive excursions from $D_0$ to $\partial U_0$ inscribed in the $\wt{Z}^{B_1}_1, \dots , \wt{Z}^{B_1}_k, \dots$~. We then argue as below (4.17) of \cite{Szni17}. When $(1-\delta)\,m \ge m'_1$ the set of excursions on the left-hand side of (\ref{2.77}) contains at least $(1 - \frac{\kappa}{10}) [(1 - \delta) \,m] \; \frac{{\rm cap}(D_0)}{{\rm cap}(B_1)} \ge t = (1 - \frac{\kappa}{10}) (1 - 2 \delta)\,m \frac{{\rm cap}(D_0)}{{\rm cap}(B_1)}$ excursions from $D_0$ to $\partial U_0$, and the set on the right-hand side contains at most $(1 + \frac{\kappa}{10}) \;\frac{(1 + 4 \delta)^2}{1 - \delta} \;m \frac{{\rm cap}(D_0)}{{\rm cap}(B_1)} \stackrel{(\ref{2.70}\,{\rm i)}}{\le} (1 + \frac{\kappa}{4}) \;t$ excursions from $D_0$ to $\partial U_0$.

\medskip
Hence, looking at the first $t$ excursions $\wh{Z}^{D_0}_\ell$, $1 \le \ell \le t$ (which are inscribed in the set of excursions on the leftmost member of (\ref{2.77})) and the first $Z^{D_0}_\ell$, $1 \le \ell \le (1 + \frac{\kappa}{4})\;t$ (which exhaust all excursions from $D_0$ to $\partial V_0$ inscribed in the set of excursions in the middle of (\ref{2.77})), we see that
\begin{equation}\label{2.78}
\big\{ \wh{Z}^{D_0}_1, \dots, \wh{Z}^{D_0}_t\big\} \subseteq \{Z_1^{D_0}, \dots , Z^{D_0}_{(1 + \frac{\kappa}{4}) \,t} \big\}.
\end{equation}
Moreover, looking at the first $t$ excursions from $D_0$ to $\partial U_0$ contained in the set of excursions in the middle of (\ref{2.77}), and at the first $(1 + \frac{\kappa}{4})\,t$ excursions $\wh{Z}^{D_0}_\ell$, $1 \le \ell \le (1 + \frac{\kappa}{4})\,t$ (which exhaust all excursions from $D_0$ to $\partial U_0$ inscribed within the set of excursions in the rightmost member of (\ref{2.77})), we see that
\begin{equation}\label{2.79}
\{ Z^{D_0}_1, \dots , Z_t^{D_0}\} \subseteq \{\wh{Z}^{D_0}_1,\dots , \wh{Z}^{D_0}_{(1 + \frac{\kappa}{4})\,t}\}.
\end{equation}
Note that $[t]$ covers all integers bigger or equal to $(1 - \frac{\kappa}{10}) (1 - 2 \delta)\ m'_1 \, \frac{{\rm cap}(D_0)}{{\rm cap}(B_1)}$ as $m$ runs over the set of integers bigger or equals to $m'_1$, in particular all integers bigger or equal to $t_N = \frac{{\rm cap}(B_1)}{{\rm cap}(D_0)} \;(\log L_1)^2$. So we have for large $N$, for any $B_1 \in \cC_{1,\tau}$, on $\wh{G}_{B_1}$:
\begin{equation}\label{2.80}
\left\{\begin{array}{rl}
{\rm i)} & \{ \wh{Z}^{D_0}_1, \dots , \wh{Z}_\ell^{D_0}\} \subseteq \{Z^{D_0}_1,\dots , Z^{D_0}_{(1 + \frac{\kappa}{4})\,\ell}\}, \; \mbox{for all $\ell \ge t_N$},
\\[2ex]
{\rm ii)} & \{ Z^{D_0}_1, \dots , Z_\ell^{D_0}\} \subseteq \{\wh{Z}^{D_0}_1,\dots , \wh{Z}^{D_0}_{(1 + \frac{\kappa}{4})\, \ell}\}, \; \mbox{for all $\ell \ge t_N$}.
\end{array}\right.
\end{equation}
Note that $L_0 \sim N^a$ as $N \r \infty$, with $a = \frac{1}{d-1} > \frac{1}{d}$, and in the same spirit as the observation below (\ref{2.73}), we now find that
\begin{equation}\label{2.81}
t_N = o\big({\rm cap}(D_0)\big), \;\mbox{as $N \r \infty$}.
\end{equation}
As we now explain,
\begin{equation}\label{2.82}
\begin{array}{l}
\mbox{for large $N$, for any $\tau \in \{0, \dots, \ov{K} - 1\}^d$ and any $B_1 \in \cC_{1,\tau}$, on $\wt{G}_{B_1}$}
\\
\mbox{both (\ref{2.28}) and (\ref{2.29}) hold.}
\end{array}
\end{equation}
We begin with the case of (\ref{2.28}). Recall from (\ref{2.73}) that $m_1 \le m'_1 = o({\rm cap}(B_1))$, as $N \r \infty$. Thus, for large $N$ with $\tau$ and $B_1$ as above, on $\wt{G}_{B_1}$ for any $\lambda < \lambda'$ in $\Sigma$ and $B_0 \subseteq {\rm Deep} \,B_1$, the first $\lambda {\rm cap}(B_1)$ excursions $\{Z_1^{B_1},\dots, Z^{B_1}_{\lambda {\rm cap}(B_1)}\}$ are contained by (\ref{2.77}) in $\{\wt{Z}^{B_1}_1, \dots, \wt{Z}^{B_1}_{\frac{(1 + 4 \delta)^2}{1 - \delta} \,\lambda {\rm cap}(B_1)}\}$ which contain by definition of $\wt{G}_{B_1}$ in (\ref{2.71}) at most $(1 + \frac{\kappa}{10}) \frac{(1 + 4 \delta)^2}{1-\delta} \;\lambda \, {\rm cap}(D_0) \stackrel{(\ref{2.70}\,i)}{\le} (1 + \frac{\kappa}{4}) \, \lambda \, {\rm cap}(D_0)$ excursions from $D_0$ to $\partial U_0$. So the excursions from $D_0$ to $\partial U_0$ contained in $\{Z_1^{B_1},\dots, Z^{B_1}_{\lambda {\rm cap}(B_1)}\}$ are among $\{\wh{Z}_1^{D_0},\dots, \wh{Z}^{D_0}_{(1 + \frac{\kappa}{4}) \,\lambda {\rm cap}(D_0)}\} \subseteq \{Z^{D_0}_1, \dots , Z^{D_0}_{(1 + \frac{\kappa}{4})^2 \, \lambda {\rm cap}(D_0)}\}$, by (\ref{2.80}) i). Since $(1 + \frac{\kappa}{4})^2 \lambda < \lambda '$ when $\lambda < \lambda '$ in $\Sigma$ by (\ref{2.23}), this shows that (\ref{2.28}) holds.

\medskip
In the case of (\ref{2.29}), we observe that for large $N$, with $\tau$ and $B_1$ as above, on $\wt{G}_{B_1}$ for $\lambda < \lambda'$ in $\Sigma$ and $B_0 \subseteq {\rm Deep} \, B_1$, the first $\lambda' {\rm cap} (B_1)$ excursions $Z_1^{B_1}, \dots , Z^{B_1}_{\lambda ' {\rm cap}(B_1)}$ contain by (\ref{2.77}) the excursions $\wt{Z}^{B_1}_1, \dots , \wt{Z}^{B_1}_{(1 - \delta) [\frac{\lambda'}{1 + 3 \delta} \,{\rm cap} (B_1)]}$, and by definition of $\wt{G}_{B_1}$ this last collection contains the excursions $\wh{Z}^{D_0}_\ell$, $1 \le \ell \le (1 - \frac{\kappa}{10}) \l \frac{1- \delta}{1 + 4 \delta} \; \lambda' {\rm cap}(D_0)$, which by (\ref{2.80}) ii) contain the excursions $Z^{D_0}_\ell$, $1 \le \ell \le a \, \lambda' \,{\rm cap}(D_0)$ where $a = \frac{1}{1 + \kappa/2} \; (1 - \frac{\kappa}{10}) \;\frac{(1-\delta)}{1 + 4 \delta} \stackrel{(\ref{2.70}) ii)}{\ge} \frac{1- \kappa/4}{1 + \kappa/2} \ge 1 - \kappa$, so that $a \, \lambda' > \lambda$ by (\ref{2.23}). This proves that (\ref{2.29}) holds. We have thus shown (\ref{2.82}).

\bigskip
Making use of (\ref{2.74}), (\ref{2.75}), a similar calculation as in Proposition 3.1 of \cite{Szni19d} (see also (4.9) of \cite{Szni19d}) completes the proof of (\ref{2.63}).

\medskip
There remains to handle the case of (\ref{2.30}). To this effect we first observe that when (\ref{2.28}) and (\ref{2.29}) hold the condition (\ref{2.83}) below implies that (\ref{2.80}) holds as well, where we have set (with the notation below (\ref{2.14}))
\begin{equation}\label{2.83}
\begin{array}{l}
\mbox{for all $B_0 \subseteq {\rm Deep} \, B_1$ and all $\overset{\vee}{\lambda}\,\!^{++}$ in the grid $\Sigma$, if two connected sets in}
\\
\mbox{$\wt{B}_0 \backslash ({\rm range} \,Z_1^{D_0} \cup \dots \cup {\rm range} \, Z^{D_0}_{\overset{\vee}{\lambda}\,\!^{++} {\rm cap} (D_0)})$ have diameter at least $\frac{L_0}{10}$, then}
\\[2ex]
\mbox{they are connected in $D_0 \backslash ({\rm range} \, Z_1^{D_0} \cup \dots \cup {\rm range} \,Z^{D_0}_{\overset{\vee}{\lambda}\,\!^+ {\rm cap}(D_0)})$}.
\end{array}
\end{equation}
It is then convenient to say that for $\lambda < \lambda'$ in $(0,\ov{u})$ the box $B_0$ is $(\lambda, \lambda')${\it -good} if
\begin{equation}\label{2.84}
\begin{array}{l}
\mbox{any two connected sets in $\wt{B}_0 \backslash ({\rm range} \,Z_1^{D_0} \cup \dots \cup {\rm range} \, Z^{D_0}_{\lambda' {\rm cap}(D_0)}$) having}
\\
\mbox{diameter at least $\frac{L_0}{10}$ are connected in $D_0 \backslash ({\rm range} \, Z_1^{D_0} \cup \dots \cup {\rm range}\, Z^{D_0}_{\lambda {\rm cap} (D_0)})$},
\end{array}
\end{equation}
and $(\lambda, \lambda ')${\it - bad} otherwise. Then, in view of (\ref{2.61}) and (\ref{2.63}) and the observation above (\ref{2.83}), the claim (\ref{2.60}) will follow, and the proof of Lemma \ref{lem2.3} will be completed, once we show
\begin{equation}\label{2.85}
\begin{array}{l}
\mbox{for $K \ge c(\gamma, u, \ve)$, for any $\tau \in \{0,\dots, \ov{K} - 1\}^d$, and $\lambda < \lambda '$ in $\Sigma$,}
\\
\lim\limits_N\;\dis\frac{1}{N^{d-2}} \; \log \IP[\mbox{there are at least $\frac{\ve}{3 \ov{K}^d |\Sigma |^2}\;  | \cC_1 |$ boxes $B_1$ that contains}
\\
\hspace{2.7cm} \mbox{a $(\lambda, \lambda ')$-bad box of $\cC_{0,\tau}] = -\infty$}.
\end{array}
\end{equation}
The proof of (\ref{2.85}) is similar, but substantially simpler than the proof of (\ref{2.62}). We briefly sketch the argument. One uses the soft local technique as in (\ref{2.42}) - (\ref{2.46}) (using a possibly smaller $\delta$ than in (\ref{2.48}) and large enough $K$ to ensure (\ref{2.43})), and for large $N$ stochastically dominates the events ``$B_0$ is $(\lambda, \lambda ')$-bad'', for $B_0 \in \cC_{0,\tau}$, by the events $(\wt{U}^{m_0}_{D_0})^c \cup \{$there are two connected sets in $\wt{B}_0 \backslash ({\rm range} \,\wt{Z}^{D_0}_1 \cup \dots \cup {\rm range} \,\wt{Z}^{D_0}_{\frac{1}{7} ( \lambda + 6 \lambda '){\rm cap}(D_0)}$) that are not connected in $D_0 \backslash ({\rm range} \, \wt{Z}_1^{D_0} \cup \dots \cup {\rm range} \, \wt{Z}^{D_0}_{\frac{1}{7} (6 \lambda + \lambda'){\rm cap}(D_0)})\}$.

\bigskip
In turn, the probability of such events (that are i.i.d.) is controlled by the probability of $(\wt{U}^{m_0}_{D_0})^c \cup \{B_0$ is $(\frac{1}{7} (5 \lambda + 2 \lambda')$, $\frac{1}{7} (2 \lambda + 5 \lambda ')$)-bad$\}$. Then, to show the super-polynomial decay in $L_0$ of the probability of such events one brings into play (\ref{1.30}) with $w = \frac{1}{7} (4 \lambda + 3 \lambda') < v = \frac{1}{7} \,(3 \lambda + 4 \lambda ')$ as well as the unlikely events $\{N_w (D_0) < \frac{1}{7} (5 \lambda + 2 \lambda ') {\rm cap}(D_0)\}$ and $\{N_v(D_0) > \frac{1}{7} \;(2 \lambda + 5 \lambda') {\rm cap} (D_0)\}$ (with $K \ge c(\lambda, \lambda')$, see below (2.22) of \cite{Szni17}). Then one can argue as above (\ref{2.62}) and conclude that (\ref{2.85}) holds. As explained above (\ref{2.85}) the proof of Lemma \ref{lem2.3} follows. \end{proof}

\medskip
With (\ref{2.33}) and Lemma \ref{lem2.3} we have thus completed the proof of (\ref{2.20}). Together with (\ref{2.19}) this completes the proof of Theorem \ref{theo2.1}.
\hfill  \mbox{\Large $\square$}

\section{On the cost of bubbles}

From Theorem \ref{theo2.1} in the previous section we know that we can replace the excess event $\cA_N$ with the event $\cA'_N$ from (\ref{2.11}) in our quest for an asymptotic upper bound on $\IP[\cA_N]$. The constraint expressed by $\cA'_N$ involves the volume of the bubble set. The main objective of this section is to construct a coarse grained random object (namely, the equilibrium potential of a random set $C_\omega$ of ``low complexity'') that will endow us with a tool to show that the bubble set induces a cost. This feature will play a major role in the next section. The challenge stems from the fact that the bubble set may be very irregular with little depth apart from its constitutive grains of size $L_1$. There is no additional thickening in the problem we study and this precludes the use of the coarse graining procedures from Section 4 of \cite{NitzSzni} (see also \cite{Szni19b} and \cite{ChiaNitz20a}).

\medskip
In this section we assume that
\begin{align}
& 0 < u < \ov{u}, \label{3.1}
\\[1ex]
& \alpha > \beta > \gamma \; \mbox{belong to $(u,\ov{u})$},  \label{3.2}
\\[1ex]
& 0 < \ve < 10^{-3} . \label{3.3}
\intertext{Further, with $c_1$ as in Lemma \ref{lem1.2} and $c_2(\alpha, \beta, \gamma)$ as in Lemma \ref{lem1.4}, we assume that}
&K \ge c_1 \vee c_2 (\alpha, \beta, \gamma). \label{3.4}
\end{align}
We also recall the asymptotically negligible bad event $\cB_N$ defined in (\ref{1.42}) and the bubble set $B u b$ from (\ref{1.47}). Here is the main result of this section. We recall that $\ov{K} = 2 K + 3$.

\begin{theorem}\label{theo3.1}
There exists a dimension dependent constant $c_0 \in (0,1)$ such that for $u, \alpha, \beta, \gamma, \ve, K$ as in (\ref{3.1}) - (\ref{3.4}), for large $N$ on $\cB^c_N$, one can construct a random subset $C_\omega$ of $[-4N, 4N]^d$, which is a union of $B_0$-boxes and satisfies the following properties:
\begin{equation}\label{3.5}
\left\{ \begin{array}{rl}
{\rm i)} & \mbox{for all $B_0 \subseteq C_\omega$, $B_0$ is $(\alpha, \beta, \gamma)$-good and $N_u (D_0) \ge \beta\, {\rm cap}(D_0)$},
\\[1ex]
{\rm ii)} & \mbox{the $B_0 \subseteq C_\omega$ have base points at  mutual  sup-distance at least $\ov{K} \,L_0$}, 
\\[1ex]
{\rm iii)} & \mbox{the set $\cS_N$ of possible values of $C_\omega$ is such that $|\cS_N| = \exp\{o(N^{d-2})\}$},
\\[1ex]
{\rm iv)} & \mbox{the $2 \ov{K} L_1$-neighborhood of $C_\omega$ has volume at most $\ve \,|D_N|$}, 
\\[1ex]
{\rm v)} & \mbox{if $h_{C_\omega}$ stands for the equilibrium potential of $C_\omega$ (see (\ref{1.4})), one has}
\\
& |\{x \in B u b; h_{C_\omega}(x) < c_0\} \le \ve \, |D_N|.
\end{array}\right.
\end{equation}
\end{theorem}

Perhaps some comments on the above conditions are helpful at this stage. Condition iii) on the ``combinatorial complexity'' of $C_\omega$ is a coarse graining control. With the help of iii) when deriving asymptotic bounds on $\IP[\cA'_N]$ in the next section, we will be able to fix the value $C_\omega = C$ of the above random set and derive bounds uniformly on $C$ in $\cS_N$. Condition i) will ensure that with $L^u$ the field of occupation times as in (\ref{1.48}), $\langle  \ov{e}_{D_0}, L^u  \rangle \ge \gamma$ for each $B_0 \subseteq C$, see (\ref{1.38}), and condition iv) that the $2 \ov{K} L_1$-thickening of $C$ has a negligible volume. This will be combined with a coarse graining of the values of $\langle \ov{e}_{B_1}, L^u \rangle$ with the help of the grid $\Sigma^0$ in (\ref{2.6}) for $B_1 \in \cC_1$ at distance at least $\ov{K} L_1$ from $C$. We will then use an exponential Chebyshev bound for this coarse grained picture of the occupation time in the spirit of Proposition 5.6 of \cite{Szni19d}. The crucial condition v) will ensure that in the constraints defining $\cA'_N$, see (\ref{2.11}), a due cost will be attached to the volume of the bubble set, at least on the main part $\cA'_N \backslash \cB_N$ of the event $\cA'_N$.

\medskip
We need some additional notation. We choose an integer $M ( \ge 4)$ solely dependent on the dimension $d$ such that
\begin{equation}\label{3.6}
M^2 / (3^d + 1) > 1 \; \mbox{(for instance the smallest such integer)}.
\end{equation}

\n
We will use in the proof of Theorem \ref{theo3.1} an ``$M$-adic decomposition'' in $\IZ^d$, where $L_1$ (attached to $B_1$-boxes) corresponds to the bottom (i.e.~smallest) scale and the top (i.e.~largest) scale corresponds to $M^{\ell_N} L_1$, where
\begin{equation}\label{3.7}
M^{\ell_N} L_1 \le N < M^{(\ell_N + 1)} L_1. 
\end{equation}
We will ``view things'' from the point of view of the top scale, and $0 \le \ell \le \ell_N$ will label the ``depth'' with respect to the top scale, setting for such $\ell$
\begin{equation}\label{3.8}
\begin{split}
\cI_\ell = & \; \mbox{the collection of $M$-adic boxes of depth $\ell$, i.e.~of boxes of the form}\\
&\; \{M^{\ell_N - \ell} L_1 \, z + [0,M^{\ell_N - \ell} L_1)^d\} \cap \IZ^d, \;\mbox{where} \; z \in \IZ^d.
\end{split}
\end{equation}
Thus, the collections $\cI_\ell$, $0 \le \ell \le \ell_N$, are naturally nested, $\cI_{\ell_N}$ corresponds to the collection of $B_1$-boxes and $\cI_0$ to boxes of approximate size $N$.

%\medskip
%\pagebreak
Given $\ell$ as above and $B \in \cI_\ell$, the ``tower above $B$'' stands for the collection of $B' \in \bigcup_{0 \le \ell' \le \ell} \cI_{\ell '}$ such that $B' \supseteq B$. We also denote by
\begin{equation}\label{3.9}
\begin{array}{l}
\mbox{$\ov{D}_N$ the union of boxes in $\cI_0$ that intersect $D_N$, so that}
\\
D_N = [-N,N]^d \subseteq \ov{D}_N \subseteq [-2N, 2N]^d \; \mbox{and} \; |D_N| \le |\ov{D}_N| \le 2^d \,|D_N| .
\end{array}
\end{equation}
Further, given $0 \le \ell \le \ell_N$, we set
\begin{equation}\label{3.10}
\mbox{$\ov{\cI}_\ell =$ the collection of boxes in $\cI_\ell$ that are contained in $\ov{D}_N$}.
\end{equation}

\n
We will now give a brief description of the main steps of the proof of Theorem \ref{theo3.1}. The random set $C_\omega$ will be extracted from the $B_0$-boundary $\partial_{B_0} \cU_1$ of the random set $\cU_1$ in (\ref{1.40}). We only need to consider the case when the bubble set has volume at least $\ve \,|D_N|$. We then distinguish between the (easy) case when for some $B_1$ in the bubble set, the box of top size in the tower above $B_1$ has a non-degenerate fraction of its volume occupied by $\cU^c_1$, and the case when no such $B_1$ exists. In the first case, both $\cU_1$ and $\cU^c_1$ occupy a non-degenerate fraction of the volume of $[-4N,4N]^d$. Then, the isoperimetric controls of \cite{DeusPisz96} together with Lemmas \ref{lem1.1} and \ref{lem1.2}, and the rarity of $(\alpha,\beta,\gamma)$-bad boxes on the event $\cB^c_N$, ensure a rather straightforward construction of $C_\omega$.

\medskip
In the second case, which is more delicate, we cover the bubble set by a collection of pairwise disjoint maximal $M$-adic boxes $B'_j$, $1 \le j \le m$, in which both $\cU_1$, and $\cU_1^c$ occupy a non-degenerate fraction of volume. We discard the boxes where too many $(\alpha, \beta, \gamma)$-bad $B_0$-boxes are present and may spoil the number of columns of $B_0$-boxes in the box that only contains $(\alpha, \beta, \gamma)$-good boxes. The bad boxes $B'_j$ that we discard occupy a small fraction of the total volume of the $B'_j$, $1 \le j \le m$. However, the remaining $B'_j$ may still have too high complexity for the type of coarse grained set we are aiming for. We thus use some elements of the {\it method of enlargement of obstacles}, see Chapter 4 of \cite{Szni98a}. We introduce a notion of rarefied boxes, where in each box in the tower above the given box the presence of the good $B'_j$ is sparse (i.e.~little felt by simple random walk). We show that rarefied boxes of depth bigger than $k$ (where $k$ solely depends on the dimension and $\ve$) occupy a small volume. We are then reduced to boxes of depth at most $k$ where somewhere in the (short) tower above them the good $B'_j$ are felt. Combined with Lemmas \ref{lem1.1} and \ref{lem1.2}, this permits to construct the coarse grained set $C_\omega$ satisfying (\ref{3.5}).

\bigskip\n
{\it Proof of Theorem \ref{theo3.1}:} Without loss of generality, we assume that
\begin{equation}\label{3.11}
|B u b| \ge \ve \, |D_N|
\end{equation}
(on the complement in $\cB^c_N$ of this event we simply choose $C_\omega = \phi$, so that on the complement in $\cB^c_N$ of this event (\ref{3.5}) holds). We also assume that $N$ is large enough so that (see (\ref{1.46}))
\begin{equation}\label{3.12}
|{\rm Deep} \, B_1| \ge \mbox{\f $\dis\frac{3}{4}$} \;|B_1|.
\end{equation}

\n
Given $B_1 \subseteq B u b$ (that is $B_1 \subseteq D_N$ such that ${\rm Deep} \,B_1 \cap \cU_1 = \phi$, see (\ref{1.47})), we consider the boxes $B$ in the tower above $B_1$ such that
\begin{equation}\label{3.13}
|B \cap \cU^c_1| \ge \mbox{\f $\dis\frac{1}{2}$} \;|B|,
\end{equation}
and note that due to (\ref{3.12}) and ${\rm Deep} \, B_1 \subseteq \cU^c_1$, $B_1$ belongs to this collection. We thus denote by
\begin{equation}\label{3.14}
\mbox{$\ov{B} (B_1)$ the maximal element in this collection.}
\end{equation}
Either we are on the event
\begin{equation}\label{3.15}
\mbox{for some $B_1 \subseteq B u b, \;\ov{B}(B_1) \in \cI_0$ (intersected with $\cB^c_N \cap \{|B u b| \ge \ve \, |D_N|\})$},
\end{equation}
or we are on the complement of this event in $\cB^c_N \cap \{ |B u b| \ge \ve\, |D_N|\}$.

\medskip
We first treat the easier case when (\ref{3.15}) occurs. By definition of $\cU_1$, see (\ref{1.40}), $\cU_1 \supseteq ([-3N - L_0, 3 N + L_0]^d)^c$, so that for large $N$ on the event in (\ref{3.15}) both $\cU_1$ and $\cU_1^c$ occupy a non-degenerate fraction of volume in $[-4N, 4N]^d$. By the isoperimetric controls (A.3) - (A.6), p.~480-481 of \cite{DeusPisz96}, there is a projection $\pi$ on one of the coordinate hyperplanes such that $\pi (\partial \cU_1 \cap [-4N,4N]^d) \ge c\,N^{d-1}$, and hence at least $c \,(\frac{N}{L_0})^{d-1}$ $B_0$-boxes in $\partial_{B_0} \, \cU_1 \cap [-4N, 4N]^d$ having distinct $\pi$-projection (see below (\ref{1.41}) for notation).

\medskip
By definition of $\cB_N$ in (\ref{1.42}), for large $N$ on the event (\ref{3.15}), since $(\frac{N}{L_0})^{d-1} \sim N^{d-2} \gg \rho(L_0) \,N^{d-2}$ as $N \r \infty$, we can find $[c'(\frac{N}{L_0})^{d-1}] \,B_0$-boxes with distinct $\pi$-projections in $[-4N, 4N]^d$ that are all $(\alpha, \beta, \gamma)$-good and in $\partial_{B_0} \,\cU_1$, and hence such that $N_u(D_0) \ge \beta \,{\rm cap} (D_0)$.

\medskip
By the combination of Lemmas \ref{lem1.1} and \ref{lem1.2} we can extract a subcollection of these $B_0$-boxes such that their $\pi$-projections are at mutual distance at least $\ov{K} L_0$, the capacity of their union at least $\wh{c} \,N^{d-2}$, and their number $[\ov{c} (\frac{N}{L_0})^{d-2}]$. We denote by $C_\omega$ the union of these $B_0$-boxes (we use some deterministic ordering to select $C_\omega$ if there are several such collections). Then, for large $N$, we see that
\begin{equation}\label{3.16}
\mbox{the $2 \ov{K} L_1$-neighborhood of $C_\omega$ has volume at most $c \ov{K} L_1 \, N^{d-1} \le \ve\, |D_N|$,}
\end{equation}
and since ${\rm cap}(C_\omega) \ge \wh{c} \,N^{d-2}$ it also follows that
\begin{equation}\label{3.17}
h_{C_\omega} (x) \ge \wt{c} \;\; \mbox{on} \; \;[-4N, 4N]^d \supseteq B u b.
\end{equation}
In addition, the number of possible shapes for the random set $C_\omega$ constructed on the event in (\ref{3.15}) is at most
\begin{equation}\label{3.18}
\big\{ c \big(\mbox{\f $\dis\frac{N}{L_0}$} \big)^d \big\}^{\ov{c}(\frac{N}{L_0})^{d-2}} = \exp\big\{ \ov{c} \log \big(c \mbox{\f $\dis\frac{N}{L_0}$} \big) \big(\mbox{\f $\dis\frac{N}{L_0}$} \big)^{d-2}\big\} = \exp \{o(N^{d-2})\}.
\end{equation}
This shows that for large $N$ the above constructed $C_\omega$ on the event in (\ref{3.15}) satisfies all conditions of (\ref{3.5}).

\medskip
We now turn to the more delicate task of constructing $C_\omega$ on the complement in $\cB^c_N \cap \{| B u b| \ge \ve \, |D_N|\}$ of the event in (\ref{3.15}). We thus assume that none of the $\ov{B}(B_1)$, with $B_1 \subseteq B u b$ has maximal size, that is, we consider the event
\begin{equation}\label{3.19}
\mbox{for all $B_1 \subseteq B u b, \ov{B}(B_1) \in \bigcup\limits_{1 \le \ell \le \ell_N} \,\ov{\cI}_\ell$ (intersected with $\cB^c_N \cap \{|B u b | \ge \ve\, |D_N|\}$)}.
\end{equation}
Then, for each $B_1 \subseteq B u b$, we define $B'(B_1)$ the box immediately above $\ov{B}(B_1)$ in the tower above $B_1$, so that on the event in (\ref{3.19})
\begin{equation}\label{3.20}
\mbox{for all $B_1 \subseteq B u b, \; B_1 \subseteq \ov{B}(B_1) \varsubsetneq B'(B_1)$ with $B'(B_1) \in \bigcup\limits_{0 \le \ell  < \ell_N} \,\ov{\cI}_\ell$}.
\end{equation}
Thus, by (\ref{3.13}), (\ref{3.14}), since $B'(B_1)$ does not satisfy (\ref{3.13}), we have (see (\ref{3.8}))
\begin{equation*}
\dis\frac{1}{2M^d} \;|B'(B_1)| \le \mbox{\f $\dis\frac{1}{2}$} \;|\ov{B}(B_1)| \le |\ov{B}(B_1) \cap \cU_1^c| \le |B'(B_1) \cap \cU^c_1| \le \mbox{\f $\dis\frac{1}{2}$} \;|B'(B_1)|.
\end{equation*}
Hence, on the event in (\ref{3.19}) we find that
\begin{equation}\label{3.21}
\mbox{for all $B_1 \subseteq B u b, \; \dis\frac{1}{2M^d} \;|B'(B_1)| \le |B'(B_1) \cap \cU^c_1| \le \mbox{\f $\dis\frac{1}{2}$} \;|B'(B_1)|$}.
\end{equation}
By construction, any two sets $B'(B_1)$, with $B_1 \subseteq B u b$, are either pairwise disjoint, or one contains the other. We then denote by
\begin{equation}\label{3.22}
\mbox{$B'_1, \dots, B'_m$ the maximal elements for inclusion in the collection of $B'(B_1)$, $B_1 \subseteq B u b$}
\end{equation}
(both $m$ and the labelling possibly depend on $\omega$ in the event in (\ref{3.19})). Thus, on the event in (\ref{3.19}) we find that:
\begin{equation}\label{3.23}
\mbox{the $B'_j$, $1 \le j \le m$, are pairwise disjoint and $B u b \subseteq \bigcup\limits^m_{j = 1} \,B'_j \subseteq \ov{D}_N$}.
\end{equation}

By (\ref{3.21}) both $\cU_1$ and $\cU^c_1$ occupy a non-vanishing fraction of volume in each $B'_j$, $1 \le j \le m$. By the isoperimetric controls (A.3) - (A.6), p.~480-481 of \cite{DeusPisz96}, for each $B'_j$, $1 \le j \le m$, we can find a projection $\pi'_j$ on one of the coordinate hyperplanes so that (recall that $M$ is a dimension dependent constant)
\begin{equation}\label{3.24}
\begin{array}{l}
\mbox{there are $c_4 \Big(\frac{|B'_j|}{|B_0|}\Big)^{\frac{d-1}{d}}$ columns in the $\pi'_j$-direction inside $B'_j$ that}
\\
\mbox{contain a box of $\partial_{B_0} \,\cU_1$.}
\end{array}
\end{equation}
Now for $1 \le j \le m$, we say that
\begin{equation}\label{3.25}
\begin{array}{l}
\mbox{``$j$ is bad'' if $B'_j$ contains more than $\frac{1}{2} \;c_4 \Big(\frac{|B'_1|}{|B_0|}\Big)^{\frac{d-1}{d}} (\alpha,\beta,\gamma)$-bad $B_0$-boxes and that}
\\
\mbox{``$j$ is good'' otherwise}.
\end{array}
\end{equation}
We write $\cG = \{1 \le j \le m$; $j$ is good$\}$ for the set of good $j$ in $\{1,\dots,m\}$, and we set
\begin{equation}\label{3.26}
a'_j = |B'_j| / \dsl^m_{k=1} \;|B'_k|, \; \mbox{for $1 \le j \le m$}.
\end{equation}
On the event in (\ref{3.19}) (which is contained in $\cB^c_N$, see (\ref{1.42}) for its definition), we have
\begin{equation*}
\begin{array}{l}
\mbox{\f $\dis\frac{1}{2}$} \; c_4 \,|B u b |^{\frac{d-1}{d}} \dsl_{j \, {\rm bad}} a'\,^{\frac{d-1}{d}}_j  \; \stackrel{(\ref{3.23}), (\ref{3.26})}{\le} \mbox{\f $\dis\frac{1}{2}$} \; c_4 \;\dsl_{j \, {\rm bad}} \;|B'_j|^{\frac{d-1}{d}} \stackrel{(\ref{3.25})}{\le} 
\\
\dsl_{j \, {\rm bad}} \; \dsl_{B_0 \subseteq B'_j} \;|B_0|^{\frac{d-1}{d}} \;\mbox{$1\{B_0$ is $(\alpha, \beta, \gamma)$-bad$\}$} \le \mbox{(since $B'_j \subseteq [-2N, 2N]^d$, see  (\ref{3.20}))} 
\\
L_0^{d-1} \;\dsl_{B_0 \subseteq [-2N,2N]^d} \; \mbox{$1\{B_0$ is $(\alpha,\beta,\gamma)$-bad$\}$} \stackrel{(\ref{1.42})}{\le} L^{d-1}_0 \l \rho(L_0) \,N^{d-2}.
\end{array}
\end{equation*}

As a result we see that on the event in (\ref{3.19})
\begin{equation}\label{3.27}
\begin{array}{l}
\big(\dsl_{j \, {\rm bad}} a'_j\big)^{\frac{d-1}{d}}  \le \dsl_{j \, {\rm bad}}  a'_j\,\!\!^{\frac{d-1}{d}}    \le \mbox{\f $\dis\frac{2}{c_4}$} \; L^{d-1}_0 \; N^{d-2} \,\rho(L_0) \, |B u b |^{-\frac{d-1}{d}}
\\[2ex]
\stackrel{(\ref{1.7}),(\ref{3.11})}{\le} \mbox{\f $\dis\frac{2}{c_4}$}  \;\rho(L_0) \;\Big(\dis\frac{N^d}{\ve\, |D_N|}\Big)^{\frac{d-1}{d}} \underset{N}{\longrightarrow} 0.
\end{array}
\end{equation}

\n
Our goal is to construct a coarse grained $C_\omega$ of low complexity, see (\ref{3.5}) iii), and from this perspective the description of the $B'_j, j \in \cG$, may still involve too many small grains. We will now aggregate the most part of the $B'_j, j \in \cG$, inside large (nearly macroscopic) boxes that will feel the presence of the $B'_j, j \in \cG$, which they contain, and hence, see (\ref{3.24}) - (\ref{3.25}), the presence of $(\alpha, \beta, \gamma)$-good boxes $B_0$ such that $N_u(D_0) \ge \beta \,{\rm cap}(D_0)$. This step will have some flavor of the ``method of enlargement of obstacles'', see Chapter 4 of \cite{Szni98a}, although in a simplified form.

\medskip
We introduce the dimension dependent constant (recall $c_*$ from (\ref{1.3})):
\begin{equation}\label{3.28}
\eta = (c_* \,M^d)^{-1} \wedge \Big\{\inf\limits_{L \ge 1} \; \dis\frac{{\rm cap}([0,L)^d)}{L^{d-2}}\Big\} > 0.
\end{equation}
We then say that an $M$-adic box $B$ in $\bigcup_{0 \le \ell \le \ell_N} \,\ov{\cI}_\ell$ (see (\ref{3.10})) is {\it sparse} if 
\begin{equation}\label{3.29}
{\rm cap}\big(B \cap \big(\bigcup\limits_{j\in \cG} \,B'_j\big)\big) < \eta \, |B|^{\frac{d-2}{d}},
\end{equation}
and otherwise {\it non-sparse}. Note that each $B = B'_j, j \in \cG$ satisfies ${\rm cap}(B'_j) \ge \eta \,|B'_j|^{\frac{d-2}{d}}$, so that
\begin{equation}\label{3.30}
\mbox{$B'_j$ is non-sparse for each $j \in \cG$.}
\end{equation}
Then, for each $j \in \cG$, we consider the tower of $M$-adic boxes above $B'_j$ and set 
\begin{equation}\label{3.31}
\mbox{$\wt{B}(B'_j) =$ the largest non-sparse box in the tower above $B'_j$.}
\end{equation}
Again, by construction (see (\ref{3.8})), we find that
\begin{equation}\label{3.32}
\mbox{the boxes $\wt{B}(B'_j)$, as $j$ varies over $\cG$, are either pairwise disjoint or equal.}
\end{equation}
We then say that
\begin{equation}\label{3.33}
\begin{array}{l}
\mbox{an $M$-adic box $B$ in $\bigcup\limits_{0 \le \ell\le \ell_N} \; \cI_\ell$ is {\it rarefied} if $B$ and the boxes in the}
\\
\mbox{tower above $B$ are sparse (see (\ref{3.29}))}.
\end{array}
\end{equation}
Note that as a consequence of (\ref{3.30})
\begin{equation}\label{3.34}
\mbox{when $B$ is rarefied and $B'_j \cap B \not= \phi$ for some $j \in \cG$, then $B'_j  \varsubsetneq B$.}
\end{equation}
Our goal in Lemma \ref{lem3.2} below is to show that the volume of good boxes $B'_j, j \in \cG$, contained in the union of rarefied boxes of depth $k$ decreases geometrically in $k$. The proof has some flavor (although in a somewhat simplified version) of the capacity and volume estimates in the {\it method of enlargement of obstacles} (see Chapter 4 \S 3 of \cite{Szni98a}). We recall the notation (\ref{3.10}).

\begin{lemma}\label{lem3.2}
For all $0 \le k \le \ell_N$, 
\begin{equation}\label{3.35}
\dsl_{B\in \ov{\cI}_k, {\rm rarefied}} \big| B \cap \big(\bigcup\limits_{j \in \cG} \,B'_j\big)\big| \le c_5 \,|D_N| \big(\dis\frac{M^2}{3^d + 1}\big)^{-k}.
\end{equation}
\end{lemma}
\begin{proof}
We recall the Green operator $G$ from (\ref{1.2}). Note that when $B$ is a box and $A \subseteq B$, then $G 1_A \le G 1_B \le c \,|B|^{\frac{2}{d}}$, see (\ref{1.5}), so that $G 1_A \le c \,|B|^{\frac{2}{d}} h_A$ on $A$. Integrating this last inequality, with respect to $e_A$, see above (\ref{1.4}), we find that
\begin{equation}\label{3.36}
c_6 \; \dis\frac{{\rm cap}(A)}{|B|^{\frac{d-2}{d}}} \ge \dis\frac{|A|}{|B|}, \;\mbox{for all $A \subseteq B$, with $B$ a box in $\IZ^d$.}
\end{equation}
We will now bound the volume in the left member of (\ref{3.35}) in terms of its capacity with the help of the above inequality, see (\ref{3.37}) below. The case $k = 0$ in (\ref{3.35}) being immediate to handle (the left member is at most $|\ov{D}_N| \le 2^d \, |D_N|$, see (\ref{3.9})), we assume that $1 \le k \le \ell_N$. We note that for each $B \in \ov{\cI}_k$, we have $|B| \le M^{-kd} |\ov{D}_N|$, and hence
\begin{equation*}
\dsl_{B \in \ov{\cI}_k,{\rm rarefied}} \big|B \cap \big(\bigcup\limits_{j \in \cG}\, B'_j\big)\big| \le \dis\frac{|\ov{D}_N|}{M^{kd}} \; \dsl_{B\, \in \ov{\cI}_k,{\rm rarefied}} \; |B \cap \big(\bigcup\limits_{j \in \cG} B'_j\big)\big| / |B|.\;.
\end{equation*}
Thus, using (\ref{3.36}) with $A = B \cap (\bigcup_{j \in \cG} \,B'_j)$, we find that for $1 \le k \le \ell_N$:
\begin{equation}\label{3.37}
\dsl_{B \in \ov{\cI}_k,{\rm rarefied}} \big|B \cap \big(\bigcup\limits_{j \in \cG}\, B'_j\big)\big| \le c_6 \, \dis\frac{|\ov{D}_N|}{M^{kd}} \; \dsl_{B \in \ov{\cI}_k,{\rm rarefied}} \dis\frac{{\rm cap}(B \cap (\bigcup_{j \in \cG}\, B'_j))|}{|B|^{\frac{d-2}{d}}} \;.
\end{equation}

\n
We will now establish an induction over scales in order to control the sum in the right member of (\ref{3.37}). For this purpose we consider $0 \le \ell < \ell_N$ and some $\ov{B} \in \ov{\cI}_\ell$ and the boxes $\wh{B} \in \ov{\cI}_{\ell + 1}$ contained in $\ov{B}$. The bound in (\ref{3.38}) below has a similar flavor (but is simpler, because we do not need truncation due to the more restrictive notion of rarefied boxes that we use here) to Lemma 3.2 on p.~170 in Chapter 4 \S 3 of \cite{Szni98a}. Our aim is to show that

\begin{equation}\label{3.38}
 \dis\frac{{\rm cap}(\ov{B} \cap (\bigcup_{j \in \cG}\, B'_j))}{|\ov{B}|^{\frac{d-2}{d}}} \ge \dis\frac{M^2}{3^d + 1} \; \dis\frac{1}{M^d} \; \dsl_{\wh{B} \subseteq \ov{B}, \wh{B} \, {\rm sparse}}  \; \dis\frac{{\rm cap}(\wh{B} \cap (\bigcup_{j \in \cG}\, B'_j))}{|{\wh{B}}|^{\frac{d-2}{d}}}.
\end{equation}
This inequality reflects a nearly additive regime of the capacity (up to the multiplicative factor $1/(3^d+ 1)$) when dealing with sparse subboxes of a box of the next scale. We will then iterate this basic control over scales corresponding to $\ell$ ranging from $k-1$ to $0$, see (\ref{3.42}) below. For the time being we prove (\ref{3.38}). To this end we introduce the measure (see below (\ref{1.3}) for notation)
\begin{equation*}
\nu = \dsl_{\wh{B} \subseteq \ov{B}, \wh{B}\,{\rm sparse}} \; e_{\wh{B} \cap (\bigcup\limits_{j \in \cG} \,B'_j)},
\end{equation*}
and note that for each $x \in \bigcup_{\wh{B} \subseteq \ov{B}, \wh{B}\,{\rm sparse}} \wh{B} \cap (\bigcup_{j \in \cG} \,B'_j)$, denoting by $\Sigma_1$ and $\Sigma_2$ the respective sums over the sparse $\wh{B} \subseteq \ov{B}$, with $\wh{B}$ containing $x$, or being a neighbor of the box in $\ov{\cI}_{\ell + 1}$, containing $x$ for $\Sigma_1$, and for $\Sigma_2$ the sum over the remaining sparse $\wh{B} \subseteq \ov{B}$, we have
\begin{equation}\label{3.39}
\begin{split}
\dsl_y \;g(x,y) \,\nu(y) & \le \underset{\!\!\!\wh{B}}{\Sigma_1} \; \dsl_{y \in \wh{B}} g(x,y) \;e_{\wh{B} \cap (\bigcup_{j \in \cG} B'_j)} (y)
\\[1ex]
&+ \; \underset{\!\!\!\wh{B}}{\Sigma_2} \;  \dsl_{y \in \wh{B}} g(x,y) \; e_{\wh{B} \cap (\bigcup_{j \in \cG} B'_j)} (y)
\\
&\!\!\!\! \stackrel{(\ref{1.3})}{\le} 3^d + \underset{\!\!\!\!\wh{B}}{\Sigma_2} \; \dsl_{y \in \wh{B}} \; \dis\frac{c_*}{|\wh{B}|^{\frac{d-2}{d}}} \;{\rm cap}\big(\wh{B} \cap (\bigcup_{j \in \cG}\,B'_j \big)\big)
\\
&\mbox{and since the $\wh{B}$ in the sum $\Sigma_2$ are sparse, see (\ref{3.29}),}
\\[-1ex]
&\mbox{and there are at most $M^d$ such $\wh{B}$} 
\\
&  \le 3^d + c_* \, \eta \,M^d \stackrel{(\ref{3.28})}{\le}  3^d + 1.
\end{split}
\end{equation}
Noting that $\nu$ is supported by the set $S = \bigcup_{\wh{B} \subseteq \ov{B}, \wh{B}\,{\rm sparse}} \,\wh{B} \cap (\bigcup_{j \in \cG} \,B'_j)$, we find that
\begin{equation}\label{3.40}
(3^d + 1)^{-1} \,G \nu \le h_S,
\end{equation}
and integrating this inequality with respect to $e_S$, we find that
\begin{equation}\label{3.41}
\begin{split}
{\rm cap}\big(\ov{B} \cap \big(\bigcup_{j \in \cG} \,B'_j\big)\big) & \ge {\rm cap} (S)
\\[-2ex]
& \ge \dis\frac{1}{3^d + 1} \; \nu(\IZ^d) = \dis\frac{1}{3^d + 1} \; \dsl_{\wh{B} \subseteq \ov{B}, \wh{B}\,{\rm sparse}} \; {\rm cap} \big(\wh{B} \cap \big(\bigcup\limits_{j \in \cG} \,B'_j\big)\big).
\end{split}
\end{equation}
Since $|\ov{B}| = M^d \,|\wh{B}|$, dividing both members of (\ref{3.41}) by $|\ov{B}|^{\frac{d-2}{d}}$ the inequality (\ref{3.38}) follows.

\medskip
%\pagebreak
We will now apply (\ref{3.38}) inductively to bound the right member of (\ref{3.37}). We thus find that
\begin{equation}\label{3.42}
\begin{array}{l}
c_6 \; \dis\frac{|\ov{D}_N|}{M^{dk}} \; \dsl_{B \in \ov{\cI}_k,{\rm rarefied}} \;\dis\frac{{\rm cap}(B \cap (\bigcup_{j \in \cG} B'_j))}{|B|^{\frac{d-2}{d}}} \; \stackrel{(\ref{3.38})}{\le}
\\
\\
c_6 \; \dis\frac{|\ov{D}_N|}{M^{d(k-1)}} \; \Big(\dis\frac{M^2}{3^d + 1}\Big)^{-1} \; \dsl_{B \in \ov{\cI}_{k-1},{\rm rarefied}} \; \dis\frac{{\rm cap}(B \cap (\bigcup_{j \in \cG} B'_j))}{|B|^{\frac{d-2}{d}}} \; \stackrel{\rm induction}{\le}
\\
\\
c_6 \; | \ov{D}_N | \;  \Big(\dis\frac{M^2}{3^d + 1}\Big)^{-k}  \; \dsl_{B \in \ov{\cI}_0,{\rm rarefied}} \;\dis\frac{{\rm cap}(B)}{|B|^{\frac{d-2}{d}}} \le c_5 \,|D_N| \;  \Big(\dis\frac{M^2}{3^d + 1}\Big)^{-k}.
\end{array}
\end{equation}
This inequality combined with (\ref{3.37}) completes the proof of Lemma 3.2.
\end{proof}

As an aside, the notion of rarefied box that we use here (where every box in the tower above a given box is sparse) is more primitive than the notion used in Chapter 4 \S 3 of \cite{Szni98a}. However, this feature permits the use of (\ref{3.38}) that does not require truncation in the right member, and is simpler to iterate than the inequality in Lemma 3.2 on p.~170 of \cite{Szni98a}, see also Lemma 3.4 on p.~173 and (3.38) on p.~175 of the same reference.

\medskip
We now specify our choice of $k(\ve)$ (we recall that $M$ is a dimension dependent constant, see (\ref{3.6})) through
\begin{equation}\label{3.43}
c_5 \;\Big( \dis\frac{M^2}{3^d + 1}\Big)^{-k} \le \mbox{\f $\dis\frac{1}{2}$} \; \ve.
\end{equation}
Thus, when $N$ is large enough so that $\ell_N \ge k$, see (\ref{3.7}), on the event in (\ref{3.19}) we have by (\ref{3.35}) and (\ref{3.43})
\begin{equation}\label{3.44}
\dsl_{B \in \ov{\cI}_k,{\rm rarefied}} \; \big|B \cap \big(\bigcup\limits_{j \in \cG} \,B'_j\big) \big|\le \mbox{\f $\dis\frac{\ve}{2}$} \; |D_N|.
\end{equation}
We now introduce the set of {\it very good} $j$ (recall the definition of $\cG$ below (\ref{3.25})):
\begin{equation}\label{3.45}
V \cG = \big\{j \in \cG; \; \wt{B}(B'_j) \in \bigcup\limits_{0 \le \ell \le k} \;\ov{\cI}_\ell\big\},
\end{equation}
in other words, $j$ is very good if it is good and some box of depth at most $k$ in the tower above $B'_j$ is non-sparse.

\medskip
As we now explain, the volume of the $B'_j$ that are good but not very good is small. Indeed, when $\ell_N \le k$, $\cG = V \cG$ by (\ref{3.30}), and when $\ell_N > k$, one has
\begin{equation}\label{3.46}
\dsl_{j \in \cG \backslash V \cG} |B'_j| \le \dsl_{B \in \ov{\cI}_k,{\rm rarefied}} \big| B \cap \big(\bigcup\limits_{j \in \cG} B'_j\big)\big| \stackrel{(\ref{3.44})}{\le} \mbox{\f $\dis\frac{\ve}{2}$} \; |D_N|.
\end{equation}
In addition, by (\ref{3.26}), (\ref{3.27}), we have
\begin{equation}\label{3.47}
\dsl_{j \notin \cG} |B'_j| \le \big(\dsl_{j=1}^m  |B'_j|\big) \; \mbox{\f $\dis\frac{c'}{\ve}$} \; \rho (L_0)^{\frac{d}{d-1}} \stackrel{(\ref{3.23})}{\le}  \mbox{\f $\dis\frac{c}{\ve}$} \; |D_N| \;\rho(L_0)^{\frac{d}{d-1}}.
\end{equation}
As a result, for large $N$ on the event in (\ref{3.19}), we have
\begin{align}
& \bigcup\limits_{j=1}^m \;B'_j \supseteq B u b, \; \mbox{and} \label{3.48}
\\
& \dsl_{j \notin V\cG} \;|B'_j| \le \mbox{\f $\dis\frac{\ve}{2}$} \; |D_N| + \mbox{\f $\dis\frac{c}{\ve}$} \;|D_N| \,\rho (L_0)^{\frac{d}{d-1}}. \label{3.49}
\end{align}
Note that (see (\ref{3.31}) for notation) the $\wt{B}(B'_j)$ for $j \in V \cG$ belong to $\bigcup_{0 \le \ell \le k} \ov{\cI}_\ell$, and are either pairwise disjoint or equal. Using a total  order on $\bigcup_{0 \le \ell \le k}  \ov{\cI}_\ell$ preserving depth we can label the $\wt{B}(B'_j), j \in V \cG$, as $\wt{B}_1,\dots , \wt{B}_p$, so that for large $N$ on the event in (\ref{3.19}) we have
\begin{equation}\label{3.50}
\mbox{$\wt{B}_1,\dots , \wt{B}_p$ are pairwise disjoint boxes in $\bigcup_{0 \le \ell \le k} \ov{\cI}_\ell$ covering $\bigcup_{j \in V \cG} B'_j$,}
\end{equation}
and setting $\wt{L}_i = |\wt{B}_i |^{\frac{1}{d}}$, for $1 \le i \le p$,
\begin{align}
&\wt{L}_1 \ge \wt{L}_2 \ge \dots \ge \wt{L}_p, \label{3.51}
\\
& \mbox{for $j \in \cG$ and $1 \le i \le p$ when $B'_j \cap \wt{B}_i \not= \phi$, then  $j \in V \cG$ and $B'_j \subseteq \wt{B}_i$ (due to (\ref{3.45}))}, \label{3.52}
\\
& {\rm cap} \big(\wt{B}_i \cap \big(\bigcup\limits_{j \in \cG} B'_j\big)\big) \ge \eta \, |\wt{B}_i |^{\frac{d-2}{d}}, \; \mbox{for $1 \le i \le p$ (due to (\ref{3.31}))}. \label{3.53}
\end{align}
In essence, the construction of the random set $C_\omega$ on the event in (\ref{3.19}) will proceed as follows. From the above collection $\wt{B}_i$, $1 \le i \le p$, we will retain a sizeable sub-collection of non-adjacent boxes and in such boxes we will retain a sizeable sub-collection of good boxes $B'_j$. Then, using Lemma \ref{lem1.1}, each such box $B'_j$ for a suitable projection $\pi'_j$ will contain a number of order $(\frac{|B'_j|}{|B_0|})^{\frac{d-1}{d}}$ of $(\alpha, \beta, \gamma)$-good boxes $B_0$ with $N_u (D_0) \ge \beta \, {\rm cap}(D_0)$  with $\pi'_j$ projections at mutual distance at least $\ov{K} L_0$ and union having a capacity comparable to that of $B'_j$. With Lemma \ref{lem3.2} from each such retained $\wt{B}_i$ we will select among these $B_0$-boxes a number of order $(\frac{|\wt{B}_i|}{|B_0|})^{\frac{d-2}{d}}$ of boxes with union having a capacity comparable to $\wt{B}_i$. The set $C_\omega$ will in essence correspond to the union of these $B_0$-boxes.

\bigskip
More precisely, for large $N$, on the event (\ref{3.19}), given the boxes $\wt{B}_1,\dots, \wt{B}_p$ with respective sizes $\wt{L}_1 \ge \dots \ge \wt{L}_p$ we construct an attachment map $\Phi$: $\{1,\dots, p\} \rightarrow \{1, \dots , p\}$ as follows. The set $\Phi^{-1}(1)$ consists of the labels $i$ such that $\wt{B}_i$ is contained in the $\wt{L}_1$-neighborhood for the sup-distance of $\wt{B}_1$, $\Phi^{-1} (2)$ is empty if $2 \in \Phi^{-1} (1)$, and otherwise consists of the labels $i \in \{1,\dots, p\} \backslash \Phi^{-1}(1)$ such that $\wt{B}_i$ is contained in the $\wt{L}_2$-neighborhood of $\wt{B}_2$, $\Phi^{-1}(3)$ is empty if $3 \in \Phi^{-1}(1) \cup \Phi^{-1}(2)$, and otherwise consists of the $i \in \{1,\dots, p\} \backslash (\Phi^{-1} (1) \cup \Phi^{-1}(2)$) such that $\wt{B}_i$ is contained in the $\wt{L}_3$-neighborhood of $\wt{B}_3$, and so on until the process terminates.

\medskip
In this fashion we can make sure that
\begin{align}
&\Phi \circ \Phi = \Phi, \label{3.54}
\intertext{and when $i \in {\rm range} \, \Phi$, then for $i' \in \{1, \dots, p\} \backslash \{i\}$}
& \left\{ \begin{array}{rl}
{\rm a)} & \mbox{$\Phi(i') = i$ implies that $\wt{B}_{i'}$ is contained in the $\wt{L}_i$-neighborhood of $\wt{B}_i$}, \label{3.55}
\\[1ex]
{\rm b)} & \mbox{$i' \in {\rm range}\, \Phi$ implies that $d_\infty(\wt{B}_i, \wt{B}_{i'}) \ge \max \{\wt{L}_i, \wt{L}_{i'}\}$}. 
\end{array}\right.
\end{align}
Now for any $i \in {\rm range} \, \Phi$ we consider the $B'_j, j \in \cG$ that are contained in $\wt{B}_i$ (such $j$ in fact belong to $V \cG$, and the union of such $B'_j$ has a sizeable presence in $\wt{B}_i$, see (\ref{3.52}), (\ref{3.53})). We apply a similar procedure as above to the collection $\{j \in \cG; B'_j \subseteq \wt{B}_i\}$, and produce an attachment map $\Phi^i$: $\{j \in \cG; B'_j \subseteq \wt{B}_i\} \longrightarrow \{ j \in \cG; B'_j \subseteq \wt{B}_i\}$ so that
\begin{align}
&\Phi^i \circ \Phi^i = \Phi^i \label{3.56}
\intertext{and for each $j_1 \in {\rm range} \, \Phi^i$ and $j_2 \in \{j \in \cG$; $B'_j \subseteq \wt{B}_i \}  \backslash \{j_1\}$}
& \left\{ \begin{array}{rl}
{\rm a)} & \mbox{$\Phi^i (j_2) = j_1$ implies that $B'_{j_2}$ is contained in the $|B'_{j_1}|^{\frac{1}{d}}$-neighborhood of $B'_{j_1}$}, \label{3.57}
\\[1ex]
{\rm b)} & \mbox{$j_2 \in {\rm range}\, \Phi^i$ implies that $d_\infty(B'_{j_1}, B'_{j_2}) \ge \max \{|B'_{j_1}|^{\frac{1}{d}}, |B'_{j_2}|^{\frac{1}{d}}\}$}. 
\end{array}\right.
\end{align}

\medskip\n
Now for each $j \in {\rm range} \, \Phi^i$ the box $B'_j$ is such that, see (\ref{3.24}), (\ref{3.25}), there is a coordinate projection $\pi'_j$ and at least $\frac{1}{2} \;c_4 (\frac{|B'_j|}{|B_0|})^{\frac{d-1}{d}} (\alpha, \beta, \gamma)$-good boxes $B_0$ in $\partial_{B_0} \, \cU_1$ with distinct $\pi'_j$-projections. All such boxes are necessarily such that $N_u(D_0) \ge \beta \, {\rm cap}(D_0)$ (otherwise they would belong to $\cU_1$, see (\ref{1.40})).

\bigskip
Thus, for large $N$, we can apply Lemma \ref{lem1.1} and for each $j \in {\rm range} \, \Phi^i$ (where $i \in {\rm range} \, \Phi)$ find a sub-collection of $B_0$-boxes with $\pi'_j$-projections which are $\ov{K} L_0$-distant, all $(\alpha, \beta, \gamma)$-good with $N_u(D_0) \ge \beta \, {\rm cap}(D_0)$, their union having capacity at least $c \, |B'_j|^{\frac{d-2}{d}}$. By (\ref{3.53}) and (\ref{3.57}) a) the simple random walk starting in $\wt{B}_i$ reaches one of the $B'_j$, $j \in {\rm range}\, \Phi^i$, with a probability uniformly bounded from below, and hence reaches the above $B_0$-boxes within $B'_j$ that are $(\alpha, \beta,\gamma)$-good with $N_u(D_0) \ge \beta\, {\rm cap}(D_0)$ and with mutually $\ov{K}\,L_0$-distant $\pi'_j$-projections, with a probability uniformly bounded as well. This means that the union of such $B_0$-boxes as $j$ ranges over ${\rm range} \, \Phi^i$ has a capacity at least $\ov{c} \, |\wt{B}_i|^{\frac{d-2}{d}}$ (and these boxes are mutually $\ov{K} L_0$-distant). We can then apply Lemma \ref{lem1.2} and extract a sub-collection of these $B_0$-boxes of at most $\wt{c} (|\wt{B}_i | / |B_0|)^{\frac{d-2}{d}}$ boxes with union having capacity at least $\wt{c} ' \,{\rm cap} (\wt{B}_i)$. This sub-collection necessarily contains at least $[c(|\wt{B}_i| / |B_0|)^{\frac{d-2}{d}}]$ boxes, see (\ref{1.6}), and we can thus find for each $i \in {\rm range}\, \Phi$ a sub-collection of $B_0$-boxes in $\wt{B}_i$ with union denoted by $A_i$, mutually $\ov{K} \,L_0$-distant, all $(\alpha, \beta, \gamma)$-good with $N_u (D_0) \ge \beta \, {\rm cap}(D_0)$, and such that
\begin{equation}\label{3.58}
|A_i| / |B_0| = [c(|\wt{B}_i| / |B_0|)^{\frac{d-2}{d}}] \;\; \mbox{and} \;\; {\rm cap} (A_i) \ge c' {\rm cap} (\wt{B}_i).
\end{equation}
Thus, for large $N$, on the event in (\ref{3.19}) the union $A = \bigcup_{i \in {\rm range} \, \Phi} A_i$ has the property that when starting in $\bigcup_{1 \le i \le p} \wt{B}_i$ a simple random walk has a probability, which is at least $c'_0$, to reach $A$, and by (\ref{3.48}) - (\ref{3.50}), 
\begin{equation}\label{3.59}
\mbox{$h_A \ge c'_0$ on $B u b$ except maybe on a set of volume at most $\ve \, |D_N|$}.
\end{equation}
Note that the $2 \ov{K} L_1$-neighborhood of $A$ has volume at most
\begin{equation}\label{3.60}
c \, \ov{K}^d |B_1| \dsl_{ i \in {\rm range}\, \Phi} \;\dis\frac{|A_i|}{|B_0|} \stackrel{(\ref{3.58})}{\le} c \, \ov{K}^d |B_1|  \,\Big\{\Big(\dis\frac{|\wt{B}_1|}{|B_0|}\Big)^{\frac{d-2}{d}} + \dots + \Big(\dis\frac{|\wt{B}_p|}{|B_0|}\Big)^{\frac{d-2}{d}}\Big\} \le c \, p \, \ov{K}^d |B_1| \;\Big(\dis\frac{|\ov{D}_N|}{|B_0|}\Big)^{\frac{d-2}{d}} .
\end{equation}
Now by (\ref{3.50}) (recall $k(\ve)$ has been chosen in (\ref{3.43}))
\begin{equation}\label{3.61}
p \le |\ov{\cI}_k| \le c\, M^{kd},
\end{equation}
so that the $2 \ov{K}\,L_1$-neighborhood of $A$ has volume at most 
\begin{equation}\label{3.62}
c \, M^{kd}\, \ov{K}^d |B_1|\; \Big(\dis\frac{|D_N|}{|B_0|}\Big)^{\frac{d-2}{d}} \le c \, M^{kd} \,\ov{K}^d \;L^d_1 \;\dis\frac{N^{d-2}}{L_0^{d-2}} \stackrel{(\ref{1.7}), (\ref{1.8})}{\le} \ve \, |D_N| \; \mbox{for large $N$}.
\end{equation}
As for the set of possible shapes of $A$, we note that the $\wt{B}_1, \dots, \wt{B}_p$ are pairwise disjoint and belong to $\bigcup_{0 \le \ell \le k} \ov{\cI}_\ell$, and as already observed $p \le c\,M^{kd}$. There are at most $2^{|\ov{\cI}_k|}$ possible choices for each $\wt{B}_1,\dots, \wt{B}_p$, so that the number of choices for $p$, $\wt{B}_1, \dots \wt{B}_p$ is at most
\begin{equation}\label{3.63}
\dsl_{p \le c\, M^{kd}} \;2^{p|\ov{\cI}_k|} \le c \, M^{kd} \, 2^{c\, M^{2kd}} \le 2^{c' M^{2kd}}.
\end{equation}

\n
The choice of $\wt{B}_1, \dots \wt{B}_p$ determines the allocation map $\Phi$ and hence the range of $\Phi$. For each $\wt{B}_i$, $i \in {\rm range} \, \Phi$, there are $[c(|\wt{B}_i| / |B_0|)^{\frac{d-2}{d}}]$ $B_0$-boxes constituting $A_i$, see (\ref{3.58}), so the number of possibilities for $A_i$, $i \in {\rm range} \, \Phi$, is at most
\begin{equation}\label{3.64}
\begin{array}{l}
\Big(\dis\frac{|\wt{B}_1|}{|B_0|}\Big)^{\big(c\frac{|\wt{B}_1|}{|B_0|}\big)^{\frac{d-2}{d}}} \times \dots \times \Big(\dis\frac{|\wt{B}_p|}{|B_0|}\Big)^{\big(c\frac{|\wt{B}_p|}{|B_0|}\big)^{\frac{d-2}{d}}} \le \exp\Big\{ c \, \log \dis\frac{|\ov{D}_N|}{|B_0|} \; \Big[\Big(\dis\frac{|\wt{B}_1|}{|B_0|}\Big)^{\frac{d-2}{d}} + \dots + \Big(\dis\frac{|\wt{B}_p|}{|B_0|}\Big)^{\frac{d-2}{d}}\Big] \Big\}
\\[2ex]
 \stackrel{p \le c\,M^{kd}}{\le} \exp\Big\{c\,M^{kd} \Big(\log \dis\frac{|\ov{D}_N|}{|B_0|}\Big) \; \Big( \dis\frac{|\ov{D}_N|}{|B_0|}\Big)^{\frac{d-2}{d}}\Big\}.
\end{array}
\end{equation}
Thus, taking into account the number of possible choices for $p$, $\wt{B}_1, \dots \wt{B}_p$, see (\ref{3.63}), and for $A$ with given $p$, $\wt{B}_1,\dots \wt{B}_p$, we find that the number of possible shapes for $A$ is at most:
\begin{equation}\label{3.65}
2^{c' M^{2k}} \exp\Big\{c\,M^{kd} \Big( \log \; \dis\frac{|\ov{D}_N|}{|B_0|}\Big) \; \Big(\dis\frac{|\ov{D}_N|}{|B_0|}\Big)^{\frac{d-2}{d}}\Big\} = \exp\{o(N^{d-2})\}, \; \mbox{as $N \r \infty$}.
\end{equation}
Using a deterministic order on the set of possible shapes, we see that for large $N$ on the event in (\ref{3.19}) we can select a random set $C_\omega$, which for each given $\omega$ coincides with one of the sets $A$ described in (\ref{3.58}) - (\ref{3.59}). Combined with the construction of $C_\omega$ on the event in (\ref{3.15}) (see (\ref{3.16}) - (\ref{3.18})), and on the complement of the union of $\cB_N$ with the event in (\ref{3.11}) (where we set $C_\omega = \phi$), we see that for large $N$ the random set $C_\omega$ is a union of $B_0$-boxes in $[-4N, 4N]^d$ that satisfies (\ref{3.5}) (with $c_0=c'_0 \wedge \wt{c}$). This concludes the proof of Theorem \ref{3.1}. \hfill \mbox{\Large $\square$}

\medskip
\begin{remark}\label{rem3.1} \rm The constant $c_0 \in (0,1)$ that appears in Theorem \ref{theo3.1} plays an important role in the asymptotic upper bounds stated in Theorem \ref{theo4.3} and Corollary \ref{cor4.3} of the next section. One may wonder whether it is possible to choose $c_0$ arbitrarily close to $1$ in Theorem \ref{theo3.1}. We refer to Remark \ref{rem4.4a} below for some of the consequences of a positive answer to this question. \hfill $\square$
\end{remark}

\section{The asymptotic upper bound}

In this section the key Theorem \ref{theo4.3} states an asymptotic upper bound on the principal exponential rate of decay of the probability of an excess of disconnected points corresponding to the event $\cA_N = \{|D_N \backslash \cC^u_{2N} \, | \ge \nu \, |D_N|\}$ from (\ref{0.7}). In the case of a small excess, i.e.~when $\nu$ is close to $\theta_0(u)$, with $u \in (0, \ov{u} \wedge \wh{u})$, we recover the value $\ov{J}_{u,\nu}$ from (\ref{0.9}) and the asymptotic upper bound that we derive matches the asymptotic lower bound from (6.32) of \cite{Szni19d}, see Corollary \ref{cor4.3}.

\medskip
In this section we assume that (see (\ref{1.26}))
\begin{align}
& 0 < u < \ov{u}, \label{4.1}
\intertext{and recall that}
&\mbox{$c_0 \in (0,1)$ is the dimension dependent constant from Theorem \ref{theo3.1}}. \label{4.2}
\end{align}
We write $\eta(\cdot)$ for the function (see (\ref{0.2}))
\begin{equation}\label{4.3}
\eta(b) = \theta_0(b^2), \; b \ge 0.
\end{equation}
We then consider an auxiliary function $\wh{\eta}$: $\IR_+ \r \IR_+$ such that
\begin{equation}\label{4.4}
\left\{ \begin{array}{rl}
{\rm i)} & \mbox{$\wh{\eta}$ is non-decreasing, continuous and bounded},
\\[1ex]
{\rm ii)} & \wh{\eta} \ge \eta ,
\\[1ex]
{\rm iii)} & \wh{b}\, \stackrel{\rm def}{=} \inf\{b \ge 0; \; \wh{\eta} (b) \ge 1\} < \sqrt{u} + c_0 (\sqrt{\ov{u}} - \sqrt{u}).
\end{array}\right.
\end{equation}

\begin{center}
\psfrag{whe}{$\wh{\eta}(\cdot)$}
\psfrag{eta}{$\eta(\cdot)$}
\psfrag{1}{$1$}
\psfrag{0}{$0$}
\psfrag{b}{$b$}
\psfrag{u*}{$\sqrt{u}_*$}
\psfrag{squ}{$\sqrt{u}$}
\psfrag{uu}{\scriptsize $\sqrt{u} + c_0 (\sqrt{\ov{u}} - \sqrt{u})$}
\includegraphics[width=11cm]{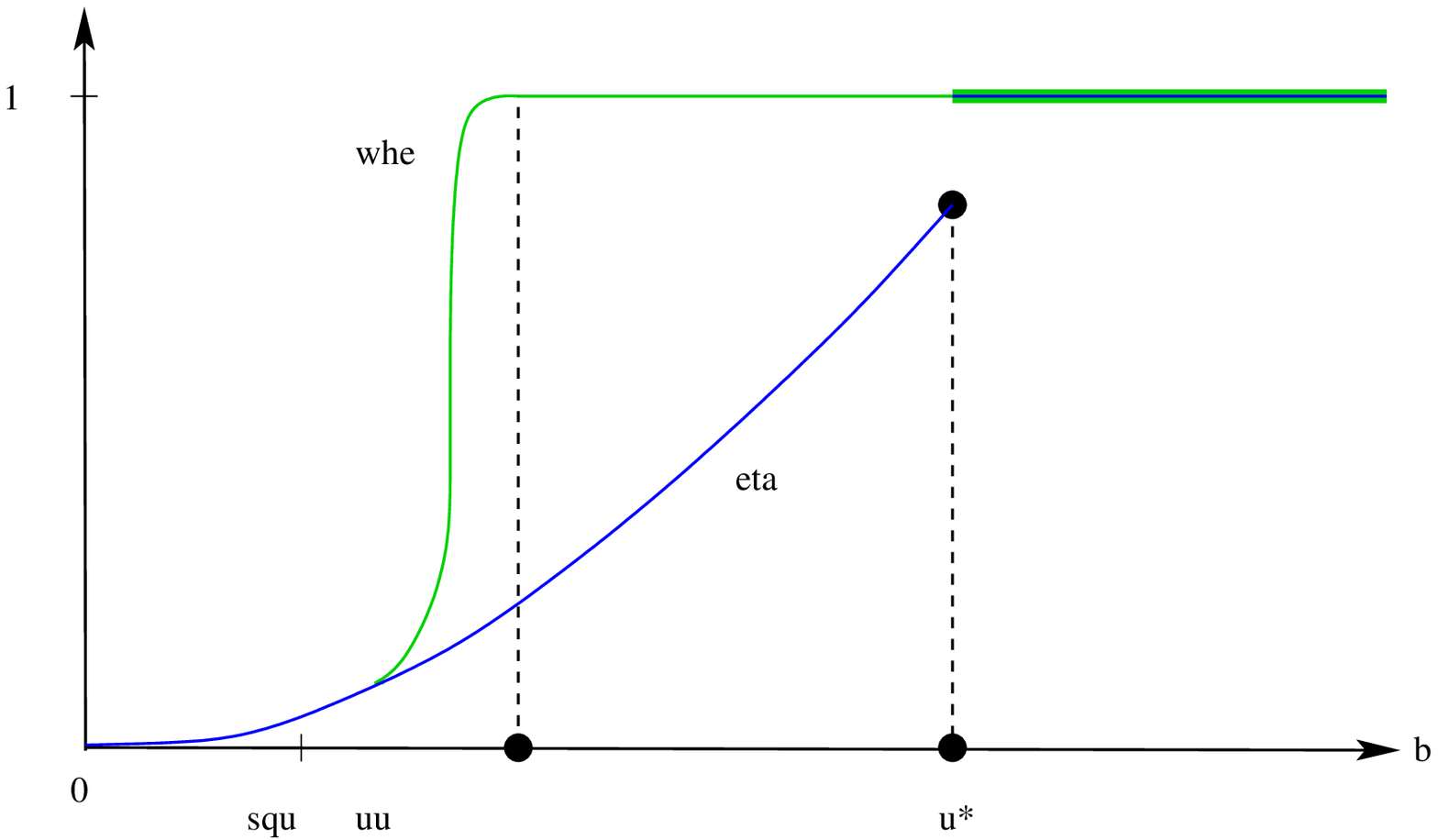}
\end{center}
\begin{center}
\begin{tabular}{ll}
Fig.~3 & An example of auxiliary function $\wh{\eta}$.
\end{tabular}
\end{center}

\n
The main step towards the key Theorem \ref{theo4.3} of this section is
\begin{proposition}\label{prop4.1}
Consider $u$ as in (\ref{4.1}), $\wh{\eta}$ as in (\ref{4.4}) and $\nu \in [\theta_0(u), 1)$. Then, one has
\begin{align}
& \limsup\limits_N \; \dis\frac{1}{N^{d-2}} \;\log \IP [\cA_N] \le - \wh{J}_{u,\nu}, \; \mbox{where} \label{4.5}
\\[1ex]
& \wh{J}_{u,\nu} = \inf\Big\{ \mbox{\f $\dis\frac{1}{2d}$} \; \dis\int_{\IR^d} | \nabla \varphi|^2 dz; \varphi \ge 0, \varphi \in D^1(\IR^d), \;\mbox{and} \; \dis\strokedint_D \wh{\eta} (\sqrt{u} + \varphi) \,dz \ge \nu\Big\} .\label{4.6}
\end{align}
\end{proposition}

\begin{remark}\label{rem4.2} \rm The above infimum is attained as can be shown by a similar compactness argument as used in the last paragraph of the proof of Corollary 5.9 of \cite{Szni19d}, i.e.~by extracting from a minimizing sequence $\varphi_n$ for (\ref{4.6}) a subsequence converging a.e.~and in $L^2_{\rm loc}(\IR^d)$ to a $\varphi \ge 0$ in $D^1(\IR^d)$ such that $\int_{\IR^d} |\nabla \varphi|^2 dz \le \liminf_n \int_{\IR^d} | \nabla \varphi_n |^2 dz = \wh{J}_{u,\nu}$. We thus have
\begin{equation}\label{4.7}
\wh{J}_{u,\nu} = \min \Big\{\mbox{\f $\dis\frac{1}{2d}$} \; \dis\int_{\IR^d} |\nabla \varphi |^2 dz; \varphi \ge 0, \varphi \in D^1(\IR^d), \; \mbox{and} \; \dis\strokedint_D \wh{\eta} (\sqrt{u} + \varphi)\,dz \ge \nu\Big\}.
\end{equation}

\vspace{-4ex}
\hfill $\square$
\end{remark}

\medskip
It may be useful at this stage to give a brief outline of the proof of Proposition \ref{prop4.1}. Thanks to Theorem \ref{theo2.1} we can in essence replace $\cA_N$ by $\cA'_N$. We use a coarse graining procedure to bound $\IP[\cA'_N]$. Compared to Section 5 of \cite{Szni19d} the main challenge here has to do with the presence in definition of the event $\cA'_N$ of the bubble set, with its non-local as well as irregular nature. To address this challenge we use the random set constructed in Theorem \ref{theo3.1}. Thanks to its coarse grained nature we essentially fix the set $C_\omega$, and keep track of discretized versions of the averages $\langle \ov{e}_{B_1}, L^u\rangle$ of occupation times in $B_1$-boxes away from $C_\omega$. A key point is to produce a formulation accounting for the constraints defining $\cA'_N$ that has a good behavior under scaling limit. For this purpose we consider certain discrete non-negative superharmonic functions $\wh{f}_\tau$ solving an obstacle problem on $\IZ^d$, see (\ref{4.26}) and (\ref{4.28}) for the constraint they satisfy. Once this proper formulation is achieved, the derivation of probabilistic bounds via exponential Chebyshev estimates expressed in terms of Dirichlet energies of these superharmonic functions, and the control of the related scaling limit behavior can be tackled along the same lines as in Section 5 of \cite{Szni19d}. The Proposition \ref{prop4.1} follows then.

\bigskip\n
{\it Proof of Proposition \ref{prop4.1}:} We consider $u$ as in (\ref{4.1}), $\wh{\eta}$ as in (\ref{4.4}) and $\nu \in (\theta_0,(u),1)$ (when $\nu = \theta_0(u)$, $\wh{J}_{u,\nu} = 0$, as seen by choosing $\varphi = 0$ in (\ref{4.6}), and (\ref{4.5}) is immediate). We then pick
\begin{equation}\label{4.8}
\mbox{$\alpha > \beta > \gamma$ in $(u,\ov{u})$ so that $\wh{b} < \sqrt{u} + c_0(\sqrt{\gamma} - \sqrt{u})$},
\end{equation}
where $\wh{b}$ is defined in (\ref{4.4}) iii). We then select $\ve \in (0,1)$ such that
\begin{equation}\label{4.9}
\nu > 10^3 \,\ve + \theta_0(u) \; \mbox{and} \; \wh{b} + \ve < \sqrt{u} + c_0 (\sqrt{\gamma} - \sqrt{u}),
\end{equation}
as well as the finite grid $\Sigma^0 (\gamma, u, \ve)$ from (\ref{2.6}). By (\ref{4.9}) and (\ref{2.6}), we see that (see (\ref{2.7}) for notation)
\begin{equation}\label{4.10}
\wh{b} \le \sqrt{u} + c_0 (\sqrt{\gamma}_- - \sqrt{u}) \quad  (< \sqrt{\gamma}_-).
\end{equation}
We then assume that (see Lemma \ref{lem1.2}, Lemma \ref{lem1.4} and Theorem \ref{theo2.1} for notation):
\begin{equation}\label{4.11}
K \ge c_1 \vee c_2 (\alpha, \beta, \gamma) \vee c_3 (\alpha, \beta, \gamma, u, \ve),
\end{equation}
so that Theorems \ref{theo2.1} and \ref{theo3.1} apply. In the notation of (\ref{2.11}) and (\ref{1.42}) we set
\begin{equation}\label{4.12}
\cA''_N = \cA'_N \backslash \cB_N,
\end{equation}
and find that as a consequence of (\ref{2.10}) of Theorem \ref{theo2.1} and (\ref{1.42}) of Lemma \ref{lem1.4}:
\begin{equation}\label{4.13}
\limsup\limits_N \; \dis\frac{1}{N^{d-2}} \;\log \IP[\cA_N] \le \limsup\limits_N \; \dis\frac{1}{N^{d-2}} \; \log \IP[\cA''_N].
\end{equation}
Recall the random set $C_\omega$ from Theorem \ref{theo3.1}. Then, for large $N$ on $\cA''_N$ by (\ref{2.11})
\begin{equation}\label{4.14}
(\nu - 6 \ve) \,|D_N| \le \dsl_{B_1 \subseteq D_N \backslash B u b} \wt{\theta}(\lambda^-_{B_1}) \,|B_1| + |B u b|,
\end{equation}
and by (\ref{3.5})~v), except maybe on a set of at most $\ve \, |D_N|$ points, $h_{C_\omega} \ge c_0$ on $B u b$, so that by (\ref{4.10}), $\sqrt{u} + (\sqrt{\gamma}_- - \sqrt{u}) \;h_{C_\omega} \ge \wh{b}$ on $B u b$ except on a set of at most $\ve \, |D_N|$ points. Taking into account (\ref{4.4}) and $\wh{b} \le \sqrt{\gamma}_-$, as well as the definition of $\wt{\theta}$ below (\ref{2.11}), we find that $1 \wedge \wh{\eta} \, (\sqrt{a}) \ge \wt{\theta}(a)$ for $a \ge 0$. We then see that for large $N$ on $\cA''_N$:
\begin{equation}\label{4.15}
\begin{split}
(\nu - 7 \ve) |D_N|& \le \dsl_{B_1 \subseteq D_N \backslash B u b} 1 \wedge \wh{\eta} \;\big(\sqrt{\lambda^-_{B_1}}\big) \; |B_1| 
\\
& \;\; + \dsl_{B_1 \subseteq B u b} \; \dsl_{x \in B_1} 1 \wedge \wh{\eta} \,\big(\sqrt{u} + (\sqrt{\gamma}_- - \sqrt{u}) \, h_{C_\omega}(x)\big)
\\[1ex]
& \le \dsl_{B_1 \in \cC_1} \; \dsl_{x \in B_1} 1 \wedge \wh{\eta} \,\big(\sqrt{\lambda^-_{B_1}} \vee  \big\{\sqrt{u} + (\sqrt{\gamma}_- - \sqrt{u}) \;h_{C_\omega}(x)\big\}\big)
\end{split}
\end{equation}
(with $\cC_1$ as in (\ref{2.9})).

\medskip
Given $C \in \cS_N$ (i.e. the set of possible values of the random set $C_\omega$ in Theorem \ref{theo3.1}) we define
\begin{align}
&\mbox{$\cC_C$ the collection of boxes $B_1$ in $\cC_1$ at $| \cdot |_\infty$-distance at least $(\ov{K} + 1) \, L_1$ from $C$}, \label{4.16}
\intertext{and for each $\tau \in \{0, \dots, \ov{K} -1\}^d$, in the notation of (\ref{2.64})}
& \cC_{C,\tau} = \cC_C \cap \cC_{1,\tau}. \label{4.17}
\end{align}
We will now use the grid $\Sigma^0$ to discretize the square root of the average values $\langle \ov{e}_{B_1}, L^u\rangle$ of the occupation time in boxes $B_1 \in \cC_C$. Specifically, for $C$ in $\cS_N$ and $\tau$ in $\{0,\dots, \ov{K} - 1\}^d$ we define
\begin{equation}\label{4.18}
\left\{ \begin{split}
\cF_C& = \; \big\{\ov{f}_C = (f_{B_1})_{B_1 \in \cC_C}; \, f_{B_1} \ge 0, \; f^2_{B_1} \in \{0\} \cup \Sigma^0 \; \mbox{for each $B_1 \in \cC_C\big\}$}
\\[1ex]
\cF_{C,\tau} & = \; \big\{\ov{f}_{C,\tau} = (f_{B_1})_{B_1 \in \cC_{C,\tau}}; \, f_{B_1} \ge 0, \; f^2_{B_1} \in \{0\} \cup \Sigma^0 \; \mbox{for each $B_1 \in \cC_{C,\tau}\big\}$}
\end{split}\right.
\end{equation}
as well as
\begin{equation}\label{4.19}
\left\{ \begin{split}
\cA_{\ov{f}_C} & = \bigcap\limits_{B_0 \subseteq C} \; \big\{ \langle \ov{e}_{D_0}, L^u\rangle \ge \gamma\big\} \cap \bigcap\limits_{B_1 \in \cC_C} \big\{\langle \ov{e}_{B_1}, L^u \rangle \ge (1- \ve) \, f^2_{B_1}\big\}, \; \mbox{for $\ov{f}_C \in \cF_C$},
\\[1ex]
\cA_{\ov{f}_{C,\tau}} & = \bigcap\limits_{B_0 \subseteq C} \; \big\{ \langle \ov{e}_{D_0}, L^u\rangle \ge \gamma\big\} \cap \bigcap\limits_{B_1 \in \cC_{C,\tau}} \big\{\langle \ov{e}_{B_1}, L^u \rangle \ge (1- \ve) \, f^2_{B_1}\big\}, \; \mbox{for $\ov{f}_{C,\tau} \in \cF_{C,\tau}$}.
\end{split}\right.
\end{equation}
As below (\ref{2.23}), we can consider $(\Sigma, \kappa, \mu)$-good boxes, where we now choose the local function $F = 0$. By Lemma 2.4 of \cite{Szni19d} one knows that
\begin{equation}\label{4.20}
\langle \ov{e}_{B_1}, L^u \rangle \ge (1 - \kappa) \, \lambda^-_{B_1} \stackrel{(\ref{2.23})}{\ge} (1 - \ve) \,\lambda^-_{B_1}, \;\mbox{when $B_1$ is a $(\Sigma, \kappa, \mu)$-good box}.
\end{equation}
Then by Proposition 3.1 and (4.9) of \cite{Szni19d}, if we assume that
\begin{align}
& K \ge c(\gamma, u, \ve), \label{4.21}
\intertext{it follows that}
& \lim\limits_N \; \dis\frac{1}{N^{d-2}} \; \log \IP [\wh{\cB}_N] = - \infty, \; \mbox{if $\wh{\cB}_N = \Big\{\dsl_{B_1 \in \cC_1} |B_1| \, 1\{B_1$ is $(\Sigma, \kappa, \mu)$-bad$\} \ge \ve \, |D_N|\Big\}$}.\label{4.22}
\end{align}
Thus, assuming (\ref{4.21}) in addition to (\ref{4.11}), we see that
\begin{equation}\label{4.23}
 \limsup\limits_N \; \dis\frac{1}{N^{d-2}} \; \log \IP [\cA_N] \le  \limsup\limits_N \; \dis\frac{1}{N^{d-2}} \; \log \IP [\wh{\cA}_N], \; \mbox{with $\wh{\cA}_N = \cA''_N \backslash \wh{\cB}_N$}.
\end{equation}
We now proceed with the coarse graining of the event $\wh{\cA}_N$. Choosing $f^2_{B_1} = \lambda^-_{B_1}$ when $B_1$ is a $(\Sigma, \kappa, \mu)$-good box and setting $f^2_{B_1} = 0$ when $B_1$ is a $(\Sigma, \kappa, \mu)$-bad box, we see that for large $N$ and any $C \in \cS_N$ (recalling that all boxes $B_0 \subseteq C_\omega$ are $(\alpha, \beta, \gamma)$-good and satisfy $N_u(D_0) \ge \beta \, {\rm cap}(D_0)$ so that by (\ref{1.38}) iii) $\langle \ov{e}_{D_0}, L^u \rangle \ge \gamma$), we have for large $N$ (with the notation (\ref{4.19}))
\begin{equation}\label{4.24}
\wh{\cA}_N \cap \{C_\omega = C\} \subseteq \bigcup\limits_{\ov{f}_C \in \wh{\cF}_C} \cA_{\ov{f}_C}, \; \; \mbox{for each $C \in \cS_N$},
\end{equation}
where using (\ref{3.5}) iv), the definition of $\wh{\cB}_N$ in (\ref{4.22}), and (\ref{4.15}), $\wh{\cF}_C$ denotes the sub-collection of $\cF_C$ of $\ov{f}_C = (f_{B_1})_{B_1 \in \cC_C}$ such that
\begin{equation}\label{4.25}
(\nu - 9 \ve) |D_N| \le \dsl_{B_1 \in \cC_C} \; \dsl_{x \in B_1} \wh{\eta} \,\big(f_{B_1} \vee \big\{ \sqrt{u} + (\sqrt{\gamma} - \sqrt{u}) \, h_C(x)\big\}\big).
\end{equation}
Now for any $\tau \in \{0,\dots, \ov{K} - 1\}^d$ and $\ov{f}_{C,\tau}$ in $\cF_{C,\tau}$ (see (\ref{4.18})) we define
\begin{equation}\label{4.26}
\begin{split}
\wh{f}_\tau  = &\; \mbox{the smallest non-negative superharmonic function on $\IZ^d$ such that}
\\
&\; \mbox{$\wh{f}_\tau \ge (f_{B_1} - \sqrt{u})_+$ on each $B_1 \in \cC_{C,\tau}$ and}
\\
& \; \mbox{$\wh{f}_\tau \ge \sqrt{\gamma} -  \sqrt{u}$ on each $D_0$ for $B_0 \subseteq C$.}
\end{split}
\end{equation}
In particular, note that $\wh{f}_\tau + \sqrt{u} \ge \sqrt{u} + (\sqrt{\gamma} - \sqrt{u}) \, h_C$. Then, for any $\ov{f}_C$ in the sub-collection $\wh{\cF}_C$ (see above (\ref{4.25})), we can consider the $\ov{f}_{C,\tau}$ with $\tau \in \{0,\dots,\ov{K} - 1\}^d$ obtained as restrictions of $\ov{f}_C$ to $\cC_{C,\tau}$, and the corresponding non-negative superharmonic functions $\wh{f}_\tau$, $\tau \in \{0,\dots, \ov{K}-1\}^d$. It now follows from the above remark that for all $\ov{f}_C$ in $\wh{\cF}_C$ one has
\begin{equation}\label{4.27}
\begin{array}{l}
(\nu - 9 \ve) \,|D_N| \le \dsl_{\tau \in \{0,\dots, \ov{K} - 1\}^d} \; \dsl_{B_1 \in \cC_{C,\tau}} \; \dsl_{x \in B_1} \wh{\eta} \,\big(\sqrt{u} + \wh{f}_\tau(x)\big),
\\
\\[-2ex]
\stackrel{(\ref{2.64})}{\le}  \dsl_{\tau \in \{0,\dots, \ov{K} - 1\}^d}  \; \dsl_{B_1 \in \cC_{1,\tau}}\; \dsl_{x \in B_1}  \wh{\eta} \,\big(\sqrt{u} + \wh{f}_\tau(x)\big)
\end{array}
\end{equation}
so that for some $\tau \in \{0,\dots, \ov{K} - 1\}^d$ one has
\begin{equation}\label{4.28}
\dis\frac{(\nu - 9 \ve)}{\ov{K}^d} \; |D_N| \le \dsl_{B_1 \in \cC_{1,\tau}} \; \dsl_{x \in B_1} \wh{\eta} \,\big(\sqrt{u} + \wh{f}_\tau (x)\big).
\end{equation}
We thus see that for large $N$,
\begin{equation}\label{4.29}
\wh{\cA}_N \cap \{C_\omega = C\} \subseteq \bigcup\limits_{\tau \in \{0, \dots, \ov{K} -1\}^d} \;\; \bigcup\limits_{\ov{f}_{C,\tau} \in \wh{\cF}_{C,\tau}} \cA_{\ov{f}_{C,\tau}}, \; \mbox{for each $C \in \cS_N$},
\end{equation}
where $\wh{\cF}_{C,\tau}$ denotes the sub-collection of $\cF_{C,\tau}$ in (\ref{4.18}) where (\ref{4.28}) holds. Note that $C$ varies in $\cS_N$ and by (\ref{3.5}) iii) one has $|\cS_N| = \exp\{o (N^{d-2})\}$, as $N \r \infty$. In addition, for each $C$ and $\tau$, we have $|\wh{\cF}_{C,\tau}| \le |\cF_{C,\tau}| \le (1 + |\Sigma^0|)^{|\cC_1|} \stackrel{(\ref{1.8})}{\le} ( 1 + | \Sigma^0| )^{c N^{d-2} / \log N} = \exp\{o(N^{d-2})\}$, as $N \r \infty$. Thus, by (\ref{4.23}) and (\ref{4.29}), we find that
\begin{equation}\label{4.30}
\limsup\limits_N \; \dis\frac{1}{N^{d-2}} \; \log \IP[\cA_N] \le \limsup\limits_N \;\; \sup\limits_C \;\; \sup\limits_\tau \; \;\sup\limits_{\wh{\cF}_{C,\tau}} \;\; \dis\frac{1}{N^{d-2}} \; \log \IP [\cA_{\ov{f}_{C,\tau}}],
\end{equation}
where $C$ varies in $\cS_N$ and $\tau$ in $\{0, \dots, \ov{K} - 1\}^d$ in the above supremum.

\medskip
The proofs in Proposition 5.4 of \cite{Szni19d} and the exponential Chebyshev bound in Proposition 5.6 of \cite{Szni19d} can be repeated in the present context (due to the fact that the $B_0$-boxes in $C$ or the corresponding $D_0$-boxes have mutual $| \cdot |_\infty$-distance at least $\ov{K}\,L_0$ by (\ref{3.5}) ii), and are at $| \cdot |_\infty$-distance at least $\ov{K}\,L_1$ from the $B_1$-boxes in $\cC_{C,\tau}$ by (\ref{4.16}), (\ref{4.17}), which themselves are at mutual $| \cdot |_\infty$-distance at least $\ov{K} \, L_1$). As a result, setting
\begin{equation}\label{4.31}
I_{\ve, K} = \liminf\limits_N \;\; \inf\limits_C \;\; \inf\limits_\tau \;\; \inf\limits_{\wh{\cF}_{C,\tau}} \;\;\dis\frac{1}{N^{d-2}} \; \cE(\wh{f}_\tau, \wh{f}_\tau), 
\end{equation}
where for $f$: $\IZ^d \r \IR$, $\cE(f,f) = \frac{1}{2} \,\sum_{|x-y| = 1} \; \frac{1}{2d} \;(f(y) - f(x))^2 ( \le \infty)$ stands for the discrete Dirichlet form, one obtains (see (5.49) of \cite{Szni19d}) that for $a \in (0,1)$ and $K \ge c_7 (\alpha, \beta, \gamma, u, \ve, a)$ 
\begin{equation}\label{4.32}
\limsup\limits_N \; \dis\frac{1}{N^{d-2}} \; \log \IP[\cA_N] \le -a \big(1 - \ve (1 + \sqrt{u})\big) \;I_{\ve, K} + c\, \sqrt{u} \, \ve.
\end{equation}
We then introduce for $b \ge 0$ and $r \ge 1$
\begin{equation}\label{4.33}
\begin{split}
J^\#_{b,r} = \inf\Big\{ \mbox{\f $\dis\frac{1}{2d}$} \; \dis\int_{\IR^d} |\nabla \varphi |^2 dz;  &\; \varphi \ge 0 \;\mbox{supported in} \; B_{\IR^d}(0,400 r), 
\\[-2ex]
&\; \varphi \in H^1(\IR^d), \strokedint_D \wh{\eta} (\sqrt{u} + \varphi) \, dz \ge b\},
\end{split}
\end{equation}
with $B_{\IR^d} (a,r)$ the closed ball with center $a$ in $\IR^d$ and radius $r$ for the supremum distance, and $H^1(\IR^d)$ the Sobolev space of measurable functions on $\IR^d$, which are square integrable together with their first partial derivatives (see Chapter 7 of \cite{LiebLoss01}). The same proof as in Proposition 5.7 of \cite{Szni19d} (actually simplified in the present context by the fact that the constraint (\ref{4.28}) is directly expressed in terms of $\wh{f}_\tau$ and the statement corresponding to (5.68) of \cite{Szni19d} easier to obtain) shows that
\begin{equation}\label{4.34}
\mbox{for $K \ge 100$, $\ve$ as in (\ref{4.9}), and integer $r \ge 10$, $\big(1 + \dis\frac{c_8}{r^{d-2}}\big) \;I_{\ve, K} \ge J^\#_{\nu - 9 \ve, r}$}.
\end{equation}
Inserting this lower bound in the right member of (\ref{4.32}) shows that for $\ve$ as in (\ref{4.9}), $r \ge 10$ integer, $a \in (0,1)$, one has
\begin{equation}\label{4.35}
\limsup\limits_N \;\; \dis\frac{1}{N^{d-2}} \; \log \IP [\cA_N] \le - a\big(1 + \dis\frac{c_8}{r^{d-2}}\big)^{-1} \big(1 - \ve(1 + \sqrt{u})\big) \, J^\#_{\nu - 9 \ve, r}+ c' \, \sqrt{u} \, \ve.
\end{equation}
Letting $\ve \r 0$ (see below (5.73) of \cite{Szni19d}), then letting $r \r \infty$, and then $a \r 1$, we obtain that
\begin{equation}\label{4.36}
\begin{split}
\limsup\limits_N \; \dis\frac{1}{N^{d-2}} \; \log \IP[\cA_N] & \le - \inf \Big\{\mbox{\f $\dis\frac{1}{2d}$} \; \dis\int_{\IR^d} |\nabla \phi|^2 dz; \varphi \ge 0, \varphi \in D^1(\IR^d) \; \mbox{and}
\\
& \qquad \quad \;\dis\strokedint_D \wh{\eta} \,(\sqrt{u} + \varphi) \, dz \ge \nu\Big\}
\\
&\!\!\!\! \stackrel{(\ref{4.6})}{=} - \wh{J}_{u,\nu}.
\end{split}
\end{equation}
This concludes the proof of (\ref{4.5}) and hence of Proposition \ref{prop4.1}. \hfill \mbox{\Large $\square$} 

\medskip
We now come to the main result of this section. With $u$ and $c_0$ and in (\ref{4.1}), (\ref{4.2}), we consider the functions (see Figure~1 in the Introduction for a sketch of $\theta^*$):
\begin{equation}\label{4.37a}
\left\{\begin{split}
\theta^*(a) & = \theta_0(a) \; 1\big\{a < \big(\sqrt{u} + c_0 (\sqrt{\ov{u}} - \sqrt{u})\big)^2\big\}  + 1\big\{a \ge \big(\sqrt{u} + c_0(\sqrt{\ov{u}} - \sqrt{u})\big)^2\big\}, \; a \ge 0,
\\
\eta^* (b) & = \theta^*(b^2) = \eta(b)\; 1\big\{b < \sqrt{u} + c_0 (\sqrt{\ov{u}} - \sqrt{u})\} + 1 \{ b \ge \sqrt{u} + c_0 (\sqrt{\ov{u}} - \sqrt{u})\big\}, b \ge 0.
\end{split}\right.
\end{equation}

\begin{theorem}\label{theo4.3}
Consider $u$ as in (\ref{4.1}) and $\nu \in [\theta_0(u), 1)$, then
\begin{align}
& \limsup\limits_N  \; \dis\frac{1}{N^{d-2}} \; \log \IP[\cA_N] \le - J^*_{u,\nu}, \;\mbox{where} \label{4.38a}
\\[1ex]
& J^*_{u,\nu} = \min \Big\{\mbox{\f $\dis\frac{1}{2d}$} \; \dis\int_{\IR^d} \; |\nabla \varphi |^2 \,dz; \varphi \ge 0, \varphi \in D^1(\IR^d), \;\mbox{and} \; \dis\strokedint_D \eta^* (\sqrt{u} + \varphi) \,dz \ge \nu\Big\}. \label{4.39a}
\end{align}
\end{theorem}

\begin{proof}
The existence of a minimizer for (\ref{4.38a}) is shown by the same argument as in the case of $\ov{J}_{u,\nu}$ in (\ref{0.9}), see Theorem 2 of \cite{Szni19c}. To prove (\ref{4.39a}) we will apply Proposition \ref{prop4.1} to a sequence of auxiliary functions $\wh{\eta}_n$, $n \ge 1$, satisfying (\ref{4.4}) and decreasing to $\eta^*$.

\medskip
More precisely, we consider two positive and increasing sequences $a_n < b_n$, $n \ge 1$, tending to $\sqrt{u} + c_0 (\sqrt{\ov{u}} - \sqrt{u})$, and denote by $\psi_n$ the continuous piecewise linear functions equal to $0$ on $[0,a_n]$, to $1$ on $[b_n, \infty)$, and linear on $[a_n, b_n]$. We set $\wh{\eta}_n = \max(\eta, \psi_n)$, $n \ge 1$, with $\eta$ as in (\ref{4.3}). We note that
\begin{equation}\label{4.40a}
\begin{array}{l}
\mbox{$\wh{\eta}_n$ satisfies (\ref{4.4}) for each $n \ge 1$, moreover the sequence $\wh{\eta}_n$, $n \ge 1$, is non-increasing}
\\
\mbox{and converges pointwise to $\eta^*$.}
\end{array}
\end{equation}
We denote by $\wh{J}^n_{u,\nu}$ the variational quantity (\ref{4.7}) corresponding to $\wh{\eta}_n$, so that by (\ref{4.40a})
\begin{equation}\label{4.41a}
\mbox{the sequence $\wh{J}^n_{u,\nu}$, $n \ge 1$ is non-decreasing and bounded by $J^*_{u,\nu}$.}
\end{equation}
The claim (\ref{4.38a}) and hence Theorem \ref{theo4.3} will follow from (\ref{4.5}) in Proposition \ref{prop4.1} once we show that
\begin{equation}\label{4.42a}
\lim\limits_n \; \wh{J}^n_{u,\nu} = J^*_{u,\nu}.
\end{equation}
To this end we consider for each $n \ge 1$ a minimizer $\varphi_n$ for $\wh{J}^n_{u,\nu}$. Then, by Theorem 8.6, p.~208 and Corollary 8.7, p.~212 of \cite{LiebLoss01}, we can extract a subsequence $\varphi_{n_\ell}$, $\ell \ge 1$, converging in $L^2_{{\rm loc}} (\IR^d)$ and a.e.~to $\varphi \ge 0$ belonging to $D^1(\IR^d)$ such that
\begin{equation}\label{4.43a}
\mbox{\f $\dis\frac{1}{2d}$} \; \dis\int_{\IR^d} |\nabla \varphi |^2 \, dz \le \liminf\limits_{\ell} \; \mbox{\f $\dis\frac{1}{2d}$} \; \dis\int_{\IR^d}  | \nabla \varphi_{n_\ell}|^2 \,dz \stackrel{(\ref{4.41a})}{=} \lim\limits_n \; \wh{J}^n_{u,\nu}.
\end{equation}
Moreover, we have $\eta^* ( \sqrt{u} + \varphi)  \ge \limsup\limits_\ell \; \wh{\eta}_{n_\ell} (\sqrt{u} + \varphi_{n_\ell})$ a.e., so that
\begin{equation}\label{4.44a}
\begin{split}
\dis\strokedint_D \eta^* (\sqrt{u} + \varphi) \,dz & \ge \dis\strokedint_D \limsup\limits_\ell \; \wh{\eta}_{n_\ell} (\sqrt{u} + \varphi_{n_\ell}) \, dz
\\
&\hspace{-4.5ex} \stackrel{\rm reverse \; Fatou}{\ge} \limsup\limits_\ell \; \dis\strokedint_D \wh{\eta}_{n_\ell} (\sqrt{u} + \varphi_{\wh{\eta}_\ell}) \, dz \ge \nu.
\end{split}
\end{equation}
This shows that $J^*_{u,\nu} \le \lim_n \;\wh{J}^n_{u,\nu}$ and with (\ref{4.41a}) we see that (\ref{4.42a}) holds (and in addition that $\varphi$ above is a minimizer for $J^*_{u,\nu}$ in (\ref{4.39a})). This concludes the proof of Theorem \ref{theo4.3}.
\end{proof}

\begin{remark}\label{rem4.4a} \rm
If $c_0$ in Theorem \ref{theo3.1} can be chosen arbitrarily close to $1$, then a similar argument as above shows that one can replace $c_0$ by $1$ in (\ref{4.37a}) and obtain the statements corresponding to (\ref{4.38a}), (\ref{4.39a}) with this replacement. If in addition the plausible (but presently open) equality $\ov{u} = u_*$ holds, this shows that (\ref{4.38a}) holds with $\ov{J}_{u,\nu}$ in place of $J^*_{u,\nu}$, and this asymptotic upper bound matches the asymptotic lower bound in (\ref{0.8}).  \hfill $\square$
\end{remark}

\subsection*{The case of a small excess}

We will now discuss an important consequence of Theorem \ref{theo4.3} in the case of a small excess $\nu$. In addition to (\ref{4.1}), we will assume that $u < \wh{u}$ (see (\ref{1.31}) for the definition of $\wh{u}$). As pointed out in Section 1, both $\ov{u}$ and $\wh{u}$ are positive and smaller or equal to $u_*$. It is plausible but open at the moment that $\ov{u} = \wh{u} = u_*$. We recall that $\cA^0_N \subseteq \cA_N$ stand for the respective excess events (see (\ref{0.7})) $\cA^0_N = \{|D_N \backslash \cC^u_N| \ge \nu \, |D_N|\}$ and $\cA_N = \{ |D_N \backslash \cC^u_{2N}| \ge \nu \, |D_N|\}$. By (6.32) of \cite{Szni19d} and Theorem 2 of \cite{Szni19c}, one knows that when $0 < u < u_*$ and $\theta_0(u) \le \nu < 1$,
\begin{equation}\label{4.37}
\liminf\limits_N \;\; \dis\frac{1}{N^{d-2}} \; \log \IP [\cA_N] \ge \liminf\limits_N \;  \dis\frac{1}{N^{d-2}} \;\; \log \IP [\cA_N^0] \ge - \ov{J}_{u,\nu},
\end{equation}
where
\begin{equation}\label{4.38}
\ov{J}_{u,\nu} = \min \Big\{ \mbox{\f $\dis\frac{1}{2d}$}  \; \dis\int_{\IR^d} |\nabla \varphi |^2 dz; \varphi \ge 0, \varphi \in D^1(\IR^d), \; \mbox{and} \; \dis\strokedint_D \, \ov{\theta}_0 \big((\sqrt{u} + \varphi)^2\big) \, dz \ge \nu\Big\}
\end{equation}
(with $\ov{\theta}_0(\cdot)$ the right-continuous modification of $\theta_0(\cdot)$).

\medskip
As we will now see, when $0 < u < \ov{u} \wedge \wh{u}$ and $\nu$ is close to $\theta_0(u)$, $\ov{J}_{u,\nu}$ governs the exponential rates of decay of $\IP [\cA^0_N]$ and $\IP[\cA_N]$.
\begin{corollary}\label{cor4.3}
Assume that $0 < u < \ov{u} \wedge \wh{u}$. Then there exists $\nu_0 \in (\theta_0 (u), 1)$ such that for any $\nu \in [\theta_0(u),\nu_0]$ one has
\begin{equation}\label{4.39}
\lim\limits_N \; \dis\frac{1}{N^{d-2}} \; \log \IP[\cA^0_N] = \lim\limits_N \; \dis\frac{1}{N^{d-2}} \; \log \IP [\cA_N] = - \ov{J}_{u,\nu}.
\end{equation}
\end{corollary}

\begin{proof}
In view of (\ref{4.37}) and the inclusion $\cA^0_N \subseteq \cA_N$ we only need to focus on the derivation of an asymptotic upper bound for $\IP[\cA_N]$. We first pick $u_0 \in (u,\ov{u} \wedge \wh{u})$ such that (with $c_0$ from Theorem \ref{theo3.1}):
\begin{equation}\label{4.40}
\sqrt{u}_0 < \sqrt{u} + c_0 (\sqrt{\ov{u}} - \sqrt{u}).
\end{equation}
Recall the notation $\eta(\cdot)$ from (\ref{4.3}) and $\eta^*$ from (\ref{4.37a}). Then, see (\ref{1.33}), $\theta_0$ is $C^1$ and $\theta'_0$ positive on a neighborhood of $[0,u_0]$. One can thus choose a  function $\wt{\eta}$ from $\IR_+$ into $\IR_+$ such that
\begin{align}
& \eta^* \le \wt{\eta}, \label{4.41}
\\[1ex]
& \left\{ \begin{array}{l}
\mbox{$\wt{\eta} = \eta$ on $[0, \sqrt{u}_0]$ and $\wt{\eta} > \eta$ on $[\sqrt{u}_0, + \infty)$}, \label{4.43}
\\[1ex]
\mbox{$\wt{\eta}$ is $C^1$ and $\wt{\eta}\,'$ bounded uniformly continuous on $\IR_+$,}
\\[1ex]
\mbox{$\wt{\eta}\,'$ is uniformly positive on each interval $[a, + \infty)$, $a > 0$}.
\end{array}\right.
\end{align}
Thus, $\wt{\eta}$ satisfies the conditions of Lemma 3 of \cite{Szni19c}. Then, as in Lemma 5 of \cite{Szni19c}, one can set for $\nu \ge \theta_0(u)$
\begin{equation}\label{4.44}
\wt{J}_{u,\nu} = \min \Big\{ \mbox{\f $\dis\frac{1}{2d}$} \; \dis\int_{\IR^d} |\nabla \varphi|^2 dz; \varphi \ge 0, \varphi \in D^1(\IR^d), \; \mbox{and} \; \dis\strokedint_D \wt{\eta} (\sqrt{u} + \varphi) \,dz \ge \nu\Big\}.
\end{equation}
Note that
\begin{equation}\label{4.45}
\ov{\theta}_0 (b^2) \le \eta^* (b) \le \wt{\eta} (b), \; \mbox{for $b \ge 0$ (and $\ov{J}_{u,\nu} \ge J^*_{u,\nu} \ge \wt{J}_{u,\nu}$ for $\theta_0 (u) \le \nu < 1$)}.
\end{equation}
One also knows, see above (98) of \cite{Szni19c}, that for a suitable $c_8(u, \wt{\eta}) > 0$, for all $\nu \in [\theta_0(u), \theta_0(u) + c_8]$, any minimizer $\wt{\varphi}$ for $\wt{J}_{u,\nu}$ in (\ref{4.44}) is bounded by $\sqrt{u}_0 - \sqrt{u}$ so that $\ov{\theta}_0((\sqrt{u} + \wt{\varphi})^2 ) = \eta^*(\sqrt{u} + \wt{\varphi}) = \wt{\eta}(\sqrt{u} + \wt{\varphi})$ and this minimizer is also a minimizer for $\ov{J}_{u,\nu}$ and $J^*_{u,\nu}$, so that
\begin{equation}\label{4.46}
\mbox{for $\nu \in [\theta_0(u), \theta_0(u) + c_8], \; \ov{J}_{u,\nu} = J^*_{u,\nu} = \wt{J}_{u,\nu}$}.
\end{equation}
The application of Theorem \ref{theo4.3} thus yields
\begin{equation}\label{4.47}
\limsup\limits_N \; \dis\frac{1}{N^{d-2}} \; \log \IP [\cA_N] \le - J^*_{u,\nu} = - \ov{J}_{u,\nu} \; \mbox{for all} \; \nu \in [\theta_0(u), \theta_0(u) + c_8].
\end{equation}
Combined with (\ref{4.37}), the claim (\ref{4.39}) now follows with $\nu_0 = \theta_0(u) + c_8$.
\end{proof}

\begin{remark}\label{rem4.4} \rm 1) In the small excess regime corresponding to $\theta_0(u) \le \nu \le \theta_0(u) + c_8(u,\wt{\eta})$ in (\ref{4.47}) above, one can actually show that all minimizers for $\ov{J}_{u,\nu}$ are minimizers for $\wt{J}_{u,\nu}$, see below (99) of \cite{Szni19c}. These minimizers are $C^{1,\alpha}$-regular for all $0 < \alpha < 1$ and their supremum norm is at most $\sqrt{u}_0 - \sqrt{u}$ $(< \sqrt{u}_* - \sqrt{u})$. We refer to Theorem 3 of \cite{Szni19c} for more properties of the minimizers of $\ov{J}_{u,\nu}$ in the small excess regime.

\bigskip\n
2) One can naturally wonder whether (\ref{4.39}) extends beyond the small excess regime and whether for all $0 < u < \ov{u} \wedge \wh{u}$, 
\begin{equation}\label{4.48}
\lim\limits_N \; \dis\frac{1}{N^{d-2}} \; \log \IP [\cA_N] = - \ov{J}_{u,\nu}  \; \mbox{for all} \; \nu \in [\theta_0(u),1)?
\end{equation}
(This asymptotics then also holds for $\cA^0_N$ due to (\ref{4.37}) and the inclusion $\cA^0_N \subseteq \cA_N$.)

\medskip
We also refer to Remark \ref{rem4.4a} on the related issue of being able to choose $c_0$ arbitrarily close to $1$ in Theorem \ref{theo3.1}.

\bigskip\n
3) Letting $\cC^u_\infty$ stand for the infinite cluster of $\cV^u$, when $u < u_*$, one can also wonder whether a similar asymptotics holds for an excess of points in $D_N$ outside the infinite cluster. Does one have
\begin{equation}\label{4.49}
\lim\limits_N \; \dis\frac{1}{N^{d-2}} \; \log \IP [|D_N \backslash \cC^u_\infty | \ge \nu \, |D_N|] = \ov{J}_{u,\nu}, \; \mbox{for $0 < u < u_*$ and $\theta_0(u) \le \nu < 1$ ?}
\end{equation}
(The lower bound corresponding to (\ref{4.49}) holds by (\ref{4.37}) and the inclusions $D_N \backslash \cC^u_N \subseteq D_N \backslash \cC^u_{2N} \subseteq D_N \backslash \cC^u_\infty$. And from a positive answer to (\ref{4.49}) the statement (\ref{4.48}) would follow as well.)  

\medskip
We refer to Theorem 2.12 on p.~21 of \cite{Cerf00} for a result concerning a similar question in the context of the Wulff droplet and Bernoulli percolation.

\bigskip\n
4) As mentioned in the Introduction it is open whether for large enough $\nu$ the minimizers $\varphi$ for $\ov{J}_{u,\nu}$ in (\ref{4.38}) reach the value $\sqrt{u}_* - \sqrt{u}$ on a set of positive Lebesgue measure. If the function $\theta_0$ is discontinuous at $u_*$ (a not very plausible assumption) this is indeed the case, see Remark 2 1) of \cite{Szni19c}. Having a better grasp of the of the behaviour of $\theta_0$ near $u_*$ would likely help making progress on this issue. We refer to Figures 4 and 2 of \cite{MariLebo06} for the result of simulations in the (closely) related model of the level-set percolation of the Gaussian free field when $d=3$.
\hfill $\square$
\end{remark}

\end{document}